\documentclass[11pt]{article}
\usepackage[utf8]{inputenc}
\usepackage[utf8]{inputenc}
\usepackage{hyperref}
\usepackage{pst-node}
\usepackage{enumitem}
\usepackage{auto-pst-pdf}
\usepackage{tikz-cd} 
\usepackage{comment}
\usepackage{amsmath,amssymb,amsthm}
\usepackage{thmtools,thm-restate}
\usepackage{amsfonts,tikz}
\usepackage[margin=1in]{geometry}
\usepackage{mathtools}
\usepackage{caption}
\usetikzlibrary{positioning,automata,shapes}
\tikzset{every loop/.style={min distance=10 mm, in=60, out=120, looseness=10}}
\DeclareMathOperator{\wts}{wts}
\DeclareMathOperator{\sink}{sink}
\DeclareMathOperator{\Sink}{Sink}
\DeclareMathOperator{\sgn}{sgn}
\newcommand{\lrp}[1]{\left(#1\right)}
\newcommand{\bp}[1]{\big(#1\big)}

\newcommand{\ts}[1]{{\textstyle #1}}

\newcommand{\N}{\mathbb{N}}

\newcommand{\bk}{\backslash}
\newcommand{\lm}{\lambda}

\newcommand{\Lm}{\Lambda}
\newcommand{\tm}{\widetilde{m}}

\renewcommand{\epsilon}{\varepsilon}

\newtheorem{lemma}{Lemma}
\newtheorem{cor}[lemma]{Corollary}
\newtheorem{theorem}[lemma]{Theorem}
\newtheorem{conj}[lemma]{Conjecture}
\newtheorem{definition}[lemma]{Definition}

\title{$e$-basis Coefficients of Chromatic Symmetric Functions}
\author{Logan Crew, Yongxing Zhang\footnote{Department of Combinatorics \& Optimization, University of Waterloo, Waterloo, ON, N2L 3G1.\newline  Emails:  lcrew@uwaterloo.ca, y3499zha@uwaterloo.ca. \newline
We acknowledge the support of the Natural Sciences and Engineering Research Council of Canada (NSERC), [funding
reference number RGPIN-2022-03093]. \newline Cette recherche a \'et\'e financ\'ee par le Conseil de recherches en sciences naturelles et en g\'enie du Canada (CRSNG),
[num\'ero de r\'ef\'erence RGPIN-2022-03093].}}
\date{\today}

\begin{document}

\maketitle

\begin{abstract}

A well-known result of Stanley's shows that given a graph $G$ with chromatic symmetric function expanded into the basis of elementary symmetric functions as $X_G = \sum c_{\lambda}e_{\lambda}$, the sum of the coefficients $c_{\lambda}$ for $\lambda$ with $\lm_1' = k$ (equivalently those $\lm$ with exactly $k$ parts) is equal to the number of acyclic orientations of $G$ with exactly $k$ sinks. 

However, more is known. The \emph{sink sequence} of an acyclic orientation of $G$ is a tuple $(s_1,\dots,s_k)$ such that $s_1$ is the number of sinks of the orientation, and recursively each $s_i$ with $i > 1$ is the number of sinks remaining after deleting the sinks contributing to $s_1,\dots,s_{i-1}$. Equivalently, the sink sequence gives the number of vertices at each level of the poset induced by the acyclic orientation.

A lesser-known follow-up result of Stanley's determines certain cases in which we can find a sum of $e$-basis coefficients that gives the number of acyclic orientations of $G$ with a given partial sink sequence. Of interest in its own right, this result also admits as a corollary a simple proof of the $e$-positivity of $X_G$ when the stability number of $G$ is $2$.

In this paper, we prove a vertex-weighted generalization of this follow-up result, and conjecture a stronger version that admits a similar combinatorial interpretation for a much larger set of $e$-coefficient sums of chromatic symmetric functions. In particular, the conjectured formula would give a combinatorial interpretation for the sum of the coefficients $c_{\lm}$ with prescribed values of $\lm_1'$ and $\lm_2'$ for any unweighted claw-free graph (not necessarily an incomparability graph, as in the setting of the Stanley-Stembridge conjecture).

\end{abstract}

\section{Introduction}

Given a graph $G$, its chromatic symmetric function $X_G$ is defined as
$$
X_G = \sum_{\kappa} \prod_{v \in V(G)} x_{\kappa(v)}
$$
where the sum ranges over all $\kappa: V(G) \rightarrow \mathbb{N}$ that are proper colorings of $G$. Since its introduction in the 1990s by Stanley \cite{stanley}, research has connected the chromatic symmetric function and its generalizations to objects in algebraic and geometric combinatorics, including the geometry of Hessenberg varieties \cite{unit, cho2022, wachs} and LLT polynomials \cite{abreu, per2022, tom}.

Of particular interest is the $e$-basis expansion of the chromatic symmetric functions, in large part due to the Stanley-Stembridge conjecture \cite{stanley}, which after a reduction by Guay-Paquet \cite{guay} claims that the chromatic symmetric function of any unit interval graph is $e$-positive. This conjecture was recently proved by Hikita \cite{hikita} by giving an explicit probabilistic interpretation to the $e$-basis coefficients of such graphs; the more general Shareshian-Wachs conjecture \cite{wachs} about $e$-basis coefficients of chromatic quasisymmetric functions of unit interval graphs currently remains open, as well as the stronger conjecture that (claw,net)-free graphs are $e$-positive \cite{foley}.

Explicit formulas for specific $e$-basis coefficients of the chromatic symmetric function in general graphs are either only known for certain special coefficients \cite{kali} or in a form that does not directly give their sign \cite{chownote}. The $e$-positivity, via direct coefficient computation or otherwise, of some graph classes is known, such as the aforementioned unit interval graphs, some cyclic analogues of unit interval graphs \cite{banaian, tom2025}, and certain classes defined by forbidden induced subgraphs, particularly subsets of (claw,net)-free graphs \cite{foley, hamel2019}.

However, a seminal result of Stanley's original paper shows that certain \emph{sums} of $e$-basis coefficients are positive for any graph. In particular, he demonstrates
\begin{theorem}[Theorem 3.3, \cite{stanley}]\label{thm:stanley3.3}

If $X_G = \sum c_{\lm}e_{\lm}$, then the number of acyclic orientations of $G$ with exactly $k$ sinks (vertices with no incident outgoing edges, including isolated vertices) is equal to 
$$
\sum_{\ell(\lm) = k} c_{\lm}.
$$

\end{theorem}
Indeed, since Stanley's original proof relied on passing to quasisymmetric generating functions of posets, recent results proving the same result in a more combinatorial manner have been noteworthy \cite{delcon, hwang}.

However, Stanley provided another theorem even more general than this, though rarely cited in the literature. The \emph{sink sequence} of an acyclic orientation $o$ is the tuple $ss(o) = (s_1, s_2, \dots)$, where $s_1$ is the number of sinks of $o$, $s_2$ is the number of sinks of $o$ after deleting the $s_1$ original sinks, and so on. Furthermore, call a partition $\nu$ of $s \leq |V(G)|$ \emph{allowable} if there exist disjoint stable sets $S_1, \dots, S_{\ell(\nu)}$ of $V(G)$ such that $|S_i| = \nu_i$. Then Stanley showed
\begin{theorem}[Theorem 3.4, \cite{stanley}]\label{thm:stanley3.4}
Let $\mu$ be a partition of $r \leq |V(G)|$ with $k$ parts. Suppose that for every allowable $\nu$, either
\begin{itemize}
    \item $\nu$ does not dominate $\mu$, or
    \item $\nu_i = \mu_i$ for $i \in \{1,2,\dots,k\}$.\footnote{For clarity, we note here that in later sections, when $\mu$ and $\nu$ satisfy one of the above two conditions, we say $\mu$ \textbf{partially dominates} $\nu$, and $\mu$ is a \textbf{maximal} partition of $G$ when it partially dominates all allowable partitions. See Definitions \ref{def:partialdom} and \ref{def:maxallow}.}
\end{itemize}

Let $X_G = \sum c_{\lm}e_{\lm}$. Then for any $j \in \{0, 1, \dots, |V(G)|-r\}$ (with $j = 0$ possible only if $r = |V(G)|$) the number of acyclic orientations of $G$ with sink sequence of the form $(\mu_1, \dots, \mu_k, j, \dots)$ is equal to 
$$
\sum_{\lm} c_{\lm}
$$
where the sum ranges over all $\lm$ such that $\lm_i' = \mu_i$ for $i \in \{1,\dots,k\}$ and $\lm_{k+1}' = j$ (here $\lm'$ is the transpose of $\lm$).\footnote{The originally published version of this theorem is not quite correct as stated because it did not mention the second bullet point above. The version given here was described in an erratum \cite{erratum}.}

\end{theorem}

This theorem effectively generalizes the previous one to more restricted sums of $e$-basis coefficients (Theorem \ref{thm:stanley3.3} could be viewed as the ``$\mu = \varnothing$ case" of Theorem \ref{thm:stanley3.4}), but only for sums dictated by those $\mu$ which satisfy certain properties depending on $G$.

We provide a concrete example for Theorem \ref{thm:stanley3.4} to facilitate understanding.
Consider the graph $G$ given in Figure \ref{fig:thm2-example}.
The allowable partitions of $G$ are $(2)$, $(2, 1)$, $(2, 1, 1)$, $(1)$, $(1, 1)$, $(1, 1, 1)$, and $(1, 1, 1, 1)$. Consider $\mu=(2, 1)$ and $j=1$. It is easy to verify that $\mu$ satisfies the two bullet point conditions in Theorem \ref{thm:stanley3.4}.
In Figure \ref{fig:thm2-example-acyclic-orientations}, we list all acyclic orientations of $G$. We note that all acyclic orientations of $G$ have sink sequence $(2, 1, 1)$. (Note that the isolated vertex of $G$ is always a sink in the first term of the sink sequence, as it never has outgoing edges.)


\begin{figure}[hbt]
\begin{center}

  \tikzset{every picture/.style={line width=0.75pt}} 

  \begin{tikzpicture}[x=0.75pt,y=0.75pt,yscale=-1,xscale=1]
  
  \draw  [fill={rgb, 255:red, 0; green, 0; blue, 0 }  ,fill opacity=1 ] (2.19,5.25) .. controls (2.19,3.8) and (3.35,2.63) .. (4.77,2.63) .. controls (6.2,2.63) and (7.35,3.8) .. (7.35,5.25) .. controls (7.35,6.7) and (6.2,7.88) .. (4.77,7.88) .. controls (3.35,7.88) and (2.19,6.7) .. (2.19,5.25) -- cycle ;
  \draw  [fill={rgb, 255:red, 0; green, 0; blue, 0 }  ,fill opacity=1 ] (52.19,5.25) .. controls (52.19,3.8) and (53.35,2.63) .. (54.77,2.63) .. controls (56.2,2.63) and (57.35,3.8) .. (57.35,5.25) .. controls (57.35,6.7) and (56.2,7.88) .. (54.77,7.88) .. controls (53.35,7.88) and (52.19,6.7) .. (52.19,5.25) -- cycle ;
  \draw  [fill={rgb, 255:red, 0; green, 0; blue, 0 }  ,fill opacity=1 ] (2.19,55.25) .. controls (2.19,53.8) and (3.35,52.63) .. (4.77,52.63) .. controls (6.2,52.63) and (7.35,53.8) .. (7.35,55.25) .. controls (7.35,56.7) and (6.2,57.88) .. (4.77,57.88) .. controls (3.35,57.88) and (2.19,56.7) .. (2.19,55.25) -- cycle ;
  \draw  [fill={rgb, 255:red, 0; green, 0; blue, 0 }  ,fill opacity=1 ] (52.19,55.25) .. controls (52.19,53.8) and (53.35,52.63) .. (54.77,52.63) .. controls (56.2,52.63) and (57.35,53.8) .. (57.35,55.25) .. controls (57.35,56.7) and (56.2,57.88) .. (54.77,57.88) .. controls (53.35,57.88) and (52.19,56.7) .. (52.19,55.25) -- cycle ;
  \draw    (7.35,5.25) -- (52.19,5.25) ;
  \draw    (4.77,7.88) -- (4.77,52.63) ;
  \draw    (5.77,54.25) -- (54.77,5.25) ;
  \end{tikzpicture}
  
\end{center}
\captionsetup{skip=2pt}
\caption{$G$, including an isolated vertex}
\label{fig:thm2-example}
\end{figure}

\begin{figure}[hbt]
\begin{center}

  \tikzset{every picture/.style={line width=0.75pt}} 

  \begin{tikzpicture}[x=0.75pt,y=0.75pt,yscale=-1,xscale=1]
  
  \draw  [fill={rgb, 255:red, 248; green, 231; blue, 28 }  ,fill opacity=1 ] (2.19,5.25) .. controls (2.19,3.8) and (3.35,2.63) .. (4.77,2.63) .. controls (6.2,2.63) and (7.35,3.8) .. (7.35,5.25) .. controls (7.35,6.7) and (6.2,7.88) .. (4.77,7.88) .. controls (3.35,7.88) and (2.19,6.7) .. (2.19,5.25) -- cycle ;
  \draw  [fill={rgb, 255:red, 0; green, 0; blue, 0 }  ,fill opacity=1 ] (52.19,5.25) .. controls (52.19,3.8) and (53.35,2.63) .. (54.77,2.63) .. controls (56.2,2.63) and (57.35,3.8) .. (57.35,5.25) .. controls (57.35,6.7) and (56.2,7.88) .. (54.77,7.88) .. controls (53.35,7.88) and (52.19,6.7) .. (52.19,5.25) -- cycle ;
  \draw  [fill={rgb, 255:red, 0; green, 0; blue, 0 }  ,fill opacity=1 ] (2.19,55.25) .. controls (2.19,53.8) and (3.35,52.63) .. (4.77,52.63) .. controls (6.2,52.63) and (7.35,53.8) .. (7.35,55.25) .. controls (7.35,56.7) and (6.2,57.88) .. (4.77,57.88) .. controls (3.35,57.88) and (2.19,56.7) .. (2.19,55.25) -- cycle ;
  \draw  [fill={rgb, 255:red, 248; green, 231; blue, 28 }  ,fill opacity=1 ] (52.19,55.25) .. controls (52.19,53.8) and (53.35,52.63) .. (54.77,52.63) .. controls (56.2,52.63) and (57.35,53.8) .. (57.35,55.25) .. controls (57.35,56.7) and (56.2,57.88) .. (54.77,57.88) .. controls (53.35,57.88) and (52.19,56.7) .. (52.19,55.25) -- cycle ;
  \draw    (9.35,5.25) -- (52.19,5.25) ;
  \draw [shift={(7.35,5.25)}, rotate = 0] [color={rgb, 255:red, 0; green, 0; blue, 0 }  ][line width=0.75]    (10.93,-3.29) .. controls (6.95,-1.4) and (3.31,-0.3) .. (0,0) .. controls (3.31,0.3) and (6.95,1.4) .. (10.93,3.29)   ;
  \draw    (4.77,9.88) -- (4.77,52.63) ;
  \draw [shift={(4.77,7.88)}, rotate = 90] [color={rgb, 255:red, 0; green, 0; blue, 0 }  ][line width=0.75]    (10.93,-3.29) .. controls (6.95,-1.4) and (3.31,-0.3) .. (0,0) .. controls (3.31,0.3) and (6.95,1.4) .. (10.93,3.29)   ;
  \draw    (5.77,54.25) -- (53.36,6.67) ;
  \draw [shift={(54.77,5.25)}, rotate = 135] [color={rgb, 255:red, 0; green, 0; blue, 0 }  ][line width=0.75]    (10.93,-3.29) .. controls (6.95,-1.4) and (3.31,-0.3) .. (0,0) .. controls (3.31,0.3) and (6.95,1.4) .. (10.93,3.29)   ;
  \draw  [fill={rgb, 255:red, 248; green, 231; blue, 28 }  ,fill opacity=1 ] (131.19,4.74) .. controls (131.19,3.29) and (132.35,2.12) .. (133.77,2.12) .. controls (135.2,2.12) and (136.35,3.29) .. (136.35,4.74) .. controls (136.35,6.19) and (135.2,7.37) .. (133.77,7.37) .. controls (132.35,7.37) and (131.19,6.19) .. (131.19,4.74) -- cycle ;
  \draw  [fill={rgb, 255:red, 0; green, 0; blue, 0 }  ,fill opacity=1 ] (181.19,4.74) .. controls (181.19,3.29) and (182.35,2.12) .. (183.77,2.12) .. controls (185.2,2.12) and (186.35,3.29) .. (186.35,4.74) .. controls (186.35,6.19) and (185.2,7.37) .. (183.77,7.37) .. controls (182.35,7.37) and (181.19,6.19) .. (181.19,4.74) -- cycle ;
  \draw  [fill={rgb, 255:red, 0; green, 0; blue, 0 }  ,fill opacity=1 ] (131.19,54.74) .. controls (131.19,53.29) and (132.35,52.12) .. (133.77,52.12) .. controls (135.2,52.12) and (136.35,53.29) .. (136.35,54.74) .. controls (136.35,56.19) and (135.2,57.37) .. (133.77,57.37) .. controls (132.35,57.37) and (131.19,56.19) .. (131.19,54.74) -- cycle ;
  \draw  [fill={rgb, 255:red, 248; green, 231; blue, 28 }  ,fill opacity=1 ] (181.19,54.74) .. controls (181.19,53.29) and (182.35,52.12) .. (183.77,52.12) .. controls (185.2,52.12) and (186.35,53.29) .. (186.35,54.74) .. controls (186.35,56.19) and (185.2,57.37) .. (183.77,57.37) .. controls (182.35,57.37) and (181.19,56.19) .. (181.19,54.74) -- cycle ;
  \draw    (138.35,4.74) -- (181.19,4.74) ;
  \draw [shift={(136.35,4.74)}, rotate = 0] [color={rgb, 255:red, 0; green, 0; blue, 0 }  ][line width=0.75]    (10.93,-3.29) .. controls (6.95,-1.4) and (3.31,-0.3) .. (0,0) .. controls (3.31,0.3) and (6.95,1.4) .. (10.93,3.29)   ;
  \draw    (133.77,9.37) -- (133.77,52.12) ;
  \draw [shift={(133.77,7.37)}, rotate = 90] [color={rgb, 255:red, 0; green, 0; blue, 0 }  ][line width=0.75]    (10.93,-3.29) .. controls (6.95,-1.4) and (3.31,-0.3) .. (0,0) .. controls (3.31,0.3) and (6.95,1.4) .. (10.93,3.29)   ;
  \draw    (135.91,52.58) -- (183.5,4.99) ;
  \draw [shift={(134.5,53.99)}, rotate = 315] [color={rgb, 255:red, 0; green, 0; blue, 0 }  ][line width=0.75]    (10.93,-3.29) .. controls (6.95,-1.4) and (3.31,-0.3) .. (0,0) .. controls (3.31,0.3) and (6.95,1.4) .. (10.93,3.29)   ;
  \draw  [fill={rgb, 255:red, 0; green, 0; blue, 0 }  ,fill opacity=1 ] (2.19,94.5) .. controls (2.19,93.05) and (3.35,91.87) .. (4.77,91.87) .. controls (6.2,91.87) and (7.35,93.05) .. (7.35,94.5) .. controls (7.35,95.95) and (6.2,97.12) .. (4.77,97.12) .. controls (3.35,97.12) and (2.19,95.95) .. (2.19,94.5) -- cycle ;
  \draw  [fill={rgb, 255:red, 248; green, 231; blue, 28 }  ,fill opacity=1 ] (52.19,94.5) .. controls (52.19,93.05) and (53.35,91.87) .. (54.77,91.87) .. controls (56.2,91.87) and (57.35,93.05) .. (57.35,94.5) .. controls (57.35,95.95) and (56.2,97.12) .. (54.77,97.12) .. controls (53.35,97.12) and (52.19,95.95) .. (52.19,94.5) -- cycle ;
  \draw  [fill={rgb, 255:red, 0; green, 0; blue, 0 }  ,fill opacity=1 ] (2.19,144.5) .. controls (2.19,143.05) and (3.35,141.87) .. (4.77,141.87) .. controls (6.2,141.87) and (7.35,143.05) .. (7.35,144.5) .. controls (7.35,145.95) and (6.2,147.12) .. (4.77,147.12) .. controls (3.35,147.12) and (2.19,145.95) .. (2.19,144.5) -- cycle ;
  \draw  [fill={rgb, 255:red, 248; green, 231; blue, 28 }  ,fill opacity=1 ] (52.19,144.5) .. controls (52.19,143.05) and (53.35,141.87) .. (54.77,141.87) .. controls (56.2,141.87) and (57.35,143.05) .. (57.35,144.5) .. controls (57.35,145.95) and (56.2,147.12) .. (54.77,147.12) .. controls (53.35,147.12) and (52.19,145.95) .. (52.19,144.5) -- cycle ;
  \draw    (7.35,94.5) -- (50.19,94.5) ;
  \draw [shift={(52.19,94.5)}, rotate = 180] [color={rgb, 255:red, 0; green, 0; blue, 0 }  ][line width=0.75]    (10.93,-3.29) .. controls (6.95,-1.4) and (3.31,-0.3) .. (0,0) .. controls (3.31,0.3) and (6.95,1.4) .. (10.93,3.29)   ;
  \draw    (4.77,99.12) -- (4.77,141.87) ;
  \draw [shift={(4.77,97.12)}, rotate = 90] [color={rgb, 255:red, 0; green, 0; blue, 0 }  ][line width=0.75]    (10.93,-3.29) .. controls (6.95,-1.4) and (3.31,-0.3) .. (0,0) .. controls (3.31,0.3) and (6.95,1.4) .. (10.93,3.29)   ;
  \draw    (5.5,143.75) -- (53.09,96.16) ;
  \draw [shift={(54.5,94.75)}, rotate = 135] [color={rgb, 255:red, 0; green, 0; blue, 0 }  ][line width=0.75]    (10.93,-3.29) .. controls (6.95,-1.4) and (3.31,-0.3) .. (0,0) .. controls (3.31,0.3) and (6.95,1.4) .. (10.93,3.29)   ;
  \draw  [fill={rgb, 255:red, 0; green, 0; blue, 0 }  ,fill opacity=1 ] (262.19,4.74) .. controls (262.19,3.29) and (263.35,2.12) .. (264.77,2.12) .. controls (266.2,2.12) and (267.35,3.29) .. (267.35,4.74) .. controls (267.35,6.19) and (266.2,7.37) .. (264.77,7.37) .. controls (263.35,7.37) and (262.19,6.19) .. (262.19,4.74) -- cycle ;
  \draw  [fill={rgb, 255:red, 0; green, 0; blue, 0 }  ,fill opacity=1 ] (312.19,4.74) .. controls (312.19,3.29) and (313.35,2.12) .. (314.77,2.12) .. controls (316.2,2.12) and (317.35,3.29) .. (317.35,4.74) .. controls (317.35,6.19) and (316.2,7.37) .. (314.77,7.37) .. controls (313.35,7.37) and (312.19,6.19) .. (312.19,4.74) -- cycle ;
  \draw  [fill={rgb, 255:red, 248; green, 231; blue, 28 }  ,fill opacity=1 ] (262.19,54.74) .. controls (262.19,53.29) and (263.35,52.12) .. (264.77,52.12) .. controls (266.2,52.12) and (267.35,53.29) .. (267.35,54.74) .. controls (267.35,56.19) and (266.2,57.37) .. (264.77,57.37) .. controls (263.35,57.37) and (262.19,56.19) .. (262.19,54.74) -- cycle ;
  \draw  [fill={rgb, 255:red, 248; green, 231; blue, 28 }  ,fill opacity=1 ] (312.19,54.74) .. controls (312.19,53.29) and (313.35,52.12) .. (314.77,52.12) .. controls (316.2,52.12) and (317.35,53.29) .. (317.35,54.74) .. controls (317.35,56.19) and (316.2,57.37) .. (314.77,57.37) .. controls (313.35,57.37) and (312.19,56.19) .. (312.19,54.74) -- cycle ;
  \draw    (269.35,4.74) -- (312.19,4.74) ;
  \draw [shift={(267.35,4.74)}, rotate = 0] [color={rgb, 255:red, 0; green, 0; blue, 0 }  ][line width=0.75]    (10.93,-3.29) .. controls (6.95,-1.4) and (3.31,-0.3) .. (0,0) .. controls (3.31,0.3) and (6.95,1.4) .. (10.93,3.29)   ;
  \draw    (264.77,7.37) -- (264.77,50.12) ;
  \draw [shift={(264.77,52.12)}, rotate = 270] [color={rgb, 255:red, 0; green, 0; blue, 0 }  ][line width=0.75]    (10.93,-3.29) .. controls (6.95,-1.4) and (3.31,-0.3) .. (0,0) .. controls (3.31,0.3) and (6.95,1.4) .. (10.93,3.29)   ;
  \draw    (266.91,52.58) -- (314.5,4.99) ;
  \draw [shift={(265.5,53.99)}, rotate = 315] [color={rgb, 255:red, 0; green, 0; blue, 0 }  ][line width=0.75]    (10.93,-3.29) .. controls (6.95,-1.4) and (3.31,-0.3) .. (0,0) .. controls (3.31,0.3) and (6.95,1.4) .. (10.93,3.29)   ;
  \draw  [fill={rgb, 255:red, 0; green, 0; blue, 0 }  ,fill opacity=1 ] (132.19,95) .. controls (132.19,93.55) and (133.35,92.37) .. (134.77,92.37) .. controls (136.2,92.37) and (137.35,93.55) .. (137.35,95) .. controls (137.35,96.45) and (136.2,97.62) .. (134.77,97.62) .. controls (133.35,97.62) and (132.19,96.45) .. (132.19,95) -- cycle ;
  \draw  [fill={rgb, 255:red, 248; green, 231; blue, 28 }  ,fill opacity=1 ] (182.19,95) .. controls (182.19,93.55) and (183.35,92.37) .. (184.77,92.37) .. controls (186.2,92.37) and (187.35,93.55) .. (187.35,95) .. controls (187.35,96.45) and (186.2,97.62) .. (184.77,97.62) .. controls (183.35,97.62) and (182.19,96.45) .. (182.19,95) -- cycle ;
  \draw  [fill={rgb, 255:red, 0; green, 0; blue, 0 }  ,fill opacity=1 ] (132.19,145) .. controls (132.19,143.55) and (133.35,142.37) .. (134.77,142.37) .. controls (136.2,142.37) and (137.35,143.55) .. (137.35,145) .. controls (137.35,146.45) and (136.2,147.62) .. (134.77,147.62) .. controls (133.35,147.62) and (132.19,146.45) .. (132.19,145) -- cycle ;
  \draw  [fill={rgb, 255:red, 248; green, 231; blue, 28 }  ,fill opacity=1 ] (182.19,145) .. controls (182.19,143.55) and (183.35,142.37) .. (184.77,142.37) .. controls (186.2,142.37) and (187.35,143.55) .. (187.35,145) .. controls (187.35,146.45) and (186.2,147.62) .. (184.77,147.62) .. controls (183.35,147.62) and (182.19,146.45) .. (182.19,145) -- cycle ;
  \draw    (137.35,95) -- (180.19,95) ;
  \draw [shift={(182.19,95)}, rotate = 180] [color={rgb, 255:red, 0; green, 0; blue, 0 }  ][line width=0.75]    (10.93,-3.29) .. controls (6.95,-1.4) and (3.31,-0.3) .. (0,0) .. controls (3.31,0.3) and (6.95,1.4) .. (10.93,3.29)   ;
  \draw    (134.77,97.62) -- (134.77,140.37) ;
  \draw [shift={(134.77,142.37)}, rotate = 270] [color={rgb, 255:red, 0; green, 0; blue, 0 }  ][line width=0.75]    (10.93,-3.29) .. controls (6.95,-1.4) and (3.31,-0.3) .. (0,0) .. controls (3.31,0.3) and (6.95,1.4) .. (10.93,3.29)   ;
  \draw    (135.5,144.25) -- (183.09,96.66) ;
  \draw [shift={(184.5,95.25)}, rotate = 135] [color={rgb, 255:red, 0; green, 0; blue, 0 }  ][line width=0.75]    (10.93,-3.29) .. controls (6.95,-1.4) and (3.31,-0.3) .. (0,0) .. controls (3.31,0.3) and (6.95,1.4) .. (10.93,3.29)   ;
  \draw  [fill={rgb, 255:red, 0; green, 0; blue, 0 }  ,fill opacity=1 ] (262.19,95.24) .. controls (262.19,93.79) and (263.35,92.62) .. (264.77,92.62) .. controls (266.2,92.62) and (267.35,93.79) .. (267.35,95.24) .. controls (267.35,96.69) and (266.2,97.87) .. (264.77,97.87) .. controls (263.35,97.87) and (262.19,96.69) .. (262.19,95.24) -- cycle ;
  \draw  [fill={rgb, 255:red, 0; green, 0; blue, 0 }  ,fill opacity=1 ] (312.19,95.24) .. controls (312.19,93.79) and (313.35,92.62) .. (314.77,92.62) .. controls (316.2,92.62) and (317.35,93.79) .. (317.35,95.24) .. controls (317.35,96.69) and (316.2,97.87) .. (314.77,97.87) .. controls (313.35,97.87) and (312.19,96.69) .. (312.19,95.24) -- cycle ;
  \draw  [fill={rgb, 255:red, 248; green, 231; blue, 28 }  ,fill opacity=1 ] (262.19,145.24) .. controls (262.19,143.79) and (263.35,142.62) .. (264.77,142.62) .. controls (266.2,142.62) and (267.35,143.79) .. (267.35,145.24) .. controls (267.35,146.69) and (266.2,147.87) .. (264.77,147.87) .. controls (263.35,147.87) and (262.19,146.69) .. (262.19,145.24) -- cycle ;
  \draw  [fill={rgb, 255:red, 248; green, 231; blue, 28 }  ,fill opacity=1 ] (312.19,145.24) .. controls (312.19,143.79) and (313.35,142.62) .. (314.77,142.62) .. controls (316.2,142.62) and (317.35,143.79) .. (317.35,145.24) .. controls (317.35,146.69) and (316.2,147.87) .. (314.77,147.87) .. controls (313.35,147.87) and (312.19,146.69) .. (312.19,145.24) -- cycle ;
  \draw    (267.35,95.24) -- (310.19,95.24) ;
  \draw [shift={(312.19,95.24)}, rotate = 180] [color={rgb, 255:red, 0; green, 0; blue, 0 }  ][line width=0.75]    (10.93,-3.29) .. controls (6.95,-1.4) and (3.31,-0.3) .. (0,0) .. controls (3.31,0.3) and (6.95,1.4) .. (10.93,3.29)   ;
  \draw    (264.77,97.87) -- (264.77,140.62) ;
  \draw [shift={(264.77,142.62)}, rotate = 270] [color={rgb, 255:red, 0; green, 0; blue, 0 }  ][line width=0.75]    (10.93,-3.29) .. controls (6.95,-1.4) and (3.31,-0.3) .. (0,0) .. controls (3.31,0.3) and (6.95,1.4) .. (10.93,3.29)   ;
  \draw    (267.19,142.83) -- (314.77,95.24) ;
  \draw [shift={(265.77,144.24)}, rotate = 315] [color={rgb, 255:red, 0; green, 0; blue, 0 }  ][line width=0.75]    (10.93,-3.29) .. controls (6.95,-1.4) and (3.31,-0.3) .. (0,0) .. controls (3.31,0.3) and (6.95,1.4) .. (10.93,3.29)   ;
  \end{tikzpicture}
  
\end{center}
\captionsetup{skip=2pt}
\caption{Acyclic orientations of $G$, with first-level sinks colored in yellow}
\label{fig:thm2-example-acyclic-orientations}
\end{figure}

On the other hand, one can compute using SageMath that the chromatic symmetric function of $G$ is 
\[X_G=6e_{3,1},\]
and $(3, 1)'=(2, 1, 1)$. Therefore, this example agrees with Theorem \ref{thm:stanley3.4}.

In previous work by the first author and Spirkl \cite{delcon}, we generalized the chromatic symmetric function to \emph{vertex-weighted graphs}, and showed that in this setting the chromatic symmetric function admits a natural deletion-contraction relation. We derived a generalization of Theorem \ref{thm:stanley3.3} to vertex-weighted graphs, and provided a novel proof using deletion-contraction that is analogous to Stanley's famous proof that $(-1)^{|V(G)|}\chi_G(-1)$ enumerates acyclic orientations of $G$ \cite{stanleyacyclic}.

This paper has two main parts. First, we prove Theorem \ref{thm:main}, a generalization of Theorem \ref{thm:stanley3.4} to vertex-weighted graphs. We begin by generalizing the notion of a sink sequence to the vertex-weighted setting, allowing us to properly extend the ideas used in \cite{delcon}. This makes the proof of Theorem \ref{thm:main} easier to express as a combinatorial argument via an inductive edge deletion-contraction proof.

Second, we introduce Conjecture \ref{conjecture}, a conjectured generalization of Theorem \ref{thm:main} when $\ell(\mu) = 1$. This generalization would allow for a much wider range of acceptable $\mu$ than those that satisfy the two bullet points of Theorem \ref{thm:stanley3.4}. In particular, Conjecture \ref{conjecture} implies that every $\mu$ with one part is acceptable in unweighted claw-free graphs, implying a combinatorial interpretation for all associated $e$-coefficient sums. The authors expect that if Conjecture \ref{conjecture} can be generalized to $\mu$ of arbitrary length, this could provide a new combinatorial interpretation of any individual $e$-basis coefficient of the chromatic symmetric functions of unweighted claw-free graphs.

The paper is organized as follows: in Section 2, we introduce necessary background in symmetric function theory and graph theory. In Section 3, we prove Theorem \ref{thm:main}, the vertex-weighted generalization of Stanley's Theorem \ref{thm:stanley3.4}. In Section 4, we introduce Conjecture \ref{conjecture} with an illustrative example, and provide supporting evidence in the form of proofs of two special cases. We also discuss the application of Conjecture \ref{conjecture} to unweighted claw-free graphs. We end with concluding remarks in Section 5.

\section{Background}

Throughout this paper, $\mathbb{N}$ will be used to mean positive integers (not including zero), and $\mathcal{P}(\N)$ means the set of all subsets of positive integers, i.e. the power set of $\N$.

\subsection{Partitions and Symmetric Functions}

A \emph{partition} $\pi = \{S_1, \dots, S_k\}$ of a set $S$ is a set  of nonempty disjoint subsets of $S$ whose union is all of $S$ (that is, $S_1 \sqcup \dots \sqcup S_k = S$), and we write $\pi \vdash S$ and $|\pi| = |S|$. The elements of $\pi$ are called \emph{blocks} of the partition.

An \emph{integer partition} is a tuple $\lm = (\lm_1, \dots, \lm_k)$ of positive integers satisfying $\lm_1 \geq \dots \geq \lm_k$. Where $\sum \lm_i = n$, we say that $\lm$ is a \emph{partition of $n$}, and we write $\lm \vdash n$ or $|\lm| = n$. Each integer in the tuple $\lm$ is called a \emph{part} of $\lm$, and the number of parts of $\lm$ is $\ell(\lm)$. We let $n_i(\lm)$ be the number of occurrences of $i$ as a part of $\lm$. For example, if $\mu = (3,2,2,1)$, then $|\mu| = 8$, $\ell(\mu) = 4$, $n_1(\mu) = 1$, and $n_2(\mu) = 2$.

An integer partition $\lm \vdash d$ may also be written as $\lm = 1^{n_1(\lm)} \dots d^{n_d(\lm)}$, giving the multiplicity of each part. In particular, $1^k = \underbrace{(1,\dots,1)}_{\text{ k ones}}$.

Given an integer partition $\lm = (\lm_1, \dots, \lm_k)$, if $m \geq \lm_1$ we write $(m,\lm)$ to indicate the partition $(m,\lm_1,\dots,\lm_k)$, and if $r \leq \lm_k$ we write $(\lm,r)$ to indicate the partition $(\lm_1,\dots,\lm_k,r)$. We write $\lm+1^n$ to mean the partition formed by adding $1$ to each of the first $n$ parts of $\lm$, extending $\lm$ by $0$s if $\ell(\lm) < n$. For example, $(3,2,1) + 1^2 = (4,3,1)$, and $(3,2,1)+1^5 = (4,3,2,1,1)$. When $\ell(\lm) \geq n$, we likewise write $\lm-1^n$ to indicate the partition formed by subtracting $1$ from the first $n$ parts of $\lm$, removing any arising $0$s and rearranging the parts into weakly decreasing order if necessary. For example, $(3,2,2)-1^2=(2,2,1)$, and $(3,1,1) - 1^2 = (2,1)$.

Given integer partitions $\lm$ and $\mu$ with $|\lm| = |\mu|$, we say that $\lm$ \emph{dominates} $\mu$ if for each \\ $i \in \{1, \dots, \ell(\mu)\}$, we have $\sum_{j=1}^i \lm_j \geq \sum_{j=1}^i \mu_j$ (when $\ell(\lambda)<\ell(\mu)$, define $\lm_{\ell(\lm)+1} = \dots = \lm_{\ell(\mu)} = 0$ for checking dominance).

Given an integer partition $\lm$, its \emph{transpose} $\lm'$ is the partition with parts $\lm'_i = \sum_{j=i}^{\infty} n_j(\lm)$. In particular, $\lm'_1 = \ell(\lm)$.

Given $\pi \vdash S$, its corresponding integer partition $\lm(\pi) \vdash |S|$ has parts equal to the cardinalities of the blocks of $\pi$. 

The following information about symmetric function theory can be found in many textbooks, such as \cite{mac, stanleybook}. A \emph{symmetric function} is a power series $f(x_1,x_2,\dots) \in \mathbb{C}[[x_1,x_2,\dots]]$ of finite degree such that for every permutation $\sigma$ of the positive integers $\N$ (with only finitely many non-fixed points), we have $f(x_1,x_2,\dots) = f(x_{\sigma(1)}, x_{\sigma(2)}, \dots)$. The space of symmetric functions, denoted $\Lm$, may be recognized as a graded vector space $\Lm = \bigoplus_{i=0}^{\infty} \Lm^i$, where $\Lm^i$ consists of those symmetric functions which are homogeneous of degree $i$. Each $\Lm^i$ is finite-dimensional, with dimension equal to the number of integer partitions of $i$, and bases of $\Lm^i$ (and thus of $\Lm$) are typically indexed by these integer partitions. Some of the most commonly used bases are
\begin{itemize}
    \item The \emph{monomial basis}, defined by 
    $$
    m_{\lm} = \sum x_{i_1}^{\lm_1} \dots x_{i_{\ell(\lm)}}^{\lm_{\ell(\lm)}}
    $$
    where the sum contains one copy of each monomial formed as $(i_1, \dots, i_{\ell(\lm)})$ ranges across all tuples of distinct positive integers.
    \item The \emph{augmented} monomial basis, defined by
    $$
    \tm_{\lm} = \left(\prod_{i=1}^{\infty} n_i(\lm)!\right)m_{\lm}.
    $$
    \item The \emph{elementary symmetric function} basis, defined by
    $$
    e_n = \sum_{i_1 < \dots < i_n} x_{i_1} \dots x_{i_n}, \, e_{\lm} = e_{\lm_1} \dots e_{\lm_{\ell(\lm)}}.
    $$
\end{itemize}

    If $\{b_{\lm}\}$ is a basis of symmetric functions indexed by integer partitions, $\mu$ is a fixed integer partition, and $f$ is any symmetric function, $[b_{\mu}]f$ denotes the coefficient of $b_{\mu}$ when $f$ is expanded into the $b$-basis. The function $f$ is said to be \emph{$b$-positive} if $[b_{\mu}]f \geq 0$ for every integer partition $\mu$.

\subsection{Graphs}

We use basic graph theory terminology as given in \cite{diestel}. A graph $G = (V(G),E(G))$ consists of a set $V(G)$ of \emph{vertices} and a set $E(G)$ of (unordered) vertex pairs called \emph{edges}. Given an edge $v_1v_2 \in E(G)$ for $v_1,v_2 \in V(G)$, we say that $v_1$ and $v_2$ are the \emph{endpoints} of $e$, and that $e$ is \emph{incident} with $v_1$ and $v_2$. In this paper graphs need not be simple, meaning that we may have multiple edges with the same two vertices (\emph{multi-edges}) or an edge containing the same vertex twice (\emph{loops}). 

Given $S \subseteq V(G)$, we define $G \bk S$ to be the graph $(V(G) \bk S, E(G) \bk E(S))$, where $E(S)$ is the set of edges with at least one endpoint in $S$. We define $G[S]$ to be the graph $(S,E'(S))$, where $E'(S) \subseteq E(G)$ is the set of edges with both endpoints in $S$. We call $G[S]$ the \emph{subgraph of $G$ induced by $S$}, and say that $G[S]$ is an \emph{induced subgraph} of $G$.

A set $S \subseteq V(G)$ is called a \emph{stable set} if there are no edges of $G$ with both endpoints in $S$. A partition $\pi \vdash V(G)$ is called a \emph{stable partition} if each block of $\pi$ is a stable set.

An \emph{orientation} $\gamma$ of $G$ is an assignment of a direction to each edge $e$ of $G$ (that is, an ordering of the two vertices comprising $e$), and we will use $(G, \gamma)$ to denote a graph with orientation $\gamma$ applied. We denote an oriented edge as $v_1 \rightarrow v_2$ and say that the edge points from $v_1$ to $v_2$. An orientation of $G$ is \emph{acyclic} if it contains no directed cycle (that is, the graph has no loops, and for each vertex $v \in V(G)$, there do not exist vertices $v_1, \dots, v_k$ for some $k \geq 1$ such that all of the oriented edges $v \rightarrow v_1, v_1 \rightarrow v_2, \dots, v_k \rightarrow v$ are present in the orientation). 

A \emph{sink} of an orientation of a graph $G$ is any vertex $v \in V(G)$ such that no edges point away from $v$ (in particular, an isolated vertex is a sink of every orientation).

Given a graph $G$ and an edge $e = v_1v_2$, the graph $G \bk e = (V,E(G) \bk e)$ is the graph of $G$ with the edge $e$ \emph{deleted}. We define the \emph{contraction} of $G$ by $e$ as $G/e = G \bk e$ if $e$ is a loop, and otherwise $G/e = (V(G)\bk\{v_1,v_2\}\cup v^*, E(G)/e)$
where $E(G)/e$ consists of all edges of $E(G)$, except that wherever $v_1$ or $v_2$ occurs as an endpoint of an edge, it is replaced by $v^*$. Intuitively, we identify the endpoints of $e$ to a single vertex, and adjust all other edges accordingly. 

\subsection{Vertex-Weighted Graphs}

A \emph{vertex-weighted graph} $(G,w)$ consists of a graph $G$, and a vertex weight function $w: V(G) \rightarrow \mathbb{N}$. All previous definitions hold identically for vertex-weighted graphs, with the exception that when a vertex-weighted graph $(G,w)$ is contracted by a non-loop edge $e = v_1v_2$, we give $G/e$ a new weight function $w/e$ satisfying $(w/e)(v^*) = w(v_1)+w(v_2)$, and for all other $v \in V(G/e) \bk \{v^*\}$, $(w/e)(v) = w(v)$.

Given $S \subseteq V(G)$, we write $w(S) = \sum_{v \in S} w(S)$, and we say that the \emph{total weight} of $(G,w)$ is $w(V(G))$.

Since the usual definition of a graph may be captured by the special case where $w(v) = 1$ for all $v \in V(G)$, we will assume in this paper that all graphs are vertex-weighted.

\subsection{Graph Coloring}

Let $(G,w)$ be a vertex-weighted graph. A \emph{coloring} of $G$ is a map $\kappa: V(G) \rightarrow \mathbb{N}$ such that whenever $v_1v_2 \in E(G)$, $\kappa(v_1) \neq \kappa(v_2)$.
\begin{definition}[\cite{delcon, stanley}]\label{def:xgw}
The \textbf{chromatic symmetric function} of a vertex-weighted graph $(G,w)$ is
$$
X_{(G,w)} = \sum_{\kappa} \prod_{v \in V(G)} x_{\kappa(v)}^{w(v)}
$$
where the sum ranges over all colorings $\kappa$ of $G$.
\end{definition}

Although the chromatic symmetric function does not admit a direct edge deletion-contraction relation for unweighted graphs, for vertex-weighted graphs the following holds.
\begin{lemma}[\cite{delcon}, Lemma 2]\label{lem:delcon}

If $(G,w)$ is a vertex-weighted graph, and $e$ is any edge of $G$, then
$$
X_{(G,w)} = X_{(G \bk e, w)}-X_{(G/e,w/e)}.
$$

\end{lemma}

\section{Generalizing Theorem \ref{thm:stanley3.4} to Vertex-Weighted Graphs}

In \cite{delcon}, Spirkl and the first author generalized Theorem \ref{thm:stanley3.3} to vertex-weighted graphs. One of the main challenges was correctly generalizing the notion of counting sinks of acyclic orientations to vertex-weighted graphs $(G,w)$: should a sink vertex $v$ be counted once, or with weight $w(v)$? The answer turns out to be something in between these two: we need to pick not just acyclic orientations, but also sink maps that assign each sink $v$ a nonempty subset of $\{1,2,\dots,w(v)\}$ (intuitively, we view the vertex as consisting of $w(v)$ ``mini-vertices", and we choose a nonempty subset of these to be the ``true" sinks). Likewise, we will see that care needs to be taken in generalizing Theorem \ref{thm:stanley3.4}.

Theorem \ref{thm:stanley3.4}, instead of simply counting sinks of an acyclic orientation, now enumerates the sequence of sinks obtained by recursively deleting the sinks of an acyclic orientation and considering the sinks of the remainder of the orientation. Already there is a minor difficulty in combining this with the notion of sink maps above: intuitively if a sink $v$ is assigned a nontrivial subset of $\{1,2,\dots,w(v)\}$, we would like to then consider the graph where the weight of $v$ is decreased, and consists only of the remainder of $\{1,2,\dots,w(v)\}$. The difficulty here is that it will be beneficial to keep track of when different subsets of $\{1,2,\dots,w(v)\}$ are used, but that is not possible under the definition of vertex-weighted graphs. Therefore, it is necessary to extend the notion of weighted graphs.

\subsection{Generalizations Related to Graphs}

\begin{definition}\label{def:xgwset}

A \textbf{set-weighted graph} $(G,\omega)$ consists of a graph $G$ and a map $\omega: V(G) \rightarrow \mathcal{P}(\N)$ such that
\begin{itemize}
    \item For each $v \in V(G)$, the set $\omega(v)$ is nonempty and finite.
    \item For each $n \in \N$, $n$ occurs as an element of at most one $\omega(v)$.
\end{itemize}

We say that the \emph{integer weight} of $v \in V(G)$ is then $w(v) = |\omega(v)|$.
\medskip

Given a set-weighted graph $(G,\omega)$ and an edge $e \in E(G)$, the \textbf{contraction} $(G,\omega)/e = (G/e, \omega/e)$ is defined analogously to contraction in vertex-weighted graphs, except that where $e = v_1v_2$ is a nonloop edge and $v^*$ is the vertex formed by contraction, we define $(\omega/e)(v^*) = \omega(v_1) \cup \omega(v_2)$, and for a vertex $v \neq v^*$ we have $(\omega/e)(v) = \omega(v)$.

\end{definition}

\begin{definition}\label{def:xgom}

The chromatic symmetric function of a set-weighted graph $(G,\omega)$ is given by
$$
X_{(G,\omega)} = \sum_{\kappa} \prod_{v \in V(G)} x_{\kappa(v)}^{w(v)}.
$$

\end{definition}

It is straightforward to verify that Lemma \ref{lem:delcon} extends to a deletion-contraction relation on set-weighted graphs.

\begin{lemma}\label{lem:delconset}

If $(G,\omega)$ is a set-weighted graph, and $e$ is any edge of $G$, then
$$
X_{(G,\omega)} = X_{(G \bk e, \omega)}-X_{(G/e,\omega/e)}.
$$

\end{lemma}

Thus, we now label the ``mini-vertices" explicitly, without changing the fundamental notion of the integer-weighted chromatic symmetric function.

Now, the definitions that follow illustrate how we combine the sink sequences of Theorem \ref{thm:stanley3.4} with the notion above of sink maps as used in the generalization of Theorem \ref{thm:stanley3.3} in \cite{delcon}.

\begin{definition}
Let $\ell\in\N$ and let $(G,\omega)$ be a set-weighted graph. An \textbf{$\ell$-step weight map} of $(G,\omega)$ is a function $S:V(G)\to \bp{\mathcal{P}(\N)}^{\ell}$ such that for all $v\in V(G)$, we have 
\[\bigsqcup_{i=1}^{\ell}S(v)_i\subseteq \omega(v)\]
where $S(v)_i$ is the $i^{th}$ coordinate of $S(v)$ (note that this is a disjoint union, so each element of $\omega(v)$ occurs in at most one of the $S(v)_i$).

We define the \textbf{$\ell$-step weight sequence} of an $\ell$-step weight map $S$ to be $\wts(G,S)=(s_1,\ldots,s_\ell)$, where for all $i \in \{1,\dots,\ell\}$, we have
\[s_i=\sum_{v\in V(G)}|S(v)_i|.\]

\end{definition}

As an example, consider the set-weighted graph $(G,\omega)$ given in Figure \ref{fig:l-step-weight-map}.
Consider $S:V(G)\to\bp{\mathcal{P}(\N)}^4$ given by
\[S(v_1)=\bp{\varnothing,\varnothing,\varnothing,\{1,2\}},\quad S(v_2)=\bp{\{3,5\}, \{4\}, \varnothing, \varnothing},\]
\begin{equation}S(v_3)=\bp{\varnothing,\varnothing,\{6\},\varnothing},\quad S(v_4)=\bp{\{7\},\{8\},\varnothing,\varnothing}.\label{eq:l-step-weight-map}\end{equation}
Then by our definition, $S$ is a 4-step weight map of $(G,\omega)$. The 4-step weight sequence given by $S$ is $(3, 2, 1, 2)$.

\begin{figure}[hbt]
\begin{center}

  \tikzset{every picture/.style={line width=0.75pt}} 

  \begin{tikzpicture}[x=0.75pt,y=0.75pt,yscale=-1,xscale=1]
  
  \draw  [fill={rgb, 255:red, 0; green, 0; blue, 0 }  ,fill opacity=1 ] (23,18.75) .. controls (23,17.23) and (24.23,16) .. (25.75,16) .. controls (27.27,16) and (28.5,17.23) .. (28.5,18.75) .. controls (28.5,20.27) and (27.27,21.5) .. (25.75,21.5) .. controls (24.23,21.5) and (23,20.27) .. (23,18.75) -- cycle ;
  \draw  [fill={rgb, 255:red, 0; green, 0; blue, 0 }  ,fill opacity=1 ] (132,18.75) .. controls (132,17.23) and (133.23,16) .. (134.75,16) .. controls (136.27,16) and (137.5,17.23) .. (137.5,18.75) .. controls (137.5,20.27) and (136.27,21.5) .. (134.75,21.5) .. controls (133.23,21.5) and (132,20.27) .. (132,18.75) -- cycle ;
  \draw  [fill={rgb, 255:red, 0; green, 0; blue, 0 }  ,fill opacity=1 ] (79,82.75) .. controls (79,81.23) and (80.23,80) .. (81.75,80) .. controls (83.27,80) and (84.5,81.23) .. (84.5,82.75) .. controls (84.5,84.27) and (83.27,85.5) .. (81.75,85.5) .. controls (80.23,85.5) and (79,84.27) .. (79,82.75) -- cycle ;
  \draw  [fill={rgb, 255:red, 0; green, 0; blue, 0 }  ,fill opacity=1 ] (186,82.75) .. controls (186,81.23) and (187.23,80) .. (188.75,80) .. controls (190.27,80) and (191.5,81.23) .. (191.5,82.75) .. controls (191.5,84.27) and (190.27,85.5) .. (188.75,85.5) .. controls (187.23,85.5) and (186,84.27) .. (186,82.75) -- cycle ;
  \draw    (28.5,18.75) -- (132,18.75) ;
  \draw    (25.75,21.5) -- (79,82.75) ;
  \draw    (84.5,82.75) -- (134.75,21.5) ;
  \draw    (84.5,82.75) -- (186,82.75) ;
  
  \draw (105,-5) node [anchor=north west][inner sep=0.75pt]  [font=\small]  {$v_{2} ,\{3,4,5\}$};
  \draw (167,91.4) node [anchor=north west][inner sep=0.75pt]  [font=\small]  {$v_{4} ,\{7,8\}$};
  \draw (-1,-5) node [anchor=north west][inner sep=0.75pt]  [font=\small]  {$v_{1} ,\{1,2\}$};
  \draw (65,91.4) node [anchor=north west][inner sep=0.75pt]  [font=\small]  {$v_{3} ,\{6\}$};
  \end{tikzpicture}
  
\end{center}
\captionsetup{skip=2pt}
\caption{A set-weighted graph $(G,\omega)$, where the set associated with each vertex is written next to the vertex label.}
\label{fig:l-step-weight-map}
\end{figure}

\begin{definition}
Let $S$ be an $\ell$-step weight map on $(G, \omega)$. For $i \in \{0, \dots, \ell\}$, define the set-weighted \textbf{graph sequence $(G_i(S), \omega_i(S))$ induced by $S$} (where we may suppress $S$ when it is clear) recursively as follows:
\begin{itemize}
    \item $(G_0, \omega_0) = (G, \omega)$.
    \item For $i\geq 1$, given $(G_{i-1}, \omega_{i-1})$:
    \begin{itemize}
        \item Set $V(G_i) = V(G_{i-1}) \setminus \{v \in V(G_{i-1}) : S(v)_{i} = \omega_{i-1}(v)\}$.
        \item For each $v \in V(G_i)$, set $\omega_i(v) = \omega_{i-1}(v) \bk S(v)_i$.
        \item Set $E(G_i)$ to be the set of all edges of $E(G_{i-1})$ with both endpoints in $V(G_i)$.
        
    \end{itemize}
\end{itemize}

\end{definition}

Intuitively, suppose we are given $G_{i-1}$. For all $v\in V(G_{i-1})$, we remove the mini-vertices given by $S(v)_i$, and we remove the ``whole" vertex if there is no mini-vertex left, and then define the resulting graph to be $G_i$.

Taking the set-weighted graph in Figure \ref{fig:l-step-weight-map} and the weight map given in equation \ref{eq:l-step-weight-map} as an example, the following Figure \ref{fig:graph-sequence} illustrates the process defined in the previous definition. The numbers in red represents the weight sequence $S$.

\begin{figure}[hbt]
\begin{center}

\tikzset{every picture/.style={line width=0.75pt}} 

\begin{tikzpicture}[x=0.75pt,y=0.75pt,yscale=-1,xscale=1]

\draw  [fill={rgb, 255:red, 0; green, 0; blue, 0 }  ,fill opacity=1 ] (23.64,32.2) .. controls (23.64,30.75) and (24.8,29.57) .. (26.22,29.57) .. controls (27.65,29.57) and (28.81,30.75) .. (28.81,32.2) .. controls (28.81,33.65) and (27.65,34.82) .. (26.22,34.82) .. controls (24.8,34.82) and (23.64,33.65) .. (23.64,32.2) -- cycle ;
\draw  [fill={rgb, 255:red, 0; green, 0; blue, 0 }  ,fill opacity=1 ] (125.93,32.2) .. controls (125.93,30.75) and (127.08,29.57) .. (128.51,29.57) .. controls (129.93,29.57) and (131.09,30.75) .. (131.09,32.2) .. controls (131.09,33.65) and (129.93,34.82) .. (128.51,34.82) .. controls (127.08,34.82) and (125.93,33.65) .. (125.93,32.2) -- cycle ;
\draw  [fill={rgb, 255:red, 0; green, 0; blue, 0 }  ,fill opacity=1 ] (76.19,93.25) .. controls (76.19,91.8) and (77.35,90.63) .. (78.77,90.63) .. controls (80.2,90.63) and (81.35,91.8) .. (81.35,93.25) .. controls (81.35,94.7) and (80.2,95.88) .. (78.77,95.88) .. controls (77.35,95.88) and (76.19,94.7) .. (76.19,93.25) -- cycle ;
\draw  [fill={rgb, 255:red, 0; green, 0; blue, 0 }  ,fill opacity=1 ] (176.6,93.25) .. controls (176.6,91.8) and (177.76,90.63) .. (179.18,90.63) .. controls (180.61,90.63) and (181.76,91.8) .. (181.76,93.25) .. controls (181.76,94.7) and (180.61,95.88) .. (179.18,95.88) .. controls (177.76,95.88) and (176.6,94.7) .. (176.6,93.25) -- cycle ;
\draw    (28.81,32.2) -- (125.93,32.2) ;
\draw    (26.22,34.82) -- (76.19,93.25) ;
\draw    (81.35,93.25) -- (128.51,34.82) ;
\draw    (81.35,93.25) -- (176.6,93.25) ;
\draw  [fill={rgb, 255:red, 0; green, 0; blue, 0 }  ,fill opacity=1 ] (234.78,29.34) .. controls (234.78,27.89) and (235.93,26.71) .. (237.36,26.71) .. controls (238.78,26.71) and (239.94,27.89) .. (239.94,29.34) .. controls (239.94,30.78) and (238.78,31.96) .. (237.36,31.96) .. controls (235.93,31.96) and (234.78,30.78) .. (234.78,29.34) -- cycle ;
\draw  [fill={rgb, 255:red, 0; green, 0; blue, 0 }  ,fill opacity=1 ] (337.06,29.34) .. controls (337.06,27.89) and (338.22,26.71) .. (339.64,26.71) .. controls (341.07,26.71) and (342.22,27.89) .. (342.22,29.34) .. controls (342.22,30.78) and (341.07,31.96) .. (339.64,31.96) .. controls (338.22,31.96) and (337.06,30.78) .. (337.06,29.34) -- cycle ;
\draw  [fill={rgb, 255:red, 0; green, 0; blue, 0 }  ,fill opacity=1 ] (287.33,90.39) .. controls (287.33,88.94) and (288.48,87.77) .. (289.91,87.77) .. controls (291.33,87.77) and (292.49,88.94) .. (292.49,90.39) .. controls (292.49,91.84) and (291.33,93.01) .. (289.91,93.01) .. controls (288.48,93.01) and (287.33,91.84) .. (287.33,90.39) -- cycle ;
\draw  [fill={rgb, 255:red, 0; green, 0; blue, 0 }  ,fill opacity=1 ] (387.73,90.39) .. controls (387.73,88.94) and (388.89,87.77) .. (390.32,87.77) .. controls (391.74,87.77) and (392.9,88.94) .. (392.9,90.39) .. controls (392.9,91.84) and (391.74,93.01) .. (390.32,93.01) .. controls (388.89,93.01) and (387.73,91.84) .. (387.73,90.39) -- cycle ;
\draw    (239.94,29.34) -- (337.06,29.34) ;
\draw    (237.36,31.96) -- (287.33,90.39) ;
\draw    (292.49,90.39) -- (339.64,31.96) ;
\draw    (292.49,90.39) -- (387.73,90.39) ;
\draw  [fill={rgb, 255:red, 0; green, 0; blue, 0 }  ,fill opacity=1 ] (435.59,30.29) .. controls (435.59,28.84) and (436.75,27.67) .. (438.17,27.67) .. controls (439.6,27.67) and (440.75,28.84) .. (440.75,30.29) .. controls (440.75,31.74) and (439.6,32.91) .. (438.17,32.91) .. controls (436.75,32.91) and (435.59,31.74) .. (435.59,30.29) -- cycle ;
\draw  [fill={rgb, 255:red, 0; green, 0; blue, 0 }  ,fill opacity=1 ] (488.14,91.34) .. controls (488.14,89.9) and (489.3,88.72) .. (490.72,88.72) .. controls (492.15,88.72) and (493.3,89.9) .. (493.3,91.34) .. controls (493.3,92.79) and (492.15,93.97) .. (490.72,93.97) .. controls (489.3,93.97) and (488.14,92.79) .. (488.14,91.34) -- cycle ;
\draw    (438.17,32.91) -- (488.14,91.34) ;
\draw  [fill={rgb, 255:red, 0; green, 0; blue, 0 }  ,fill opacity=1 ] (566.03,44.6) .. controls (566.03,43.15) and (567.18,41.98) .. (568.61,41.98) .. controls (570.03,41.98) and (571.19,43.15) .. (571.19,44.6) .. controls (571.19,46.05) and (570.03,47.22) .. (568.61,47.22) .. controls (567.18,47.22) and (566.03,46.05) .. (566.03,44.6) -- cycle ;
\draw [color={rgb, 255:red, 0; green, 0; blue, 0 }  ,draw opacity=1 ][line width=1]    (181.29,62.01) -- (228.49,62.01) ;
\draw [shift={(231.49,62.01)}, rotate = 180] [color={rgb, 255:red, 0; green, 0; blue, 0 }  ,draw opacity=1 ][line width=1]    (14.21,-4.28) .. controls (9.04,-1.82) and (4.3,-0.39) .. (0,0) .. controls (4.3,0.39) and (9.04,1.82) .. (14.21,4.28)   ;
\draw [color={rgb, 255:red, 0; green, 0; blue, 0 }  ,draw opacity=1 ][line width=1]    (383.98,62.01) -- (431.18,62.01) ;
\draw [shift={(434.18,62.01)}, rotate = 180] [color={rgb, 255:red, 0; green, 0; blue, 0 }  ,draw opacity=1 ][line width=1]    (14.21,-4.28) .. controls (9.04,-1.82) and (4.3,-0.39) .. (0,0) .. controls (4.3,0.39) and (9.04,1.82) .. (14.21,4.28)   ;
\draw [color={rgb, 255:red, 0; green, 0; blue, 0 }  ,draw opacity=1 ][line width=1]    (510,62.01) -- (550.34,62.01) ;
\draw [shift={(553.34,62.01)}, rotate = 180] [color={rgb, 255:red, 0; green, 0; blue, 0 }  ,draw opacity=1 ][line width=1]    (14.21,-4.28) .. controls (9.04,-1.82) and (4.3,-0.39) .. (0,0) .. controls (4.3,0.39) and (9.04,1.82) .. (14.21,4.28)   ;

\draw (100.25,5.76) node [anchor=north west][inner sep=0.75pt]  [font=\small]  {$v_{2} ,\{\textcolor[rgb]{0.82,0.01,0.11}{3} ,4,\textcolor[rgb]{0.82,0.01,0.11}{5}\}$};
\draw (156.08,101.15) node [anchor=north west][inner sep=0.75pt]  [font=\small]  {$v_{4} ,\{\textcolor[rgb]{0.82,0.01,0.11}{7} ,8\}$};
\draw (2.18,4.8) node [anchor=north west][inner sep=0.75pt]  [font=\small]  {$v_{1} ,\{1,2\}$};
\draw (59.85,101.15) node [anchor=north west][inner sep=0.75pt]  [font=\small]  {$v_{3} ,\{6\}$};
\draw (312.28,2.89) node [anchor=north west][inner sep=0.75pt]  [font=\small]  {$v_{2} ,\{\textcolor[rgb]{0.82,0.01,0.11}{4}\}$};
\draw (367.64,98.29) node [anchor=north west][inner sep=0.75pt]  [font=\small]  {$v_{4} ,\{\textcolor[rgb]{0.82,0.01,0.11}{8}\}$};
\draw (213.32,1.94) node [anchor=north west][inner sep=0.75pt]  [font=\small]  {$v_{1} ,\{1,2\}$};
\draw (270.99,98.29) node [anchor=north west][inner sep=0.75pt]  [font=\small]  {$v_{3} ,\{6\}$};
\draw (414.13,2.89) node [anchor=north west][inner sep=0.75pt]  [font=\small]  {$v_{1} ,\{1,2\}$};
\draw (471.8,99.25) node [anchor=north west][inner sep=0.75pt]  [font=\small]  {$v_{3} ,\{\textcolor[rgb]{0.82,0.01,0.11}{6}\}$};
\draw (544.56,17.2) node [anchor=north west][inner sep=0.75pt]  [font=\small]  {$v_{1} ,\{1,2\}$};
\draw (78.1,136.41) node [anchor=north west][inner sep=0.75pt]    {$( G_{0} ,\omega _{0})$};
\draw (280.79,135.45) node [anchor=north west][inner sep=0.75pt]    {$( G_{1} ,\omega _{1})$};
\draw (445.94,135.45) node [anchor=north west][inner sep=0.75pt]    {$( G_{2} ,\omega _{2})$};
\draw (545.41,135.45) node [anchor=north west][inner sep=0.75pt]    {$( G_{3} ,\omega _{3})$};

\end{tikzpicture}
\end{center}
\captionsetup{skip=2pt}
\caption{Graph sequence formed by the 4-step weight sequence $S$}
\label{fig:graph-sequence}
\end{figure}

Essentially, we care about when an $\ell$-step weight map $S$ yields a graph sequence $(G_i,\omega_i)$ that corresponds to the graphs and sinks that are recursively formed in computing the sink sequence of an acyclic orientation $\gamma$ of $(G,\omega)$. However, we write the definitions above to not depend inherently on such a choice of $\gamma$, because it will be easier for proofs to consider all choices of $\gamma$ and $S$ and discard those that do not work together.

\begin{definition}
Given a set-weighted graph $(G,\omega)$ and an acyclic orientation $\gamma$ of $G$, we say that an $\ell$-step weight map $S$ of $(G,\omega)$ is \textbf{$\gamma$-admissible} if for all $i \in \{1,\dots,\ell\}$ and for all $v \in V(G)$, it holds that $S(v)_i\neq \varnothing$ if and only if $v \in V(G_{i-1})$ and $v$ is a sink of the restriction of $\gamma$ to  $G_{i-1}$. When $S$ is $\gamma$-admissible, we will denote its corresponding weight sequence as $\wts(G,\gamma,S)$.
\end{definition}

We again take the graph $(G,\omega)$ in Figure \ref{fig:l-step-weight-map} and the 4-step weight map $S$ in equation \ref{eq:l-step-weight-map} as an example. We give an acyclic orientation $\gamma$ of $G$ by directing the edges as $v_1\to v_2$, $v_1\to v_3$, $v_3\to v_2$, and $v_3\to v_4$. Let us visualize the graph sequence obtained from $(G,\omega)$, $\gamma$, and $S$.

\begin{figure}[hbt]
\begin{center}

  \tikzset{every picture/.style={line width=0.75pt}} 

  \begin{tikzpicture}[x=0.75pt,y=0.75pt,yscale=-1,xscale=1]
  
  \draw  [fill={rgb, 255:red, 0; green, 0; blue, 0 }  ,fill opacity=1 ] (23.64,32.2) .. controls (23.64,30.75) and (24.8,29.57) .. (26.22,29.57) .. controls (27.65,29.57) and (28.81,30.75) .. (28.81,32.2) .. controls (28.81,33.65) and (27.65,34.82) .. (26.22,34.82) .. controls (24.8,34.82) and (23.64,33.65) .. (23.64,32.2) -- cycle ;
  \draw  [fill={rgb, 255:red, 0; green, 0; blue, 0 }  ,fill opacity=1 ] (125.93,32.2) .. controls (125.93,30.75) and (127.08,29.57) .. (128.51,29.57) .. controls (129.93,29.57) and (131.09,30.75) .. (131.09,32.2) .. controls (131.09,33.65) and (129.93,34.82) .. (128.51,34.82) .. controls (127.08,34.82) and (125.93,33.65) .. (125.93,32.2) -- cycle ;
  \draw  [fill={rgb, 255:red, 0; green, 0; blue, 0 }  ,fill opacity=1 ] (76.19,93.25) .. controls (76.19,91.8) and (77.35,90.63) .. (78.77,90.63) .. controls (80.2,90.63) and (81.35,91.8) .. (81.35,93.25) .. controls (81.35,94.7) and (80.2,95.88) .. (78.77,95.88) .. controls (77.35,95.88) and (76.19,94.7) .. (76.19,93.25) -- cycle ;
  \draw  [fill={rgb, 255:red, 0; green, 0; blue, 0 }  ,fill opacity=1 ] (176.6,93.25) .. controls (176.6,91.8) and (177.76,90.63) .. (179.18,90.63) .. controls (180.61,90.63) and (181.76,91.8) .. (181.76,93.25) .. controls (181.76,94.7) and (180.61,95.88) .. (179.18,95.88) .. controls (177.76,95.88) and (176.6,94.7) .. (176.6,93.25) -- cycle ;
  \draw    (28.81,32.2) -- (123.93,32.2) ;
  \draw [shift={(125.93,32.2)}, rotate = 180] [color={rgb, 255:red, 0; green, 0; blue, 0 }  ][line width=0.75]    (10.93,-3.29) .. controls (6.95,-1.4) and (3.31,-0.3) .. (0,0) .. controls (3.31,0.3) and (6.95,1.4) .. (10.93,3.29)   ;
  \draw    (26.22,34.82) -- (74.89,91.73) ;
  \draw [shift={(76.19,93.25)}, rotate = 229.46] [color={rgb, 255:red, 0; green, 0; blue, 0 }  ][line width=0.75]    (10.93,-3.29) .. controls (6.95,-1.4) and (3.31,-0.3) .. (0,0) .. controls (3.31,0.3) and (6.95,1.4) .. (10.93,3.29)   ;
  \draw    (81.35,93.25) -- (127.25,36.38) ;
  \draw [shift={(128.51,34.82)}, rotate = 128.9] [color={rgb, 255:red, 0; green, 0; blue, 0 }  ][line width=0.75]    (10.93,-3.29) .. controls (6.95,-1.4) and (3.31,-0.3) .. (0,0) .. controls (3.31,0.3) and (6.95,1.4) .. (10.93,3.29)   ;
  \draw    (81.35,93.25) -- (174.6,93.25) ;
  \draw [shift={(176.6,93.25)}, rotate = 180] [color={rgb, 255:red, 0; green, 0; blue, 0 }  ][line width=0.75]    (10.93,-3.29) .. controls (6.95,-1.4) and (3.31,-0.3) .. (0,0) .. controls (3.31,0.3) and (6.95,1.4) .. (10.93,3.29)   ;
  \draw  [fill={rgb, 255:red, 0; green, 0; blue, 0 }  ,fill opacity=1 ] (234.78,29.34) .. controls (234.78,27.89) and (235.93,26.71) .. (237.36,26.71) .. controls (238.78,26.71) and (239.94,27.89) .. (239.94,29.34) .. controls (239.94,30.78) and (238.78,31.96) .. (237.36,31.96) .. controls (235.93,31.96) and (234.78,30.78) .. (234.78,29.34) -- cycle ;
  \draw  [fill={rgb, 255:red, 0; green, 0; blue, 0 }  ,fill opacity=1 ] (337.06,29.34) .. controls (337.06,27.89) and (338.22,26.71) .. (339.64,26.71) .. controls (341.07,26.71) and (342.22,27.89) .. (342.22,29.34) .. controls (342.22,30.78) and (341.07,31.96) .. (339.64,31.96) .. controls (338.22,31.96) and (337.06,30.78) .. (337.06,29.34) -- cycle ;
  \draw  [fill={rgb, 255:red, 0; green, 0; blue, 0 }  ,fill opacity=1 ] (287.33,90.39) .. controls (287.33,88.94) and (288.48,87.77) .. (289.91,87.77) .. controls (291.33,87.77) and (292.49,88.94) .. (292.49,90.39) .. controls (292.49,91.84) and (291.33,93.01) .. (289.91,93.01) .. controls (288.48,93.01) and (287.33,91.84) .. (287.33,90.39) -- cycle ;
  \draw  [fill={rgb, 255:red, 0; green, 0; blue, 0 }  ,fill opacity=1 ] (387.73,90.39) .. controls (387.73,88.94) and (388.89,87.77) .. (390.32,87.77) .. controls (391.74,87.77) and (392.9,88.94) .. (392.9,90.39) .. controls (392.9,91.84) and (391.74,93.01) .. (390.32,93.01) .. controls (388.89,93.01) and (387.73,91.84) .. (387.73,90.39) -- cycle ;
  \draw    (239.94,29.34) -- (335.06,29.34) ;
  \draw [shift={(337.06,29.34)}, rotate = 180] [color={rgb, 255:red, 0; green, 0; blue, 0 }  ][line width=0.75]    (10.93,-3.29) .. controls (6.95,-1.4) and (3.31,-0.3) .. (0,0) .. controls (3.31,0.3) and (6.95,1.4) .. (10.93,3.29)   ;
  \draw    (237.36,31.96) -- (286.03,88.87) ;
  \draw [shift={(287.33,90.39)}, rotate = 229.46] [color={rgb, 255:red, 0; green, 0; blue, 0 }  ][line width=0.75]    (10.93,-3.29) .. controls (6.95,-1.4) and (3.31,-0.3) .. (0,0) .. controls (3.31,0.3) and (6.95,1.4) .. (10.93,3.29)   ;
  \draw    (292.49,90.39) -- (338.39,33.52) ;
  \draw [shift={(339.64,31.96)}, rotate = 128.9] [color={rgb, 255:red, 0; green, 0; blue, 0 }  ][line width=0.75]    (10.93,-3.29) .. controls (6.95,-1.4) and (3.31,-0.3) .. (0,0) .. controls (3.31,0.3) and (6.95,1.4) .. (10.93,3.29)   ;
  \draw    (292.49,90.39) -- (385.73,90.39) ;
  \draw [shift={(387.73,90.39)}, rotate = 180] [color={rgb, 255:red, 0; green, 0; blue, 0 }  ][line width=0.75]    (10.93,-3.29) .. controls (6.95,-1.4) and (3.31,-0.3) .. (0,0) .. controls (3.31,0.3) and (6.95,1.4) .. (10.93,3.29)   ;
  \draw  [fill={rgb, 255:red, 0; green, 0; blue, 0 }  ,fill opacity=1 ] (435.59,30.29) .. controls (435.59,28.84) and (436.75,27.67) .. (438.17,27.67) .. controls (439.6,27.67) and (440.75,28.84) .. (440.75,30.29) .. controls (440.75,31.74) and (439.6,32.91) .. (438.17,32.91) .. controls (436.75,32.91) and (435.59,31.74) .. (435.59,30.29) -- cycle ;
  \draw  [fill={rgb, 255:red, 0; green, 0; blue, 0 }  ,fill opacity=1 ] (488.14,91.34) .. controls (488.14,89.9) and (489.3,88.72) .. (490.72,88.72) .. controls (492.15,88.72) and (493.3,89.9) .. (493.3,91.34) .. controls (493.3,92.79) and (492.15,93.97) .. (490.72,93.97) .. controls (489.3,93.97) and (488.14,92.79) .. (488.14,91.34) -- cycle ;
  \draw    (438.17,32.91) -- (486.84,89.82) ;
  \draw [shift={(488.14,91.34)}, rotate = 229.46] [color={rgb, 255:red, 0; green, 0; blue, 0 }  ][line width=0.75]    (10.93,-3.29) .. controls (6.95,-1.4) and (3.31,-0.3) .. (0,0) .. controls (3.31,0.3) and (6.95,1.4) .. (10.93,3.29)   ;
  \draw  [fill={rgb, 255:red, 0; green, 0; blue, 0 }  ,fill opacity=1 ] (566.03,44.6) .. controls (566.03,43.15) and (567.18,41.98) .. (568.61,41.98) .. controls (570.03,41.98) and (571.19,43.15) .. (571.19,44.6) .. controls (571.19,46.05) and (570.03,47.22) .. (568.61,47.22) .. controls (567.18,47.22) and (566.03,46.05) .. (566.03,44.6) -- cycle ;
  \draw [color={rgb, 255:red, 0; green, 0; blue, 0 }  ,draw opacity=1 ][line width=1]    (181.29,62.01) -- (228.49,62.01) ;
  \draw [shift={(231.49,62.01)}, rotate = 180] [color={rgb, 255:red, 0; green, 0; blue, 0 }  ,draw opacity=1 ][line width=1]    (14.21,-4.28) .. controls (9.04,-1.82) and (4.3,-0.39) .. (0,0) .. controls (4.3,0.39) and (9.04,1.82) .. (14.21,4.28)   ;
  \draw [color={rgb, 255:red, 0; green, 0; blue, 0 }  ,draw opacity=1 ][line width=1]    (383.98,62.01) -- (431.18,62.01) ;
  \draw [shift={(434.18,62.01)}, rotate = 180] [color={rgb, 255:red, 0; green, 0; blue, 0 }  ,draw opacity=1 ][line width=1]    (14.21,-4.28) .. controls (9.04,-1.82) and (4.3,-0.39) .. (0,0) .. controls (4.3,0.39) and (9.04,1.82) .. (14.21,4.28)   ;
  \draw [color={rgb, 255:red, 0; green, 0; blue, 0 }  ,draw opacity=1 ][line width=1]    (510,62.01) -- (550.34,62.01) ;
  \draw [shift={(553.34,62.01)}, rotate = 180] [color={rgb, 255:red, 0; green, 0; blue, 0 }  ,draw opacity=1 ][line width=1]    (14.21,-4.28) .. controls (9.04,-1.82) and (4.3,-0.39) .. (0,0) .. controls (4.3,0.39) and (9.04,1.82) .. (14.21,4.28)   ;
  
  \draw (100.25,5.76) node [anchor=north west][inner sep=0.75pt]  [font=\small]  {$v_{2} ,\{\textcolor[rgb]{0.82,0.01,0.11}{3} ,4,\textcolor[rgb]{0.82,0.01,0.11}{5}\}$};
  \draw (156.08,101.15) node [anchor=north west][inner sep=0.75pt]  [font=\small]  {$v_{4} ,\{\textcolor[rgb]{0.82,0.01,0.11}{7} ,8\}$};
  \draw (2.18,4.8) node [anchor=north west][inner sep=0.75pt]  [font=\small]  {$v_{1} ,\{1,2\}$};
  \draw (59.85,101.15) node [anchor=north west][inner sep=0.75pt]  [font=\small]  {$v_{3} ,\{6\}$};
  \draw (312.28,2.89) node [anchor=north west][inner sep=0.75pt]  [font=\small]  {$v_{2} ,\{\textcolor[rgb]{0.82,0.01,0.11}{4}\}$};
  \draw (367.64,98.29) node [anchor=north west][inner sep=0.75pt]  [font=\small]  {$v_{4} ,\{\textcolor[rgb]{0.82,0.01,0.11}{8}\}$};
  \draw (213.32,1.94) node [anchor=north west][inner sep=0.75pt]  [font=\small]  {$v_{1} ,\{1,2\}$};
  \draw (270.99,98.29) node [anchor=north west][inner sep=0.75pt]  [font=\small]  {$v_{3} ,\{6\}$};
  \draw (414.13,2.89) node [anchor=north west][inner sep=0.75pt]  [font=\small]  {$v_{1} ,\{1,2\}$};
  \draw (471.8,99.25) node [anchor=north west][inner sep=0.75pt]  [font=\small]  {$v_{3} ,\{\textcolor[rgb]{0.82,0.01,0.11}{6}\}$};
  \draw (544.56,17.2) node [anchor=north west][inner sep=0.75pt]  [font=\small]  {$v_{1} ,\{1,2\}$};
  \draw (78.1,136.41) node [anchor=north west][inner sep=0.75pt]    {$( G_{0} ,\omega _{0})$};
  \draw (280.79,135.45) node [anchor=north west][inner sep=0.75pt]    {$( G_{1} ,\omega _{1})$};
  \draw (445.94,135.45) node [anchor=north west][inner sep=0.75pt]    {$( G_{2} ,\omega _{2})$};
  \draw (545.41,135.45) node [anchor=north west][inner sep=0.75pt]    {$( G_{3} ,\omega _{3})$};

  \end{tikzpicture}

\end{center}
\captionsetup{skip=2pt}
\caption{Graph sequence formed by $S$ and $\gamma$}
\label{fig:directed-graph-sequence}
\end{figure}

We note that it does hold that for all $i\in\{1,\ldots,\ell\}$ and for all vertices $v$, $S(v)_i\neq \varnothing$ if and only if $v\in V(G_{i-1})$ and $v$ is a sink of the restriction of $\gamma$ to $G_{i-1}$. Therefore, $S$ is $\gamma$-admissible.

For proving our main theorem, we will often be interested in the case where during iteration with respect to an acyclic orientation $\gamma$ we delete all the mini-vertices of each sink up to some point. To make this particular discussion easier we introduce two more terms.

\begin{definition}\label{def:sinkseq}
Let $\gamma$ be an acyclic orientation of $G$. For each $i\in\N$, recursively define graphs by $G^0 = G$ and $G^i$ the graph formed by deleting all sinks from $G^{i-1}$ induced by $\gamma$. Define $\Sink_i(\gamma)$ be the set of all sinks of the restriction of $\gamma$ to $G^{i-1}$. For each $i$, let $\sink_i(\gamma)=|\Sink_i(\gamma)|$.

The \textbf{type} of $\gamma$ is the sequence $\lambda$ of positive integers such that 
\[\lambda_i=\sum_{v\in \Sink_i(\gamma)}w(v)\]
for all $i$ such that $Sink_i(\gamma) \neq \varnothing$. In particular, if we rearrange the terms of $\lambda$ in non-increasing order, then we get an allowable partition of $(G,\omega)$.
\end{definition}

This is just generalizing the equivalent notion of Stanley \cite{stanley} to weighted graphs.

\begin{definition}\label{def:standard}
Let $S$ be an $(\ell+1)$-step weight map of $(G,\omega)$, and let $\gamma$ be an acyclic orientation of $G$. We say that $S$ is in \textbf{$\gamma$-standard form} if for all $1\leq i\leq\ell$,
\[S(v)_i=\begin{cases}
\omega(v) & \text{ if } v\in \Sink_i(\gamma) \\
\varnothing & \text{ if } v\notin \Sink_i(\gamma),
\end{cases}\]
and $S(v)_{\ell+1}\neq \varnothing$ if and only if $v\in\Sink_{\ell+1}(\gamma)$. Note that if $S$ is in $\gamma$-standard form, then $S$ is $\gamma$-admissible.
\end{definition}

Again consider $(G,\omega)$, $S$, and $\gamma$ given in previous examples. Note that the $S$ we gave previously is \textit{not} in $\gamma$-standard form since, for example, $S(v_2)_1=\{3,5\}$ is not the entire weight $\omega(v_2)=\{3,4,5\}$.
One can check that the following $S'$
\[S'(v_1)=\bp{\varnothing,\varnothing,\{1,2\}},\quad S'(v_2)=\bp{\{3, 4, 5\},\varnothing,\varnothing},\]
\[S'(v_3)=\bp{\varnothing,\{6\},\varnothing},\quad S'(v_4)=\bp{\{7,8\},\varnothing,\varnothing}\]
\textit{is} in $\gamma$-standard form.

\subsection{Generalizations Related to Partitions and Symmetric Functions}

We also introduce notation for the specific sums of $e$-basis coefficients we will be looking at.

\begin{definition}\label{def:sigmamuj}
Let $f\in \Lambda^d$ and write $f=\sum_{\lambda\vdash d}c_{\lambda}e_{\lambda}$. Let $\mu=(\mu_1,\ldots,\mu_\ell)\vdash r\leq d$.
Then we define 
\[\sigma_{\mu}(f):=\sum_{\substack{\lambda\vdash d \\ \lambda'=(\mu_1,\ldots,\mu_\ell,\ldots)}}c_\lambda\]
where the sum ranges over all $\lm \vdash d$ such that the first $\ell$ parts of $\lm'$ are nonzero and satisfy $\lm'_i = \mu_i$ for all $i \in \{1,\dots,\ell\}$.
\end{definition}

We also briefly state without proof a well-known property of symmetric function expansions (\cite[Theorem 7.4.4]{stanleybook}) that will be used repeatedly.
\begin{lemma}[\cite{stanleybook}]\label{lem:edom}

Let $\lm$ and $\mu$ be partitions of the same integer such that $[e_{\mu}]m_{\lm} \neq 0$. Then $\mu$ dominates $\lm'$.

\end{lemma}

Before proceeding to the main theorem of this section, we first prove auxiliary lemmas that will be necessary.

\begin{lemma}\label{lem:mu+1^k}
Let $\mu$ be a partition, and let $\lambda\vdash|\mu|$.
Let $k\geq \mu_1$ be given.
Then 
\[[e_\lambda]m_{\mu}=[e_{\lambda+1^k}]m_{(k,\mu)}\]
where $\lm+1^k$ is the partition formed by adding $1$ to the first $k$ parts of $\lm$ (extending $\lm$ with $0$s if necessary to make $\ell(\lm) \geq k$), and if $\mu = (\mu_1, \dots, \mu_{\ell(\mu)})$ then $(k,\mu) = (k, \mu_1, \dots, \mu_{\ell(\mu)})$.
\end{lemma}

\begin{proof}

It is well-known that for any integer partitions $\lm$ and $\mu$ it holds that $[e_{\lm}]m_{\mu} = [e_{\mu}]m_{\lm}$ (that is, the transition matrix between these bases is symmetric) \cite[Corollary 7.4.2]{stanleybook}. Therefore, it is equivalent to show that 
\[[e_\mu]m_\lambda=[e_{(k,\mu)}]m_{\lambda+1^k}.\]
Furthermore, since $[e_\mu]m_\lambda=[e_{(k,\mu)}]m_\lambda e_k$ and $e_k=m_{1^k}$,
it suffices to show that 
\[[e_{(k,\mu)}]m_{\lambda+1^k}=[e_{(k,\mu)}]m_{\lambda}m_{1^k}.\]
By expanding $m_\lambda m_{1^k}$, it is straightforward to verify that $[m_{\lambda+1^k}]m_\lambda m_{1^k}=1$.
Let $\nu$ be such that $[m_\nu]m_\lambda m_{1^k}\neq 0$ and $\nu\neq \lambda+1^k$.
Then again by expanding $m_\lambda m_{1^k}$, we may verify that $\ell(\nu)>k$. It follows that $(k,\mu)$ does not dominate $\nu'$, in which case it follows from Lemma \ref{lem:edom} that $[e_{(k,\mu)}]m_\nu=0$, and this finishes the proof.
\end{proof}

\begin{lemma}
\label{base-case-induction}
Let $\mu$ and $\nu$ be partitions. Let $k\in\N$ be such that $k\geq \mu_1$ and $k\geq\nu_1$.
Then 
\[\sigma_{(k,\nu)}(m_{(k,\mu)})=\sigma_\nu(m_\mu).\]
\end{lemma}

\begin{proof}

Let $\lambda\vdash k+|\mu|$ be such that $\lambda'$ has the form $(k,\nu,\ldots)$ (so in particular $\ell(\lm) = k$). 
Then $\lambda_* = \lambda-1^k$ is a partition of $|\mu|$ such that $\lambda_*'$ has the form $(\nu,\ldots)$, and by the previous lemma, we have 
\[[e_\lambda]m_{(k,\mu)}=[e_{\lambda_*}]m_\mu.\]
Conversely, suppose $\lambda$ is a partition of $|\mu|$ such that $\lambda'$ has the form $(\nu,\ldots)$. Let $\lambda_*=\lambda+1^k$ be a partition of $k+|\mu|$.
Note that since $k\geq\nu_1$, $\lambda_*'$ has the form $(k,\nu,\ldots)$.
Again, using the previous lemma, we get 
\[[e_\lambda]m_\mu=[e_{\lambda_*}]m_{(k,\mu)}.\]
Therefore, we have $\sigma_{k,\nu}(m_{(k,\mu)})=\sigma_\nu(m_\mu)$, as desired.
\end{proof}

Before proceeding, we introduce terminology naming the specific properties from Theorem \ref{thm:stanley3.4} that will be referenced in some final lemmas and in the main theorem generalizing Theorem \ref{thm:stanley3.4}.

\begin{definition}\label{def:partialdom}
Let $\mu=(\mu_1,\ldots,\mu_\ell)$ and $\nu=(\nu_1,\ldots,\nu_m)$ be finite sequences of positive integers (note that $\mu$ and $\nu$ need not be partitions).
We say $\mu$ \textbf{partially dominates} $\nu$ if either $\mu_i=\nu_i$ for all $i \in \{1,\dots,\ell\}$, or there exists some $i \in \{1,\dots,\ell\}$ such that $\mu_1+\cdots+\mu_i>\nu_1+\cdots+\nu_i$ (where we take $\nu_j = 0$ if $j > m$).
\end{definition}

In other words, $\mu$ partially dominates $\nu$ if $\nu$ does not dominate $\mu$ ``nontrivially". This is the required condition for either of the two bullet points of Theorem \ref{thm:stanley3.4}.

\begin{definition}\label{def:maxallow}
A partition $\mu=(\mu_1,\ldots,\mu_\ell)$ is \textbf{allowable} in a set-weighted graph $(G,\omega)$ if there exists $W \subseteq V(G)$ and a stable partition $(S_1,\ldots,S_\ell)$ of $W$ such that $\sum_{v\in S_i}w(v)=\mu_i$ for all $i \in \{1,\dots,\ell\}$.

A partition $\mu\vdash r\leq d$ is \textbf{maximal} in $(G,\omega)$ if $\mu$ partially dominates all allowable partitions in $(G,\omega)$ (not just allowable partitions of $r$). Note that $\mu$ need not be allowable.
\end{definition}

One example of these definitions for unweighted graphs was given in the introduction after Theorem \ref{thm:stanley3.4}. 
For an example on set-weighted graphs, let us consider the graph $(G,\omega)$ in Figure \ref{fig:l-step-weight-map}.
All \textit{allowable} partitions of $(G,\omega)$ are 
\begin{gather*}
(1), (2), (3), (2,2), \\
(4, 1), (3, 2), (2, 2, 1), \\
(5, 1), (3, 2, 1), \\
(5, 2), (4, 3), (3, 2, 2), \\
(5, 2, 1), (4, 3, 1), (3, 2, 2, 1).
\end{gather*}
We note that $\mu=(5,2)$ is a \textit{maximal} partition.
Indeed, for all allowable partitions $\nu$ that do not start with 5, we have $\mu_1>\nu_1$.
For $\nu=(5,1)$, we have $\mu_1+\mu_2>\nu_1+\nu_2$.
For $\nu=(5, 2)$ or $\nu=(5, 2, 1)$, we have $\mu_1=\nu_1$ and $\mu_2=\nu_2$.
Therefore, $\mu$ partially dominates all allowable partitions, so $\mu$ is maximal.

Now we proceed with further technical lemmas.

\begin{lemma}
\label{maximal}
Let $(G,\omega)$ be a set-weighted graph with $n$ vertices and total weight $d$. Let \\ $\mu=(\mu_1,\ldots,\mu_\ell)$ be a partition.
Let $\gamma$ be an acyclic orientation of $G$, and let $\lambda$ be the type of $\gamma$. Then 
\begin{itemize}
  \item[(a)] If $\mu_i=\lambda_i$ for all $1\leq i\leq\ell$, then there exists exactly one $\gamma$-admissible $\ell$-step weight map on $(G,\omega)$ with $\wts(G,\gamma,S)=(\mu_1,\ldots,\mu_\ell)$, namely the map $S$ such that for all $1\leq i\leq\ell$,
  \[S(v)_i=\begin{cases}
  \omega(v) & \text{ if } v\in \Sink_i(\gamma) \\
  \varnothing & \text{ if } v\notin \Sink_i(\gamma).
  \end{cases}\]
  \item[(b)] If there exists some $1\leq i\leq\ell$ such that $\mu_1+\cdots+\mu_i>\lambda_1+\cdots+\lambda_i$, then there does not exist a $\gamma$-admissible $\ell$-step weight map on $(G,\omega)$ such that $\wts(G,S)=(\mu_1,\ldots,\mu_\ell)$.
  \item[(c)] Let $\mu$ be a maximal partition in $(G,\omega)$, and let $S$ be an $(\ell+1)$-step weight map on $(G,\omega)$ with $\wts(G,\gamma,S)=(\mu_1,\ldots,\mu_\ell,j)$ for some $j$. Then $S$ is $\gamma$-admissible if and only if $S$ is in $\gamma$-standard form. 
\end{itemize}
\end{lemma}

\begin{proof}

It is straightforward to verify (a). For (b), assume to the contrary that there exists a $\gamma$-admissible $\ell$-step weight map $S$ such that $\wts(G,\gamma,S)=(\mu_1,\ldots,\mu_\ell)$.
Then since $S$ is $\gamma$-admissible, we have $\mu_1+\cdots+\mu_i\leq \lambda_1+\cdots+\lambda_i$, as otherwise there is not enough weight among the corresponding vertices of $G$ to build $S$, giving a contradiction.

For (c), it follows from the definition that if $S$ is in $\gamma$-standard form, then $S$ is $\gamma$-admissible.
Conversely, 
assume that $S$ is $\gamma$-admissible.
Since $\mu$ is maximal, $\mu$ partially dominates the partition obtained by sorting the parts of $\lambda$ in non-decreasing order. It is easy to verify that then $\mu$ also partially dominates $\lambda$ as an unordered integer sequence.
Then by part (b), it is the case that $\mu_i=\lambda_i$ for all $1\leq i\leq\ell$, and then by part (a) it follows that $S$ is in $\gamma$-standard form.
\end{proof}

\begin{lemma}\label{lem:sinkpath}

Let $G$ be a graph and $\gamma$ an acyclic orientation of $G$. Then for every $v \in V(G)$, $v \in \Sink_i(\gamma)$ if and only if the length of the longest directed path in $\gamma$ starting at $v$ contains $i$ vertices.

\end{lemma}

\begin{proof}

The proof is by induction on $i$.  It is straightforward to verify the claim for $i=1$. 

For the inductive step, assume the claim holds for all positive integers less than or equal to a fixed $k$. Using the notation of Definition \ref{def:sinkseq}, suppose that $v \in \Sink_{k+1}(\gamma)$, so $v$ is a sink of the restriction of $\gamma$ to $G^{k}$. Note that each step in the construction moving from $G^i$ to $G^{i+1}$ for $i \in \{0, \dots, k-1\}$ removes the last vertex of each directed path starting at $v$ that is present in $G^i$, and no other vertices along these paths. Since $v$ has not been deleted, it follows that at least one such directed path contains at least $k$ vertices in addition to $v$, so has length at least $k+1$.

On the other hand, if there was a directed path starting at $v$ in $\gamma$ of length at least $k+2$, then $v$ would not be a sink of $G^k$, since it would have an outgoing edge remaining. Therefore, the longest directed path in $\gamma$ starting at $v$ contains $k+1$ vertices.

Conversely, it is easy to check that if the longest directed path in $\gamma$ starting at $v$ contains $i$ vertices, then $v \in \Sink_i(\gamma)$, since as noted above the process defining the sink sets removes exactly one vertex from each such directed path at each step.
\end{proof}

As our proof will mainly use edge deletion and contraction, we will need to understand how these operations affect the type of an acyclic orientation. Actually, we will be using a non-edge as our focal point, so we introduce the following notation.

\begin{definition}\label{def:delconori}

Let $\gamma$ be a (not necessarily acyclic) orientation of $G$, and let $S$ be an $(\ell+1)$-step weight map of $G$. Let $e = v_1v_2$ be a \textnormal{non-edge} of $G$.
\begin{itemize}
  \item In $G+e$, let $\varphi_{v_1}(\gamma)$ be the orientation whose restriction to $G$ is $\gamma$, and the direction of $e$ is $v_1\to v_2$. Let $\varphi_{v_2}(\gamma)$ be the orientation whose restriction to $G$ is $\gamma$, and the direction of $e$ is $v_2\to v_1$. Let $\psi_{+}(S)$ be the $(\ell+1)$-step weight map on $G+e$ such that $\psi_{+}(S)=S$.
  \item In $G/e$, let $v^*$ represent the vertex formed by contraction, and let $\varphi_{v^*}(\gamma)$ be the orientation obtained by contracting $e$ in $\gamma$.
  Let $\psi_{*}(S)$ be the $(\ell+1)$-step weight map on $G/e$ such that for all $v\in V(G/e)$ and for all $i \in \{1,\dots,\ell\}$,
  \[\psi_{*}(S)(v)_i=\begin{cases}
  S(v)_i & \text{ if } v\neq v^* \\
  S(v_1)_i\cup S(v_2)_i & \text{ if } v=v^*.
  \end{cases}\]
  \item If $\gamma$ is an acyclic orientation with type $\lm$, then whenever they are acyclic, denote the types of $\varphi_{v_1}(\gamma)$, $\varphi_{v_2}(\gamma)$, and $\varphi_{v^*}(\gamma)$ by $\lm^{v_1}, \lm^{v_2},$ and $\lm^{v^*}$ respectively.
\end{itemize}
    
\end{definition}

\begin{figure}[hbt]
\begin{center}

\tikzset{every picture/.style={line width=0.75pt}} 



\end{center}
\captionsetup{skip=2pt}
\caption{Examples of $(G+e,\varphi_{v_1}(\gamma))$, $(G+e,\varphi_{v_2}(\gamma))$, and $(G/e,\varphi_{v^*}(\gamma))$.}
\label{fig:non-edge-example}
\end{figure}

Thus, for example, if $\gamma$ has type $\lm$, then $\lm^{v_1}_2$ denotes the second entry of the type of $\varphi_{v_1}(\gamma)$ (recalling as before that the entries of the type need not be in non-increasing order).
In Figure \ref{fig:non-edge-example}, we give examples of the various constructions mentioned above in Definition \ref{def:delconori}.

We will need one more important lemma for our proof. 

\begin{lemma}\label{obs:sinks}
Let $\gamma$ be an \textbf{acyclic} orientation of $G$, and let $e = v_1v_2$ be a \textnormal{non-edge} of $G$. Suppose that there is no directed path in $\gamma$ from $v_1$ to $v_2$ nor from $v_2$ to $v_1$, so that all of $\varphi_{v_1}(\gamma)$, $\varphi_{v_2}(\gamma)$, and $\varphi_{v^*}(\gamma)$ are acyclic as well.

Then for all $k \in \mathbb{N}$ and all $a \in \{v_1,v_2,v^*\}$,
\begin{equation}\label{eq:sinkcontract}
\lm_1^a + \dots + \lm_k^a \leq \lm_1 + \dots + \lm_k.
\end{equation}

Suppose further that $v_1 \in \Sink_i(\gamma)$ and $v_2 \in \Sink_j(\gamma)$. Then the above inequality is \textbf{strict} in the following cases:
\begin{itemize}
\item $a = v^*$, $k = i$ and $i < j$.
\item $a=v^*$, $k=j$ and $j < i$.
\item $a = v_1$, $k = i$, and $i < j+1$.
\item $a = v_2$, $k = j$ and $j < i+1$.
\end{itemize}

\end{lemma}

\begin{proof}

The main inequality is clear, since the sink level of any vertex does not decrease in passing from $\gamma$ to any of $\varphi_{v_1}(\gamma), \varphi_{v_2}(\gamma),$ or $\varphi_{v^*}(\gamma)$, since any directed path in $\gamma$ from a vertex to a sink still exists. Thus, the number of vertices at sink level $k$ or below in $\gamma$ cannot increase in passing to any of $\varphi_{v_1}(\gamma), \varphi_{v_2}(\gamma),$ or $\varphi_{v^*}(\gamma)$.

To prove strict inequality in the special cases, it is enough to show that some vertex at sink level $k$ or below has its sink level increase to above level $k$. This is easy to check directly in each case; for example, the important vertex is $v_1$ for the first case, since the weight of this vertex is incorporated into $v^*$ and thus rises from contributing to $\lm_i$ to contributing to $\lm^*_j$ for $j > i$ after contraction. Similarly, it is straightforward to verify that the critical vertex is $v_1$ for the third case, and $v_2$ for the second and fourth cases.
\end{proof}

In particular, we will often be using this lemma in conjunction with Lemma \ref{maximal}(b) in the proof of the main theorem.

\medskip






We will need one more tool, which is a set-weighted version of Theorem \ref{thm:stanley3.3}. The original formulation of this theorem was proved by the first author and Spirkl for vertex-weighted graphs, but it is straightforward to extend to set-weighted graphs, and it is presented here using the terminology developed so far.

\begin{theorem}[Theorem 8, \cite{delcon}]\label{one-level-sink}

Let $(G,\omega)$ be a vertex-weighted graph with $n$ vertices and total weight $d$. Then 
\[\sigma_j\bp{X_{(G,\omega)}}=(-1)^{d-n}\sum_{\wts(\gamma,S)=(j)}(-1)^{j-\sink_1(\gamma)},\]
where the sum ranges over all ordered pairs consisting of an acyclic orientation $\gamma$ of $G$ and a $\gamma$-admissible one-step weight map $S$ of $G$ such that $\wts(G,\gamma,S)=(j)$.

\end{theorem}

\subsection{Main Theorem}

We are now ready to present the main theorem.

\begin{theorem}\label{thm:main}
Let $(G,\omega)$ be a set-weighted graph with $n$ vertices and total weight $d$. Suppose that $X_{(G,\omega)}=\sum_{\lambda\vdash d}c_{\lambda}e_{\lambda}$. Let $\mu=(\mu_1,\ldots,\mu_\ell)\vdash r\leq d$ be a maximal partition in $(G,\omega)$. Fix $0\leq j\leq d-r$, where we can choose $j=0$ only when $r=d$. Then
\begin{equation}
\label{main-thm}
\sigma_{\mu,j}\bp{X_{(G,\omega)}}=(-1)^{d-n}\sum_{\substack{\wts(\gamma,S)=(\mu_1,\ldots,\mu_\ell,j) \\ S\text{ admissible}}} (-1)^{|(\mu,j)|-\sum_{i=1}^{\ell+1}\sink_i(\gamma)},
\end{equation}
summed over all ordered pairs consisting of an acyclic orientation $\gamma$ of $G$, and a $\gamma$-admissible $(\ell+1)$-step weight map $S$ of $G$ such that $\wts(G, \gamma,S)=(\mu_1,\ldots,\mu_\ell,j)$.
\end{theorem}

\begin{figure}[hbt]
\begin{center}

  \tikzset{every picture/.style={line width=0.75pt}} 

  \begin{tikzpicture}[x=0.75pt,y=0.75pt,yscale=-1,xscale=1]
  
  \draw  [fill={rgb, 255:red, 0; green, 0; blue, 0 }  ,fill opacity=1 ] (22.19,26.25) .. controls (22.19,24.8) and (23.35,23.63) .. (24.77,23.63) .. controls (26.2,23.63) and (27.35,24.8) .. (27.35,26.25) .. controls (27.35,27.7) and (26.2,28.88) .. (24.77,28.88) .. controls (23.35,28.88) and (22.19,27.7) .. (22.19,26.25) -- cycle ;
  \draw  [fill={rgb, 255:red, 0; green, 0; blue, 0 }  ,fill opacity=1 ] (72.19,26.25) .. controls (72.19,24.8) and (73.35,23.63) .. (74.77,23.63) .. controls (76.2,23.63) and (77.35,24.8) .. (77.35,26.25) .. controls (77.35,27.7) and (76.2,28.88) .. (74.77,28.88) .. controls (73.35,28.88) and (72.19,27.7) .. (72.19,26.25) -- cycle ;
  \draw  [fill={rgb, 255:red, 0; green, 0; blue, 0 }  ,fill opacity=1 ] (22.19,76.25) .. controls (22.19,74.8) and (23.35,73.63) .. (24.77,73.63) .. controls (26.2,73.63) and (27.35,74.8) .. (27.35,76.25) .. controls (27.35,77.7) and (26.2,78.88) .. (24.77,78.88) .. controls (23.35,78.88) and (22.19,77.7) .. (22.19,76.25) -- cycle ;
  \draw  [fill={rgb, 255:red, 0; green, 0; blue, 0 }  ,fill opacity=1 ] (72.19,76.25) .. controls (72.19,74.8) and (73.35,73.63) .. (74.77,73.63) .. controls (76.2,73.63) and (77.35,74.8) .. (77.35,76.25) .. controls (77.35,77.7) and (76.2,78.88) .. (74.77,78.88) .. controls (73.35,78.88) and (72.19,77.7) .. (72.19,76.25) -- cycle ;
  \draw    (27.35,26.25) -- (72.19,26.25) ;
  \draw    (24.77,28.88) -- (24.77,73.63) ;
  \draw    (25.5,75.5) -- (74.5,26.5) ;
  \draw    (27.35,76.25) -- (72.19,76.25) ;
  
  \draw (-10,4.15) node [anchor=north west][inner sep=0.75pt]  [font=\small]  {$v_1, \{1,2\}$};
  \draw (55,4.15) node [anchor=north west][inner sep=0.75pt]  [font=\small]  {$v_2, \{3,4\}$};
  \draw (-17,83.15) node [anchor=north west][inner sep=0.75pt]  [font=\small]  {$v_3, \{5,6,7\}$};
  \draw (60,83.15) node [anchor=north west][inner sep=0.75pt]  [font=\small]  {$v_4, \{8,9\}$};

  \end{tikzpicture}
  
\end{center}
\captionsetup{skip=2pt}
\caption{$(G,\omega)$}
\label{fig:main-thm-example-graph}
\end{figure}

Before diving into the proof of the main theorem, we give an example applying the theorem. Consider the graph $(G,\omega)$ in Figure \ref{fig:main-thm-example-graph}.
All \textit{allowable} partitions are 
\[(2), (2, 2), (2, 2, 2),\]
\[(3), (3, 2), (3, 2, 2), (3, 2, 2, 2),\]
\[(4), (4, 3), (4, 2),\]
\[(4, 3, 2).\]
We observe that $\mu=(4)$ is a \textit{maximal} partition since it partially dominates all allowable partitions. Let us apply Theorem \ref{thm:main} with $\mu=(4)$ and $j=3$ on $(G,\omega)$. Using SageMath, one can compute 
\[\sigma_{4,3}\bp{X_{(G,\omega)}}=-2.\]
We then compute the right-hand side of (\ref{main-thm}). 
Note that $d=9$ and $n=4$.
We first list all acyclic orientations of $G$ in Figure \ref{fig:main-thm-example-orientations}. (We omit the vertex names $v_1,v_2,v_3,v_4$ in the diagrams.) For each acyclic orientation, the first-level sinks are colored in yellow. We observe that only the two acyclic orientations boxed in red can produce admissible weight maps that have the weight sequence $(4, 3)$.
Let us call the boxed acyclic orientation in the first row of Figure \ref{fig:main-thm-example-orientations} $\gamma_1$, and the other boxed acyclic orientation $\gamma_2$.
Note that the weight map $S_1$ given by: (the vertex labels are given in Figure \ref{fig:main-thm-example-graph})
\[S_1(v_1)=\bp{\{1, 2\},\varnothing}, S_1(v_2)=\bp{\varnothing,\varnothing},\]
\[S_1(v_3)=\bp{\varnothing, \{5,6,7\}}, S_1(v_4)=\bp{\{8,9\},\varnothing}\]
is the only $\gamma_1$-admissible 2-step weight map that produces a weight sequence of $(4,3)$.
For $\gamma_2$, the weight map $S_2$ given by:
\[S_2(v_1)=\bp{\varnothing,\varnothing}, S_2(v_2)=\bp{\{3, 4\},\varnothing},\]
\[S_2(v_3)=\bp{\varnothing, \{5, 6, 7\}}, S_2(v_4)=\bp{\{8, 9\},\varnothing}\]
is the only $\gamma_2$-admissible 2-step weight map that produces a weight sequence of $(4,3)$. Therefore, the right-hand side of (\ref{main-thm}) becomes
\begin{align*}
  &\quad (-1)^{9-4}\lrp{(-1)^{|(4,3)|-\sum_{i=1}^{2}\sink_i(\gamma_1)}+(-1)^{|(4,3)|-\sum_{i=1}^{2}\sink_i(\gamma_1)}} \\
  &= -\lrp{(-1)^{7-2-1}+(-1)^{7-2-1}} = -2 = \sigma_{4,3}\bp{X_{(G,\omega)}},
\end{align*}
which agrees with Theorem \ref{thm:main}.

\begin{figure}[hbt]
\begin{center}

\tikzset{every picture/.style={line width=0.75pt}} 



\end{center}
\captionsetup{skip=2pt}
\caption{All acyclic orientations of $(G,\omega)$, sinks are colored in yellow}
\label{fig:main-thm-example-orientations}
\end{figure}

We will shortly start the proof of Theorem \ref{thm:main}.

\subsection*{Proof Outline}

For clarity given the length of the proof, we first provide a roadmap here. The main points in this outline will correspond to headings in the proof for ease of reading.

The proof will be by induction on the number of \textit{non-edges} of $(G, \omega)$. For the base case, we consider when the underlying graph $G$ has 0 non-edges, i.e. when $G$ is a complete graph. In this case $X_{(G,\omega)} = \tm_{\lm}$ where the parts of $\lm$ are equal to the integer weights of the vertices of $G$. We prove the theorem in this case using the fact that $\mu$ partially dominates $\lambda$; by definition, partial dominance can happen in two ways, so we consider the two cases separately.

For the inductive step, we use the deletion-contraction relation of chromatic symmetric functions. We fix some non-edge $e=v_1v_2$ and consider the graphs $G+e$ and $G/e$, which both have fewer non-edges than $G$. To apply the inductive hypothesis, we first show that every maximal partition of $G$ is also a maximal partition in $G+e$ and a maximal partition in $G/e$. We then use the inductive hypothesis and the deletion-contraction relation to simplify \eqref{main-thm}, the equation of Theorem \ref{thm:main}. Finally, we divide into cases depending on the sink-level relationship between $v_1$ and $v_2$.


\begin{proof}

As stated before, the proof is by induction on the number of \textit{non-edges} of $(G, \omega)$.
We assume that all orientations $\gamma$ occurring in sums within the proof are acyclic unless otherwise stated.

\subsection*{Base Case}

It suffices to consider the case when $G$ is a simple complete graph, since any graph with multi-edges has the same chromatic symmetric function as the same graph with each multi-edge replaced by a single edge.

Let $(G,\omega)$ be a complete set-weighted graph, with vertices labelled $v_1, \dots, v_n$ such that their integer weights satisfy $w(v_1)\geq w(v_2)\geq\cdots\geq w(v_n)$.
Note that $X_{(G,\omega)}=\widetilde{m}_\lambda$, where $\lambda=(w(v_1),\ldots,w(v_n))\vdash d$ is an allowable partition of $(G,\omega)$.
Since $\mu$ is maximal, $\mu$ partially dominates $\lambda$. There are two ways this can happen; we examine what happens in both.

\subsubsection*{$\mu$ partially dominates $\lm$: Case 1}

First suppose there exists some $i \in \{1,\dots,\ell\}$ such that $\mu_1+\ldots+\mu_i>w(v_1)+\cdots+w(v_i)$.
Let $\nu\vdash d$ such that $\nu'$ has the form $(\mu_1,\ldots,\mu_\ell,j,\ldots)$.
Then we know that $\lambda$ does not dominate $\nu'$, or equivalently $\nu$ does not dominate $\lambda'$.
Hence by Lemma \ref{lem:edom} we have $[e_\nu]\tm_\lambda = [e_\nu]m_\lambda=0$, so the left-hand side of \eqref{main-thm} in this case is
\[\sigma_{\mu,j}\bp{X_{(G,\omega)}}=0.\]
To compute the right-hand side of \eqref{main-thm}, we claim that there does not exist an acyclic orientation $\gamma$ and a $\gamma$-admissible $(\ell+1)$-step weight map $S$ such that $\wts(G, \gamma,S)=(\mu_1,\ldots,\mu_\ell,j)$.
Indeed, let $\gamma$ be any acyclic orientation, and let $\nu$ be the type of $\gamma$. Since $G$ is a complete graph, $\nu$ is a permutation of $\lambda$. Since $w(v_1)\geq\cdots\geq w(v_n)$, we must have 
\[\nu_1+\cdots+\nu_i\leq w(v_1)+\cdots+w(v_i)<\mu_1+\cdots+\mu_i.\]
Then by Lemma \ref{maximal}, there does not exist a $\gamma$-admissible $\ell$-step weight map with weight sequence $(\mu_1,\ldots,\mu_\ell)$, so it follows that there does not exist a $\gamma$-admissible $(\ell+1)$-step weight map with weight sequence $(\mu_1,\ldots,\mu_\ell,j)$. Therefore, the right-hand side of \eqref{main-thm} is $0$, and \eqref{main-thm} holds in this case.

\subsubsection*{$\mu$ partially dominates $\lm$: Case 2}

We then assume that $\mu_i = w(v_i)$ for each $i \in \{1,\dots,\ell\}$. Let us first simplify the left-hand side of \eqref{main-thm}. 
Recall that for each $a\in\N$, $n_a$ denotes the number of times that $a$ occurs as a part of $\lambda$. Then 
\[\sigma_{\mu,j}\bp{X_{(G,\omega)}}=\sigma_{\mu,j}(\widetilde{m}_\lambda)=\sigma_{\mu,j}\bp{\lrp{\ts{\prod_{a=1}^{\infty}n_a!}}m_\lambda}=\bp{\ts{\prod_{a=1}^{\infty}n_a!}}\cdot \sigma_{\mu,j}(m_\lambda).\]
Then since $\mu_i=w(v_i)$ for all $i \in \{1,\dots,\ell\}$, we can apply Lemma \ref{base-case-induction} repeatedly and obtain 

\begin{align*}
  \sigma_{\mu,j}\bp{X_{(G,\omega)}} &= \bp{\ts{\prod_{a=1}^{\infty}n_a!}}\cdot \sigma_{\mu,j}(m_\lambda) \\
  &= \bp{\ts{\prod_{a=1}^{\infty}n_a!}}\cdot \sigma_j\bp{m_{(w(v_{\ell+1}),\ldots,w(v_n))}} \\
  &= \frac{\prod_{a=1}^{\infty}n_a!}{\prod_{a=1}^{\infty}n'_a!}\cdot \sigma_j\bp{\widetilde{m}_{(w(v_{\ell+1}),\ldots,w(v_n))}},
\end{align*}
where for each $a\in\N$, $n'_a$ is the number of times that $a$ occurs as a part of the partition $\bp{w(v_{\ell+1}),\ldots,w(v_n)}$.

Now we may use Theorem \ref{one-level-sink} to evaluate $\sigma_j\bp{\widetilde{m}_{w(v_{\ell+1}),\ldots,w(v_n)}}$.
Let $(G',\omega')$ be the complete graph formed by deleting the vertices $v_1, \dots, v_{\ell}$ from $(G,\omega)$, so $(G',\omega')$ is the complete graph with vertex set $\{v_{\ell+1},\ldots,v_{n}\}$ and vertex set weights $\omega'(v_i) = \omega(v_i)$ for $i \in \{\ell+1,\dots,n\}$. Note that then $w'(v_i)=w(v_i)$ for all $i \in \{\ell+1,\dots,n\}$.

Let $\mathcal{S}'$ be the set of all ordered pairs $(\gamma,S)$ such that $\gamma$ is an acyclic orientation of $G'$, and $S$ is a $\gamma$-admissible one-step weight map with weight sequence $(j)$. Note that $\sink_1(G',\gamma)=1$ for any such $\gamma$ since $G'$ is a complete graph.

Then by applying Theorem \ref{one-level-sink} we have
\begin{align*}
\sigma_{\mu,j}\bp{X_{(G,\omega)}} &= \frac{\prod_{a=1}^{\infty}n_a!}{\prod_{a=1}^{\infty}n'_a!}\cdot 
(-1)^{w(v_{\ell+1})+\cdots+w(v_n)-n+\ell}\sum_{\wts(G',\gamma,S)=(j)}(-1)^{j-1} \\
&= \frac{\prod_{a=1}^{\infty}n_a!}{\prod_{a=1}^{\infty}n'_a!}\cdot 
(-1)^{w(v_{\ell+1})+\cdots+w(v_n)-n+\ell+j-1}\cdot |\mathcal{S}'|.
\end{align*}

Next we will simplify the right-hand side of \eqref{main-thm} and show that it is equal to the above. By Lemma \ref{maximal}, it suffices to consider only the ordered pairs $(\gamma, S)$ in which $\gamma$ is an acyclic orientation of $G$ whose type has the form $(\mu_1,\ldots,\mu_\ell,\ldots)$. Therefore, let $\mathcal{S}$ be the set of all ordered pairs $(\gamma,S)$ such that $\gamma$ is such an acyclic orientation, and $S$ is a $\gamma$-admissible $(\ell+1)$-step weight map with weight sequence $(\mu_1,\ldots,\mu_\ell,j)$.


Then the right-hand side of \eqref{main-thm} becomes
\[(-1)^{d-n}\sum_{(\gamma,S)\in \mathcal{S}}(-1)^{j+\mu_1+\cdots+\mu_\ell-\ell-1}=(-1)^{d-n+j+\mu_1+\cdots+\mu_\ell-\ell-1}\cdot |\mathcal{S}|
,\]
where we have used the fact that $\sum_{i=1}^{\ell+1}\sink_i(\gamma)=\ell+1$ since $G$ is a complete graph.

To construct an element $(\gamma,S)\in \mathcal{S}$, we first pick vertices $v_{k_1},\ldots,v_{k_\ell}$ such that $\Sink_i(\gamma)=\{v_{k_i}\}$ and $w(v_{k_i})=\mu_i$ for all $i \in \{1,\dots,\ell\}$. 
This can be done in $\bp{\prod_{a=1}^{\infty}n_a!}/\bp{\prod_{a=1}^{\infty}n'_a!}$ ways, where $n_a$ and $n'_a$ are defined as above when considering the left-hand side of \eqref{main-thm}. We then must construct $S$ such that for $i \in \{1,\dots,\ell\}$, we have $S(v_{k_i})_i = \omega(v_{k_i})$ and $S(v)_i = \varnothing$ if $v \neq v_{k_i}$.

Then, since we need $\wts(\gamma,S)=(\mu_1,\ldots,\mu_\ell,j)$, to finish constructing $(\gamma, S)$ we require an acyclic orientation $\gamma'$ of $G' = G \bk \{v_{k_1},\ldots,v_{k_\ell}\}$ and a $\gamma'$-admissible one-step weight map on $G'$ with sink weight sequence $(j)$. By definition this may be chosen in $|\mathcal{S}'|$ ways, where $\mathcal{S}'$ is defined above. This shows that 
\[|\mathcal{S}|=\frac{\prod_{a=1}^{\infty}n_a!}{\prod_{a=1}^{\infty}n'_a!}\cdot|\mathcal{S}'|.\]

Therefore, the right-hand side of \eqref{main-thm} becomes
\begin{align*}
  (-1)^{d-n+j+\mu_1+\cdots+\mu_\ell-\ell-1}\cdot\frac{\prod_{a=1}^{\infty}n_a!}{\prod_{a=1}^{\infty}n'_a!}\cdot|\mathcal{S}'| &= (-1)^{d-n+j-\mu_1-\cdots-\mu_\ell+\ell-1}\cdot\frac{\prod_{a=1}^{\infty}n_a!}{\prod_{a=1}^{\infty}n'_a!}\cdot|\mathcal{S}'| \\
  &= (-1)^{w(v_{\ell+1})+\cdots+w(v_n)-n+\ell+j-1}\cdot\frac{\prod_{a=1}^{\infty}n_a!}{\prod_{a=1}^{\infty}n'_a!}\cdot|\mathcal{S}'| \\
  &= \sigma_{\mu,j}\bp{X_{(G,\omega)}},
\end{align*}
which establishes the base case of the theorem.

\subsection*{Inductive Step}

Let $(G,\omega)$ be a set-weighted graph with $g\geq 1$ non-edges, and assume by induction that \eqref{main-thm} holds for all set-weighted graphs with fewer than $g$ non-edges.

Let $e=v_1v_2\notin E(G)$. We may assume $e$ is not a loop, as otherwise both sides of \eqref{main-thm} are equal to $0$, and the result holds. For clarity, we will sometimes suppress explicit mention of $\omega$ in the remainder of this proof, but whenever a graph is mentioned, it is assumed to be a set-weighted graph, and the set-weighting function will be apparent in terms of $\omega$.

Let $G+e$ be the (set-weighted) graph such that $V(G+e)=V(G)$ and $E(G+e)=E(G)\cup\{e\}$, and let $G/e = (G+e)/e$.
We will use previously established terminology, so when $e$ is a nonloop edge, $v^* \in V(G/e)$ is the resulting vertex after contracting $v_1$ and $v_2$.

It is easy to see that both $G+e$ and $G/e$ have fewer than $g$ non-edges. In order to apply the inductive hypothesis, we would like to show that if $\mu$ is a maximal partition of $G$, then $\mu$ is also a maximal partition of both $G+e$ and $G/e$.

\subsubsection*{Checking the Inductive Hypothesis for $G+e$}

We first check that $\mu$ is a maximal partition for $G+e$. Let $\nu$ be an allowable partition in $G+e$. Then there exists some $W\subseteq V(G+e)=V(G)$ and a stable partition $(S_1,\ldots,S_k)$ of $W$ such that $w(S_i)=\nu_i$ for all $i \in \{1, \dots, \ell(\nu)\}$. For each such $i$, since $S_i$ is stable in $G+e$, $S_i$ is also stable in $G$, so $\nu$ is also an allowable partition in $G$. Since $\mu$ is a maximal partition in $G$, $\mu$ partially dominates $\nu$. Since the choice of $\nu$ was arbitrary, we conclude that $\mu$ is maximal in $G+e$.

\subsubsection*{Checking the Inductive Hypothesis for $G/e$}

Next, we check that $\mu$ is also a maximal partition for $G/e$.
Let $\nu$ be an allowable partition in $G/e$.
Then there exists some $W\subseteq V(G/e)$ and a stable partition $(S_1,\ldots,S_k)$ of $W$ such that $w(S_i)=\nu_i$ for all $i \in \{1,\dots,\ell(\nu)\}$.
If $v^*\notin S_i$ for some such $i$, then each $S_i$ is a stable set in the graph $G$, so that $(S_1,\ldots,S_k)$ is a stable partition of a subset of $V(G)$, and in this case, $\nu$ is an allowable partition in $G$. Then since $\mu$ is maximal in $G$, we have that $\mu$ partially dominates $\nu$. 

Suppose that instead $v^*\in S_{i_0}$ for some $i_0 \in \{1,\dots,\ell(\nu)\}$. Let $S'_{i_0}=(S_{i_0}\setminus\{v^*\})\cup\{v_1,v_2\}$ be a set of vertices in $G$.
Since $v_1v_2\notin E(G)$, $S'_{i_0}$ is stable in $G$, and for $i\neq i_0$, it is clear that $S_i$ is stable in $G$. It follows that $((S_1,\dots,S_k) \bk S_{i_0}) \cup S'_{i_0}$ is a stable partition of some subset of vertices of $G$. Since $w(S_{i_0})=w(S'_{i_0})$, it follows that $\nu$ is an allowable partition in $G$, and that $\mu$ partially dominates $\nu$.

Since $\mu$ partially dominates $\nu$ in both cases, it follows that $\mu$ is a maximal partition in $G/e$.

\subsubsection*{Simplification Using the Inductive Hypothesis}

Therefore, we can apply the inductive hypothesis on $G+e$ and $G/e$. 
We have 
\[
\sigma_{\mu,j}\bp{X_{G+e}}=(-1)^{d-n}\sum_{\substack{\wts(G+e,\gamma,S)=(\mu_1,\ldots,\mu_\ell,j) \\ S\text{ admissible}}}(-1)^{j+\sum_{i=1}^{\ell}\mu_i-\sum_{i=1}^{\ell+1}\sink_i(\gamma)}
\]
and
\[
\sigma_{\mu,j}\bp{X_{G/e}}=(-1)^{d-n+1}\sum_{\substack{\wts(G/e,\gamma,S)=(\mu_1,\ldots,\mu_\ell,j) \\ S\text{ admissible}}}(-1)^{j+\sum_{i=1}^{\ell}\mu_i-\sum_{i=1}^{\ell+1}\sink_i(\gamma)}.
\]

By the deletion-contraction relation of Lemma \ref{lem:delconset}, we have 
\[(-1)^{d-n}\sigma_{\mu,j}\bp{X_{G+e}}=(-1)^{d-n}\sigma_{\mu,j}\bp{X_{G}}-(-1)^{d-n}\sigma_{\mu,j}\bp{X_{G/e}}.\]
To prove \eqref{main-thm}, it thus suffices to prove that 
\begin{align*}
  &\quad \sum_{\substack{\wts(G,\gamma,S)=(\mu,j) \\ S\text{ admissible}}}(-1)^{j+\sum_{i=1}^{\ell}\mu_i-\sum_{i=1}^{\ell+1}\sink_i(\gamma)} \\
  &= \sum_{\substack{\wts(G+e,\gamma,S)=(\mu,j) \\ S\text{ admissible}}}(-1)^{j+\sum_{i=1}^{\ell}\mu_i-\sum_{i=1}^{\ell+1}\sink_i(\gamma)}-\sum_{\substack{\wts(G/e,\gamma,S)=(\mu,j) \\ S\text{ admissible}}}(-1)^{j+\sum_{i=1}^{\ell}\mu_i-\sum_{i=1}^{\ell+1}\sink_i(\gamma)}
\end{align*}

or equivalently, it suffices to show that 
\begin{align}
&\quad \sum_{\substack{\wts(G,\gamma,S)=(\mu,j) \\ S\text{ admissible}}}(-1)^{\sum_{i=1}^{\ell+1}\sink_i(\gamma)} \nonumber \\
&=\sum_{\substack{\wts(G+e,\gamma,S)=(\mu,j) \\ S\text{ admissible}}}(-1)^{\sum_{i=1}^{\ell+1}\sink_i(\gamma)}-\sum_{\substack{\wts(G/e,\gamma,S)=(\mu,j) \\ S\text{ admissible}}}(-1)^{\sum_{i=1}^{\ell+1}\sink_i(\gamma)}.
\label{induction}
\end{align}

Given a set-weighted graph $H$ (with weight function suppressed), an orientation $\gamma$ on $H$ (not necessarily acyclic), and an $(\ell+1)$-step weight map $S$ of $H$, we define 
\[
T(H,\gamma,S)=\begin{cases}
(-1)^{\sum_{i=1}^{\ell+1}\sink_i(\gamma)} & \text{ if }\gamma\text{ is acyclic and } S \text{ is }\gamma\text{-admissible} \\
0 & \text{ otherwise.} 
\end{cases}  
\]

Hence, in order to prove the new equality \eqref{induction}, it suffices to show that 
\begin{equation}
\label{induction-goal}
\sum_{\substack{(G,\gamma,S) \\ \wts(G,S)=(\mu,j)}} T(G,\gamma,S)
=\sum_{\substack{(G+e,\gamma,S) \\ \wts(G+e,S)=(\mu,j)}} T(G+e,\gamma,S)
-\sum_{\substack{(G/e,\gamma,S) \\ \wts(G/e,S)=(\mu,j)}} T(G/e,\gamma,S),
\end{equation}
where the sums each range over all (not necessarily acyclic) orientations $\gamma$ of the corresponding graph, and all $(\ell+1)$-step weight maps $S$ of the corresponding graph, and not just the ones which are $\gamma$-admissible. This will allow us to more easily demonstrate the necessary bijections to show the equality holds.

We reiterate a summary of Definition \ref{def:delconori} here, as this notation will be used frequently throughout the remainder of this proof. For $\gamma$ any orientation of $G$ and $S$ an $(\ell+1)$-step weight map of $G$, we defined:
\begin{itemize}
  \item In $G+e$ 
  \begin{itemize}
      \item $\varphi_{v_1}(\gamma)$ is the orientation adding $v_1\to v_2$.
      \item $\varphi_{v_2}(\gamma)$ is the orientation adding $v_2\to v_1$.
      \item $\psi_{+}(S)$ be the $(\ell+1)$-step weight map with $\psi_{+}(S)=S$.
  \end{itemize}   
  \item In $G/e$ \begin{itemize}
      \item $\varphi_{v^*}(\gamma)$ is the orientation contracting $e$.
      \item $\psi_{*}(S)$ is the $(\ell+1)$-step weight map that is the same as $S$ except at $v^*$, where its entries are from the union of those at $v_1$ and $v_2$.
  \end{itemize} 
\end{itemize}

Note that we may easily verify $\wts(G+e,\psi_+(S))=(\mu_1,\ldots,\mu_\ell,j)$ and $\wts(G/e,\psi_*(S))=(\mu_1,\ldots,\mu_\ell,j)$.

We claim that 
\begin{equation}
\label{ind-bijection}
T(G,\gamma,S) = T\bp{G+e,\varphi_{v_1}(\gamma),\psi_+(S)} + T\bp{G+e,\varphi_{v_2}(\gamma),\psi_+(S)} - T\bp{G/e,\varphi_*(\gamma),\psi_*(S)}
\end{equation}
for all orientations $\gamma$ of $G$ and all $(\ell+1)$-step weight maps $S$ on $G$ with $\wts(G,S)=(\mu_1,\ldots,\mu_\ell,j)$. This is sufficient to prove the result, since 
\begin{itemize}
    \item Every orientation of either $G+e$ or $G/e$ corresponds under the inverse of an appropriate $\varphi$ to a unique orientation of $G$.
    \item Every $(\ell+1)$-step weight map of either $G+e$ or $G/e$ corresponds under the inverse of an appropriate $\psi$ to a unique $(\ell+1)$-step weight map of $G$.
\end{itemize}
so the equalities of the form \eqref{ind-bijection} across all orientations $\gamma$ of $G$ and all $(\ell+1)$-step weight maps $S$ on $G$ include every term among the sums in \eqref{induction-goal} exactly once.

\medskip\medskip

\subsubsection*{Proving the Final Equality Holds in All Cases}

For the rest of the proof, we split into cases depending on the nature of $\gamma$, $S$, and the endpoints of the edge $e$.

\subsubsection*{Case 1. $\gamma$ is not an acyclic orientation.}

In this case, each of $\varphi_{v_1}(\gamma)$, $\varphi_{v_2}(\gamma)$, and $\varphi_{v^*}(\gamma)$ are also not acyclic, so every term of \eqref{ind-bijection} is equal to $0$, and equality holds.

For the remainder of the proof, we may assume that $\gamma$ is acyclic.

\subsubsection*{Case 2. $\gamma$ has a directed path from $v_1$ to $v_2$ or from $v_2$ to $v_1$.}

Note that since $\gamma$ is acyclic, there cannot be a directed path from $v_1$ to $v_2$ and a directed path from $v_2$ to $v_1$ at the same time.
Assume without loss of generality that there exists a directed path from $v_1$ to $v_2$ in $\gamma$.

In this case, neither $\varphi_{v_2}(\gamma)$ nor $\varphi_{v^*}(\gamma)$ is acyclic, so 
\[T\bp{G+e,\varphi_{v_2}(\gamma),\psi_+(S)}=0\quad\text{and}\quad T\bp{G/e,\varphi_*(\gamma),\psi_*(S)}=0.\]
We then consider the orientation $\varphi_{v_1}(\gamma)$ on $G+e$. Since there is no directed path from $v_2$ to $v_1$ in $\gamma$, $\varphi_{v_1}(\gamma)$ is acyclic.
We claim that $S$ is $\gamma$-admissible if and only if $\psi_+(S)$ is $\varphi_{v_1}(\gamma)$-admissible.
Indeed, by Lemma \ref{maximal}(c), we have that $S$ is $\gamma$-admissible if and only if $S$ is in $\gamma$-standard form.
Let $v_1\in\Sink_i(\gamma)$ and let $v_2\in\Sink_j(\gamma)$ for some $i$ and $j$. Since there is a directed path from $v_1$ to $v_2$, we must have $i>j$. It follows that $\Sink_m(\gamma)=\Sink_m(\varphi_{v_1}(\gamma))$ for all $m$.
Then since $\psi_+(S)=S$, we have that $S$ is in $\gamma$-standard form if and only if $\psi_+(S)$ is in $\varphi_{v_1}(\gamma)$-standard form, which happens if and only if $\psi_+(S)$ is $\varphi_{v_1}(\gamma)$-admissible in $G+e$ by Lemma \ref{maximal}(c).
This means that for every $S$,
\[T(G,\gamma,S)=T\bp{G+e,\varphi_{v_1}(\gamma),\psi_+(S)}.\]
Therefore, \eqref{ind-bijection} holds in this case.

For the remainder of the proof, we may thus assume that there does not exist a direct path from $v_1$ to $v_2$ or from $v_2$ to $v_1$, and we may therefore assume that $\varphi_{v_1}(\gamma)$, $\varphi_{v_2}(\gamma)$, and $\varphi_{v^*}(\gamma)$ are all acyclic. 

Furthermore, define the indices $i$ and $j$ by letting $v_1 \in \Sink_i(\gamma)$ and $v_2 \in \Sink_j(\gamma)$. The remaining cases will consider different possibilities for $i$ and $j$.

\subsubsection*{Case 3. At least one of $i$ and $j$ is an element of $\{1,\dots,\ell\}$, and $i \neq j$.}

Assume without loss of generality that $i \in \{1,\dots,\ell\}$ (recalling that $\ell = \ell(\mu)$).

We first consider the subcase where $j>i$.
In this subcase, we claim that $\psi_{*}(S)$ is not $\varphi_{v^*}(\gamma)$-admissible in $G/e$.
Note that $v^*\in\Sink_j(\varphi_*(\gamma))$. 
Let $\lambda$ be the type of $\gamma$ and let $\lambda^*$ be the type of $\varphi_*(\gamma)$. 
Recall that $\mu$ partially dominates $\lambda$.

If $\mu_m=\lambda_m$ for all $m \in \{1,\dots,\ell\}$, then we have $\mu_1+\cdots+\mu_i=\lambda_1+\cdots+\lambda_i>\lambda^*_1+\cdots+\lambda^*_i$ by Lemma \ref{obs:sinks}. If instead there exists some $m \in \{1,\dots,\ell\}$ such that $\mu_1+\cdots+\mu_m>\lambda_1+\cdots+\lambda_m$, then we also have $\mu_1+\cdots+\mu_m>\lambda^*_1+\cdots+\lambda^*_m$ by Lemma \ref{obs:sinks}.
Either way, by Lemma \ref{maximal}(b), we conclude that $\psi_*(S)$ cannot be $\varphi_*(\gamma)$-admissible in $G/e$.

We also claim that $\psi_+(S)$ is not $\varphi_{v_1}(\gamma)$-admissible.
Note that $v_1\in\Sink_{j+1}(\varphi_{v_1}(\gamma))$. Therefore, since $\mu$ partially dominates the type of $\varphi_{v_1}(\gamma)$, following the same steps as above using the type of $\varphi_{v_1}(\gamma)$, we conclude that $\psi_{+}(S)$ cannot be $\varphi_{v_1}(\gamma)$-admissible in $G+e$.

Thus we have demonstrated that
\[T\bp{G+e,\varphi_{v_1}(\gamma),\psi_+(S)}=0\quad\text{and}\quad T\bp{G/e,\varphi_*(\gamma),\psi_*(S)}=0.\]

It remains to consider the orientation $\varphi_{v_2}(\gamma)$ on $G+e$.
Proceeding exactly as in Case 2, we can show that $S$ is $\gamma$-admissible if and only if $\psi_+(S)$ is $\varphi_{v_2}(\gamma)$-admissible.
Hence, 
\[T(G,\gamma,S)=T\bp{G+e,\varphi_{v_2}(\gamma),\psi_+(S)},\]
and \eqref{ind-bijection} holds in this case.

For the case where $j<i$, we can use an identical argument to show that $\psi_{*}(S)$ is not $\varphi_{v^*}(\gamma)$-admissible in $G/e$, and $\psi_+(S)$ is not $\varphi_{v_2}(\gamma)$-admissible in $G+e$. It is also easy to show that $S$ is $\gamma$-admissible in $G$ if and only if $\psi_{+}(S)$ is $\varphi_{v_1}(\gamma)$-admissible in $G+e$, so \eqref{ind-bijection} holds.

\subsubsection*{Case 4. $i$ and $j$ are elements of $\{1,\dots,\ell\}$ with $i=j$.}

We claim that $\psi_+(S)$ is not $\varphi_{v_1}(\gamma)$-admissible in $G+e$. Let $\lambda$ be the type of $\gamma$ and let $\lambda^{v_1}$ be the type of $\varphi_{v_1}(\gamma)$.

Then by Lemma \ref{obs:sinks}, $\lambda^{v_1}_1 + \dots + \lambda^{v_1}_i<\lambda_1 + \dots + \lambda_i$ and $\lambda^{v_1}_1+\cdots+\lambda^{v_1}_m\leq \lambda_1+\cdots+\lambda_m$ for all $m$, so by the same argument as in Case 3 we see that $\psi_+(S)$ is not $\varphi_{v_1}(\gamma)$-admissible. By a symmetrical argument, also $\psi_+(S)$ is not $\varphi_{v_2}(\gamma)$-admissible. Therefore, we have 

\[T\bp{G+e,\varphi_{v_1}(\gamma),\psi_+(S)}=0\quad\text{and}\quad T\bp{G+e,\varphi_{v_2}(\gamma),\psi_+(S)}=0.\]
We then claim that $S$ is $\gamma$-admissible in $G$ if and only if $\psi_*(S)$ is $\varphi_*(\gamma)$-admissible in $G/e$.
By Lemma \ref{maximal}(c), $S$ is $\gamma$-admissible if and only if $S$ is in $\gamma$-standard form.
By Lemma \ref{obs:sinks}, $\Sink_i(\varphi_*(\gamma))=(\Sink_i(\gamma)\setminus\{v_1,v_2\})\cup\{v^*\}$, and $\Sink_m(\varphi_*(\gamma))=\Sink_m(\gamma)$ for all $m\neq i$, so we have that $S$ is in $\gamma$-standard form if and only if $\psi_*(S)$ is in $\varphi_{v^*}(\gamma)$-standard form.
Using Lemma \ref{maximal}(c) again, we know that $\psi_*(S)$ is in standard form in $(G/e,\varphi_{v^*}(\gamma))$ if and only if $\psi_*(S)$ is $\varphi_{v^*}(\gamma)$-admissible in $G/e$. Furthermore, when both are admissible we have $(\sum_{k=1}^{\ell+1} \sink_k(\gamma))-1$ = $\sum_{k=1}^{\ell+1} \sink_k(\varphi_*(\gamma))$, so in any case
\[T(G,\gamma,S)=-T\bp{G/e,\varphi_*(\gamma),\psi_*(S)},\]
and thus \eqref{ind-bijection} holds.

From now on, we can assume that both $i$ and $j$ are larger than $\ell$.

\subsubsection*{Case 5. $i,j>\ell$, and it is \textit{not} the case that for all $v \in V(G)$ and $k \in \{1,\dots,\ell\}$, we have $S(v)_k=\{1,\ldots,w(v)\}$ if $v\in\Sink_k(\gamma)$ and $S(v)_k=\varnothing$ if $v\notin\Sink_k(\gamma)$.}

That is, $S$ does not satisfy the first part of the definition for being in $\gamma$-standard form.

By Lemma \ref{maximal}(c), we know that $S$ is not $\gamma$-admissible, so $T(G,\gamma,S)=0$.

Since $i,j>\ell$, neither $v_1$ nor $v_2$ lie on any directed path starting at a vertex $v \in Sink_k(\gamma)$ with $k \leq \ell$. Therefore, for every $m \in \{1,\dots,\ell\}$, $\Sink_m(\gamma)=\Sink_m(\varphi_{v_1}(\gamma))=\Sink_m(\varphi_{v_2}(\gamma))=\Sink_m(\varphi_*(\gamma),\psi_*(S))$. It follows that $\psi_+(S)$ is not in either $\varphi_{v_1}(\gamma)$-standard form or in $\varphi_{v_2}(\gamma)$-standard form, and that $\psi_*(S)$ is not in $\varphi_*(\gamma)$-standard form
.
Hence by Lemma \ref{maximal}(c), $\psi_+(S)$ is neither $\varphi_{v_1}(\gamma)$-admissible nor $\varphi_{v_2}(\gamma)$-admissible, and $\psi_*(S)$ is not $\varphi_*(\gamma)$-admissible, so $T\bp{G+e,\varphi_{v_1}(\gamma),\psi_+(S)}=T\bp{G+e,\varphi_{v_2}(\gamma),\psi_+(S)}=T\bp{G/e,\varphi_{v^*}(\gamma),\psi_*(S)}=0$, and \eqref{ind-bijection} holds.
\medskip

For all remaining cases, we may thus assume that for all $k \in \{1,\dots,\ell\}$, $S(v)_k=\{1,\ldots,w(v)\}$ if $v\in\Sink_k(\gamma)$ and $S(v)_k=\varnothing$ if $v\notin\Sink_k(\gamma)$, so $S$ satisfies the first portion of the definition for being in $\gamma$-standard form.

\subsubsection*{Case 6. $i=j=\ell+1$.}

That is, $v_1$ and $v_2$ are both elements of $\Sink_{\ell+1}(\gamma)$. If $S(v_1)_{\ell+1}=S(v_2)_{\ell+1}=\varnothing$, then using Lemma \ref{obs:sinks} it is straightforward to check that all terms in \eqref{ind-bijection} will be $0$.
Thus, we assume that at least one of $S(v_1)_{\ell+1}$ and $S(v_2)_{\ell+1}$ is non-empty.
We proceed with subcases.

\subsubsection*{Case 6.1. Exactly one of $S(v_1)_{\ell+1}$ and $S(v_2)_{\ell+1}$ is non-empty.}

Suppose first that $S(v_1)_{\ell+1}\neq \varnothing$ and $S(v_2)_{\ell+1}=\varnothing$.
Then $S$ is not in $\gamma$-standard form, so $T(G,\gamma,S)=0$.
Also, $\psi_{+}(S)$ is not in $\varphi_{v_1}(\gamma)$-standard form, so $T\bp{G+e,\varphi_{v_1}(\gamma),\psi_+(S)}=0$.

Now, if for all $v \in V(G) \setminus \{v_1,v_2\}$, we have $v\in\Sink_{\ell+1}(\gamma)$ if and only if $S(v)_{\ell+1}\neq \varnothing$, then $\psi_+(S)$ is in $\varphi_{v_2}(\gamma)$-standard form and $\psi_*(S)$ is in $\varphi_*(\gamma)$-standard form, so $\psi_+(S)$ is $\varphi_{v_2}(\gamma)$-admissible and $\psi_*(S)$ is $\varphi_{v^*}(\gamma)$-admissible.
Otherwise, $\psi_+(S)$ is not in standard form in $\varphi_{v_2}(\gamma)$, and $\psi_*(S)$ is not in standard form in $\varphi_*(\gamma)$, and it follows that $\psi_+(S)$ is not $\varphi_{v_2}(\gamma)$-admissible and $\psi_*(S)$ is not $\varphi_{v^*}(\gamma)$-admissible.
Either way, we have $T\bp{G+e,\varphi_{v_2}(\gamma),\psi_+(S)}=T\bp{G/e,\varphi_{v^*}(\gamma),\psi_*(S)}$, and \eqref{ind-bijection} holds.

A symmetrical argument shows that if $S(v_1)_{\ell+1}=\varnothing$ and $S(v_2)_{\ell+1}\neq\varnothing$, then \eqref{ind-bijection} also holds.

\subsubsection*{Case 6.2. Both $S(v_1)_{\ell+1}$ and $S(v_2)_{\ell+1}$ are non-empty.}

In this case $\psi_+(S)$ is not in either $\varphi_{v_1}(\gamma)$-standard form in  or $\varphi_{v_2}(\gamma)$-standard form, so $T\bp{G+e,\varphi_{v_1}(\gamma),\psi_+(S)}=T\bp{G+e,\varphi_{v_2}(\gamma),\psi_+(S)}=0$.

If for all $v \in V(G)$, we have $S(v)_{\ell+1}\neq \varnothing$ if and only if $v\in\Sink_{\ell+1}(\gamma)$, then $S$ is in $\gamma$-standard form and thus is $\gamma$-admissible, and also $\psi_*(S)$ is $\varphi_*(\gamma)$-admissible.
Since when both are admissible $(\sum_{k=1}^{\ell+1} \sink_k(\gamma))-1$ = $\sum_{k=1}^{\ell+1} \sink_k(\varphi_*(\gamma))$, we have $T(G,\gamma,S)=-T\bp{G/e,\varphi_*(\gamma),\psi_*(S)}$.

Otherwise, $S$ is not in $\gamma$-standard form and $\psi_*(S)$ is not in $\varphi_*(\gamma)$-standard form, so $S$ is not $\gamma$-admissible and $\psi_*(S)$ is not $\varphi_*(\gamma)$-admissible. Thus,
$T(G,\gamma,S)=T\bp{G/e,\varphi_*(\gamma),\psi_*(S)}=0$. 

In either case, \eqref{ind-bijection} holds.

\subsubsection*{Case 7. $\{i,j\} = \{\ell+1, k\}$ with $k > \ell+1$.}

We assume without loss of generality that $i = \ell+1$ and $j > \ell+1$.
First note that if $S(v_2)_{\ell+1}\neq \varnothing$, using Lemma \ref{obs:sinks} it is straightforward to verify that $S$ is not in $\gamma$-standard form, $\psi_+(S)$ is not in $\varphi_{v_1}(\gamma)$-standard form or in $\varphi_{v_2}(\gamma)$-standard form, and $\psi_*(S)$ is not in $\varphi_*(\gamma)$-standard form, so by Lemma \ref{maximal}(c) all terms in \eqref{ind-bijection} are $0$.

We may thus assume that $S(v_2)_{\ell+1}=\varnothing$, and divide into subcases.

\subsubsection*{Case 7.1. $S(v_1)_{\ell+1}=\varnothing$.}

Then $S$ is not in $\gamma$-standard form, so $T(G,\gamma,S)=0$.
We also have that $\psi_+(S)$ is not in $\varphi_{v_2}(\gamma))$-standard form, so $T\bp{G+e,\varphi_{v_2}(\gamma),\psi_+(S)}=0$.

Suppose first that for all $v \in V(G) \bk \{v_1\}$, we have $S(v)_{\ell+1}\neq \varnothing$ if and only if $v\in\Sink_{\ell+1}(\gamma)$. For $w \in \Sink_k(\gamma)$ with $k \leq \ell+1$, no directed path in $\gamma$ starting at $w$ contains $v_1$ or $v_2$, so $w \in \Sink_{k}(\varphi_*(\gamma))$ and $w \in \Sink_k(\varphi_{v_1}(\gamma))$. If $w \in \Sink_{k}(\gamma)$ with $k > \ell+1$, then by Lemma \ref{obs:sinks}, $w \in \Sink_{k^{*}}(\varphi_*(\gamma))$ and $w \in \Sink_{k^{v_1}}(\varphi_{v_1}(\gamma))$ for $k^{*}, k^{v_1} \geq k$. It follows that $\psi_*(S)$ is in $\varphi_*(\gamma)$-standard form, so $\psi_*(S)$ is $\varphi_*(\gamma)$-admissible.
Likewise, $\psi_+(S)$ is $\varphi_{v_1}(\gamma)$-admissible. Hence, $T\bp{G/e,\varphi_*(\gamma),\psi_*(S)}=T\bp{G+e,\varphi_{v_1}(\gamma),\psi_+(S)}$, and \eqref{ind-bijection} holds.

Otherwise, using a similar argument we see that $\psi_*(S)$ is not $\varphi_*(\gamma)$-admissible and $\psi_+(S)$ is not $\varphi_{v_1}(\gamma)$-admissible, so all terms in \eqref{ind-bijection} are $0$, and \eqref{ind-bijection} holds.

\subsubsection*{Case 7.2. $S(v_1)_{\ell+1}\neq \varnothing$.}

We see that $\psi_+(S)$ is not in $\varphi_{v_1}(\gamma)$-standard form, so $\psi_+(S)$ is not $\varphi_{v_1}(\gamma)$-admissible, and $T\bp{G+e,\varphi_{v_1}(\gamma),\psi_+(S)}=0$.
Additionally, $\psi_*(S)$ is not in $\varphi_*(\gamma)$-standard form, so $T\bp{G/e,\varphi_*(\gamma),\psi_*(S)}=0$.

Using an argument analogous to that in Case 7.1,
if for all $v \in V(G)$ we have $S(v)_{\ell+1}\neq \varnothing$ if and only if $v\in\Sink_{\ell+1}(\gamma)$, then $S$ is $\gamma$-admissible and $\psi_+(S)$ is $\varphi_{v_2}(\gamma)$-admissible.
Otherwise, $S$ is not $\gamma$-admissible and $\psi_+(S)$ is not $\varphi_{v_2}(\gamma)$-admissible.
Either way $T(G,\gamma,S)=T\bp{G+e,\varphi_{v_2}(\gamma),\psi_+(S)}$, and \eqref{ind-bijection} holds.

\subsubsection*{Case 8. $i > \ell+1$ and $j > \ell+1$.}

It is straightforward to check using Lemma \ref{obs:sinks} that if one of $S(v_1)_{\ell+1}$ and $S(v_2)_{\ell+1}$ is non-empty, then all terms in \eqref{ind-bijection} will be $0$.
We thus assume $S(v_1)_{\ell+1}=S(v_2)_{\ell+1}=\varnothing$, and so $S$ is $\gamma$-admissible.

If for all $v \in V(G) \setminus \{v_1,v_2\}$ we have $S(v)_{\ell+1}\neq \varnothing$ if and only if $v\in\Sink_{\ell+1}(\gamma)$, then using an argument analogous to that in Case 7.1, it follows that $\psi_+(S)$ is both $\varphi_{v_1}(\gamma)$-admissible and $\varphi_{v_2}(\gamma)$-admissible, and $\psi_*(S)$ is $\varphi_*(\gamma)$-admissible.
Since then $\sum_{k=1}^{\ell+1} \sink_k(\gamma)$ =$\sum_{k=1}^{\ell+1} \sink_k(\varphi_{v_1}(\gamma))$ = $\sum_{k=1}^{\ell+1} \sink_k(\varphi_{v_2}(\gamma))$ =  $1+\sum_{k=1}^{\ell+1} \sink_k(\varphi_*(\gamma))$, we have $$T(G,\gamma,S)=T\bp{G+e,\varphi_{v_1}(\gamma),\psi_+(S)}+T\bp{G+e,\varphi_{v_2}(\gamma),\psi_+(S)}-T\bp{G/e,\varphi_{v^*}(\gamma),\psi_*(S)}$$ and  \eqref{ind-bijection} holds.

Otherwise, we may check that all terms in \eqref{ind-bijection} are $0$, so \eqref{ind-bijection} holds.

\medskip
\medskip

It is easy to check that Cases 1 through 8 include all possibilities. Thus \eqref{ind-bijection} holds for all orientations $\gamma$ on $G$ and all $(\ell+1)$-step weight maps $S$ of $G$, and we are done. 
\end{proof}

\section{A Conjectured Strengthening of Theorem \ref{thm:stanley3.4}}

In this section, we conjecture a stronger version of Theorem \ref{thm:stanley3.4} for the case when $\mu$ has one part, with the expectation that a proof of this strengthening would likely extend to prove a corresponding theorem for any $\mu$. 

Due to the strength of Lemma \ref{maximal} (c), in the previous section any time we deleted mini-vertices from a vertex, we could assume that we always deleted all mini-vertices except possibly in the last step. In general, we would like to be able to consider partial deletions from a set-weighted vertex throughout the process, potentially allowing for a more general construction and theorem.

In this section we first introduce the conjecture. We then provide supporting numerical evidence that also serves to illustrate the main ideas behind the conjecture. In later sections, we demonstrate the particular relevance of the conjecture to unweighted claw-free graphs, and provide supporting theoretical evidence in the form of proofs of nontrivial special cases.

\subsection{Introducing the Conjecture}

First, we begin with some definitions that extend previously given ones (and provide examples shortly for illustration).

\begin{definition}
Let $\ell\in\N$ and let $(G,\omega)$ be a set-weighted graph. A \textbf{generalized $\ell$-step weight map} of $(G,\omega)$ is a function $S:\mathcal{P}(V(G))\to (\mathcal{P}(\N))^{\ell}$ such that for all nonempty $A \subseteq V(G)$, we have 
\[\bigsqcup_{i=1}^{\ell}S(A)_i\subseteq \omega(A)\]
where $\omega(A) = \bigsqcup_{v \in A} \omega(v)$.

We define the \textbf{generalized $\ell$-step weight sequence} of a generalized $\ell$-step weight map 
$S$ to be $\wts(G,S)=(|S_1|,\ldots,|S_\ell|)$, where for all $i \in \{1,\dots,l\}$, we have
\[|S_i|=\sum_{A\subseteq V(G)}|S(A)_i|.\]

\end{definition}

Thus, now instead of only each vertex having its own map, every possible set of vertices may have its own map. Our conjecture is for $\mu$ with one part, so we will specifically be looking at the interplay between acyclic orientations and generalized $2$-step weight maps.

As a concrete example, consider the graph $(G,\omega)$ given in Figure \ref{fig:generalized-map-example}, and the function $S:\mathcal{P}(V(G))\to(\mathcal{P}(\N))^2$ given by 
\begin{gather*}
  S(\{v_1\})=\bp{\{1,2\},\varnothing}, S(\{v_2\})=\bp{\{3\},\varnothing}, S(\{v_3\})=\bp{\{4,7\},\{6,8\}} \\
  S(\{v_4,v_5\})=\bp{\varnothing,\{9,11\}} \\
  S(A)=(\varnothing,\varnothing)\text{ for all other }A\subseteq V(G).
\end{gather*}
Then we can check that $S$ is a \textit{generalized 2-step weight map} of $(G,\omega)$.
The \textit{generalized 2-step weight sequence} of $S$ is $(5, 4)$.

\begin{figure}[hbt]
\begin{center}
\tikzset{every picture/.style={line width=0.75pt}} 

\begin{tikzpicture}[x=0.75pt,y=0.75pt,yscale=-1,xscale=1]

\draw  [fill={rgb, 255:red, 0; green, 0; blue, 0 }  ,fill opacity=1 ] (85.36,24.4) .. controls (85.36,22.22) and (87.04,20.45) .. (89.1,20.45) .. controls (91.17,20.45) and (92.84,22.22) .. (92.84,24.4) .. controls (92.84,26.58) and (91.17,28.35) .. (89.1,28.35) .. controls (87.04,28.35) and (85.36,26.58) .. (85.36,24.4) -- cycle ;
\draw  [fill={rgb, 255:red, 0; green, 0; blue, 0 }  ,fill opacity=1 ] (210.58,24.4) .. controls (210.58,22.22) and (212.25,20.45) .. (214.31,20.45) .. controls (216.38,20.45) and (218.05,22.22) .. (218.05,24.4) .. controls (218.05,26.58) and (216.38,28.35) .. (214.31,28.35) .. controls (212.25,28.35) and (210.58,26.58) .. (210.58,24.4) -- cycle ;
\draw  [fill={rgb, 255:red, 0; green, 0; blue, 0 }  ,fill opacity=1 ] (147.97,24.4) .. controls (147.97,22.22) and (149.64,20.45) .. (151.71,20.45) .. controls (153.77,20.45) and (155.45,22.22) .. (155.45,24.4) .. controls (155.45,26.58) and (153.77,28.35) .. (151.71,28.35) .. controls (149.64,28.35) and (147.97,26.58) .. (147.97,24.4) -- cycle ;
\draw  [fill={rgb, 255:red, 0; green, 0; blue, 0 }  ,fill opacity=1 ] (54.76,80.05) .. controls (54.76,77.87) and (56.43,76.1) .. (58.5,76.1) .. controls (60.56,76.1) and (62.23,77.87) .. (62.23,80.05) .. controls (62.23,82.23) and (60.56,84) .. (58.5,84) .. controls (56.43,84) and (54.76,82.23) .. (54.76,80.05) -- cycle ;
\draw  [fill={rgb, 255:red, 0; green, 0; blue, 0 }  ,fill opacity=1 ] (179.97,80.05) .. controls (179.97,77.87) and (181.64,76.1) .. (183.71,76.1) .. controls (185.77,76.1) and (187.44,77.87) .. (187.44,80.05) .. controls (187.44,82.23) and (185.77,84) .. (183.71,84) .. controls (181.64,84) and (179.97,82.23) .. (179.97,80.05) -- cycle ;
\draw  [fill={rgb, 255:red, 0; green, 0; blue, 0 }  ,fill opacity=1 ] (117.36,80.05) .. controls (117.36,77.87) and (119.04,76.1) .. (121.1,76.1) .. controls (123.17,76.1) and (124.84,77.87) .. (124.84,80.05) .. controls (124.84,82.23) and (123.17,84) .. (121.1,84) .. controls (119.04,84) and (117.36,82.23) .. (117.36,80.05) -- cycle ;
\draw    (58.57,79.65) -- (90.57,24) ;
\draw    (121.1,80.05) -- (89.1,24.4) ;
\draw    (122.57,79.65) -- (151.79,24) ;
\draw    (183.78,79.65) -- (151.79,24) ;
\draw    (183.78,79.65) -- (215.78,24) ;

\draw (57,2.4) node [anchor=north west][inner sep=0.75pt]  [font=\small]  {$v_{1} ,\{1,2\}$};
\draw (130,2.4) node [anchor=north west][inner sep=0.75pt]  [font=\small]  {$v_{2} ,\{3\}$};
\draw (177,2.4) node [anchor=north west][inner sep=0.75pt]  [font=\small]  {$v_{3} ,\{4,5,6,7,8\}$};
\draw (2,88.4) node [anchor=north west][inner sep=0.75pt]  [font=\small]  {$v_{4} ,\{9,10\}$};
\draw (72,88.4) node [anchor=north west][inner sep=0.75pt]  [font=\small]  {$v_{5} ,\{11,12,13\}$};
\draw (176,88.4) node [anchor=north west][inner sep=0.75pt]  [font=\small]  {$v_{6} ,\{14\}$};

\end{tikzpicture}

\end{center}
\captionsetup{skip=2pt}
\caption{Example graph $(G,\omega)$ for the generalized weight map.}
\label{fig:generalized-map-example}
\end{figure}

\begin{definition}\label{def:drop}
Let $(G,\omega)$ be a set-weighted graph, and let $\gamma$ be an acyclic orientation of $G$.

\medskip

For each vertex $v \in V(G)$, let $\omega(v)_j$ be the $j^{th}$-smallest element of $\omega(v)$ for $j \in \{1,\dots,w(v)\}$. Define $C_v$ to be the oriented (cyclic) graph with vertex set $\omega(v)$ and directed edge set \[ \bigsqcup_{j=1}^{w(v)-1} \big\{\omega(v)_j\omega(v)_{j+1}\big\} \sqcup \big\{\omega(v)_{w(v)}\omega(v)_1\big\}.\]

\medskip

A generalized $2$-step weight map $S$ of $(G,\omega)$ is \textbf{$\gamma$-admissible} if
\begin{itemize}
  \item[(1)] Across all nonempty $A \subseteq V(G)$, $S(A)_1\neq \varnothing$ if and only if $A=\{v\}$ for some $v\in\Sink_1(\gamma)$. 
  
  In this case, let $D=\big\{v\in \Sink_1(\gamma):S(\{v\})_1=\omega(v)\big\}$ (the set of sinks of $\gamma$ which are annihilated by $S$).
  
  Let $\Sink_2'(\gamma,S)=\Sink_1(\gamma|_{V(G)-D})\setminus\Sink_1(\gamma)$. Note that $\Sink_2'(\gamma,S)\subseteq\Sink_2(\gamma)$ is the set of second-level sinks of $\gamma$ that are ``uncovered" by the removal of the vertices of $D$.
  
  Let $M_1, \dots, M_k$ be the connected components of $G[D \cup \Sink_2'(\gamma,S)]$.

  \item[(2)] For all $v\in\Sink_1(\gamma)\setminus D$, $S(\{v\})_2$ is the set of sinks in $C_v-S(\{v\})_1$. 
  
  \item[(3)] For each $i \in \{1,\dots,k\}$, there is exactly one nonempty $B_i\subseteq(\Sink_2'(\gamma,S) \cap M_i)$ such that $S(B_i)_2\neq \varnothing$. Define \[D_{B_i}=\{v\in (D \cap M_i):v\text{ is adjacent with some vertex in }B_i\}.\] 
  (Intuitively, $D_{B_i}$ are the annihilated sinks that caused $B_i$ to become sinks.)
  Let $T_i = \omega(B_i)$ if $w(B_i) \leq w(D_{B_i})$, and otherwise let $T_i$ consist of the $w(D_{B_i})$ smallest elements of $\omega(B_i)$. Then $$\varnothing \subsetneq S(B_i)_2\subseteq T_i.$$
  \item[(4)] For all remaining $A\subseteq V(G)$ such that $S(A)_2$ has not yet been defined, $S(A)_2=\varnothing$.
\end{itemize}
\end{definition}

\begin{figure}[hbt]
\begin{center}
\tikzset{every picture/.style={line width=0.75pt}} 

\end{center}
\captionsetup{skip=2pt}
\caption{Illustration of Definition \ref{def:drop}. Elements selected by the weight map are colored in red.}
\label{fig:generalized-admissibility}
\end{figure}

As an example of applying Definition \ref{def:drop}, we show that the generalized 2-step map $S$ given for the graph $(G,\omega)$ in Figure \ref{fig:generalized-map-example} is $\gamma$-admissible, where $\gamma$ is given by $v_4\to v_1$, $v_5\to v_1$, $v_5\to v_2$, $v_6\to v_2$, and $v_6\to v_3$.
Recall that we defined $S$ as 
\begin{gather*}
  S(\{v_1\})=\bp{\{1,2\},\varnothing}, S(\{v_2\})=\bp{\{3\},\varnothing}, S(\{v_3\})=\bp{\{4,7\},\{6,8\}} \\
  S(\{v_4,v_5\})=\bp{\varnothing,\{9,11\}} \\
  S(A)=(\varnothing,\varnothing)\text{ for all other }A\subseteq V(G).
\end{gather*}
The setup of applying $S$ to $(G,\omega,\gamma)$ is illustrated in Figure \ref{fig:generalized-admissibility}, with the elements selected by $S$ colored in red.
We will go through the conditions in Definition \ref{def:drop} one by one to verify that $S$ is $\gamma$-admissible:
\begin{itemize}
  \item[(1)] We have $\Sink_1(\gamma)=\{v_1,v_2,v_3\}$, and we see that $S(A)_1\neq \varnothing$ if and only if $A=\{v\}$ for some $v\in\{v_1,v_2,v_3\}$. Then we note that $D=\{v_1,v_2\}$ and $\Sink_2'(\gamma,S)=\{v_4,v_5\}$. Finally, there is only one connected component $M_1$ in $G[D\cup\Sink_2'(\gamma,S)]=G[\{v_1,v_2,v_4,v_5\}]$.
  \item[(2)] We then note that $\Sink_1(\gamma)\setminus D=\{v_3\}$. The cycle graph $C_{v_3}=C_5$ is illustrated in Figure \ref{fig:generalized-admissibility}. We see that the set of sinks in $C_{v_3}-\{4,7\}$ is $\{6,8\}$, which coincides with $S(\{v_3\})_2$.
  \item[(3)] Then we observe that $B_1=\{v_4,v_5\}$ is the only subset of $\Sink_2'(\gamma,S)\cap M_1=\{v_4,v_5\}$ such that $S(B_1)_2\neq \varnothing$. We have $D_{B_1}=\{v_1,v_2\}$ and $w(D_{B_1})=3<w(B_1)=5$. We then perform a \textit{weight-drop}, so $T_1=\{9,10,11\}$. Finally, we check that it is indeed the case that $\varnothing\subsetneq S(B_1)_2\subseteq T_1$.
  \item[(4)] Finally, it is indeed the case that $S(A)_2=\varnothing$ for all remaining $A\subseteq V(G)$.  
\end{itemize}
Therefore, $S$ is $\gamma$-admissible.

To summarize, there are two substantial additions made here to the definition of $\gamma$-admissibility. 

First, each weighted vertex is effectively replaced by a directed cycle for the purpose of determining second-level sinks ``within" a weighted vertex; that this is the correct way to do so is implied by the proof of Theorem \ref{no-edge-thm} in Section 4.3, a special case of our conjecture.

Second, when vertices are entirely removed as first-level sinks, we give a new process for choosing second-level sinks from among the newly uncovered vertices. Not only must we choose a subset of the revealed vertices, but we may see a \emph{weight-drop} phenomenon where if the weight of this subset is greater than the weight of its annihilated neighbors, we must drop the weight permitted for second-level sinks to match the smaller value. The method of choosing a subset of revealed vertices is suggested by numerical data and may be easily seen to agree with Theorem \ref{thm:main} where both apply. That the weight-drop phenomenon is necessary is implied by the proof of Theorem \ref{one-edge-thm} in Section 4.3, another special case of our conjecture.

For the original definition of maximal $\mu$ with respect to $(G,\omega)$ used in Theorem \ref{thm:main}, we did not require this complexity. However, our stronger conjecture would allow for a broader range of viable $\mu$ in the case where $\mu$ has one part (so is an integer).

\begin{definition}\label{def:sallow}

Given $U,T$ disjoint subsets of the vertex set of a set-weighted graph $(G,\omega)$, let $M_1, \dots, M_k$ be the connected components of $G[U \cup T]$. For $i \in \{1,\dots,k\}$, let $U_i = U \cap V(M_i)$ and $T_i = T \cap V(M_i)$.

Let $\mu$ be a positive integer. We say $\mu$ is \textbf{s-allowable} in $(G,\omega)$ if for all disjoint stable sets $U, T \subseteq V(G)$ such that 
\begin{itemize}
    \item For all $v\in U$ there exists some $u\in T$ with $uv\in E(G)$, and
    \item For all $u\in T$ there exists some $v\in U$ with $uv\in E(G)$,
\end{itemize} 
it holds that for any choice of positive integers $\mu_1,\dots,\mu_k$ such that $\mu_1+\dots+\mu_k = \mu$, we have that for all $i \in \{1,\dots,k\}$,
$\mu_i\leq\min\{w(U_i),w(T_i)\}$ or $\mu_i\geq\max\{w(U_i),w(T_i)\}$.

\end{definition}

Note that under previous definitions, a single integer $\mu$ is maximal if and only if $\mu$ is greater than or equal to the weight of the largest stable set of $(G,\omega)$. $s$-allowability is a much more flexible condition; for instance, $\mu$ smaller than the size of the smallest vertex weight of $(G,\omega)$ is always $s$-allowable, and in particular $\mu=1$ is always $s$-allowable. On the other end, it is possible for $\mu$ smaller than the weight of the largest stable set of $(G,\omega)$ to be $s$-allowable if this largest stable set does not meet the criteria above with respect to some other disjoint stable set of $(G,\omega)$.

We require one more piece, which is a generalization of the sign of a pair $(\gamma, S)$.

\begin{definition}
Let $(G,\omega)$ be a set-weighted graph. Let $\gamma$ be an acyclic orientation on $G$, and let $S$ be a $\gamma$-admissible $2$-step weight map on $G$. Let \[\mathcal{I}=\{A\subseteq V(G):\text{at least one of }S(A)_i\text{ is non-empty for }i\in\{1,2\}\}.\]
For each $A\in I$, let $i_A$ be the smallest index $i$ such that $S(A)_{i}\neq \varnothing$.
We define the \textbf{sign} of $(\gamma,S)$ to be 
\[\sgn(\gamma,S):= (-1)^{\sum_{A\in \mathcal{I}}\left|S(A)_{i_A}\right|-|A|}.\]
\end{definition}

Thus, the sign generated by a given $A$ corresponds to the parity of the number of unused mini-vertices of $A$. It is straightforward to verify that this definition of the sign of $(\gamma,S)$ agrees with that in Theorem \ref{thm:main} in the cases where $A$ is a single vertex. We now state our main conjecture.

\begin{conj}
\label{conjecture}
Let $(G,\omega)$ be a set-weighted graph with $n$ vertices and total weight $d$.
Write $X_{(G,\omega)}=\sum_{\lambda\vdash d}c_{\lambda}e_{\lambda}$.
Let $\mu \leq d$ be an integer (viewed as a partition with a single part) that is $s$-allowable in $(G,\omega)$.
Fix $j \in \{0, \dots, d-\mu\}$, where we can have $j=0$ only when $\mu=d$. Then 
\begin{equation}
\label{conjecture-equation}
\sigma_{\mu,j}\bp{X_{(G,\omega)}}=(-1)^{d-n}\sum_{\substack{\wts(\gamma,S)=(\mu,j) \\ S\text{ admissible}}}\sgn(\gamma,S),
\end{equation}
summed over all acyclic orientations $\gamma$ of $G$ and all $\gamma$-admissible generalized $2$-step weight maps $S$ of $G$ such that $\wts(\gamma,S)=(\mu,j)$.
\end{conj}

The main generalization that Conjecture \ref{conjecture} makes from Theorem \ref{thm:main} is that in Theorem \ref{thm:main}, we always delete the whole vertex except possibly in the last step. But in Conjecture \ref{conjecture}, we allow partial deletions of mini-vertices in the first step.
We now provide a concrete example to illustrate clearly what Conjecture \ref{conjecture} claims.

Consider $P_{5,7,5}$, a set-weighted three-vertex path with vertices $v_1,v_2,v_3$, edges $v_1v_2$ and $v_2v_3$, $\omega(v_1) = \{1,\dots,5\}$, $\omega(v_2) = \{6,\dots,12\}$, and $\omega(v_3) = \{13,\dots,17\}$. It is easy to verify using the $p$-basis expansion of vertex-weighted chromatic symmetric functions \cite[Lemma 3]{delcon} that $X_{(P_{5,7,5},\, \omega)} = p_{755}-2p_{(12)5}+p_{(17)}$ (where integers of more than one digit are enclosed in parentheses for clarity). Using SageMath to convert this to the $e$-basis, we can determine $\sigma_{\mu,j}(X_{(P_{5,7,5},\, \omega)})$ for any desired integers $\mu \geq 1$ and $j \geq 0$.

\begin{figure}[hbt]
\begin{center}
    \begin{tikzpicture}[scale=1.5]
        \node[label=above:{$v_1, 5$}, fill=black, circle] at (0, 0)(1){};
        \node[label=above:{$v_2, 7$}, fill=black, circle] at (1, 0)(2){};
        \node[label=above:{$v_3, 5$}, fill=black, circle] at (2, 0)(3){};
    
        \draw[black, thick] (2) -- (1);
        \draw[black, thick] (3) -- (2);
    
    \end{tikzpicture}
\end{center}
\caption{The graph $P_{5,7,5}$ with vertex weights given.}
\label{fig:p232}
\end{figure}

For example, we may compute that $\sigma_{7,3}(X_{(P_{5,7,5},\, \omega)}) = \sigma_{7,3}(20e_{6431111}+20e_{6521111}+70e_{7331111}+140e_{7421111}-210e_{8321111}-105e_{9221111}) = -65$.

Note that the only pairs $(U,T)$ of stable sets satisfying the conditions outlined in Definition \ref{def:sallow} on $s$-allowability are $\{U,T\} = \{\{v_i\},\{v_2\}\}$ for $i \in \{1,3\}$ and $\{U,T\} = \{\{v_1,v_3\},\{v_2\}\}$, and all of these cases result in $G[U \sqcup T]$ connected. In the former cases $\{|U|,|T|\} = \{5,7\}$ and in the latter case $\{|U|,|T|\} = \{7,10\}$, so a positive integer $\mu$ is $s$-allowable for $P_{5,7,5}$ if and only if $\mu \notin \{6,8,9\}$. In particular, $\mu = 7$ is $s$-allowable.

Now, we determine what Conjecture \ref{conjecture} predicts for the value of $\sigma_{7,3}(X_{(P_{5,7,5},\, \omega)})$. Note that this graph has weight $17$ and $3$ vertices, so our outer sign is $(-1)^{17-3} = 1$, so Conjecture \ref{conjecture} predicts that
\begin{equation}\label{eq:sigma63}
 \sum_{\substack{\wts(\gamma,S)=(3,2) \\ S\text{ admissible}}}\sgn(\gamma,S) = \sigma_{7,3}(X_{(P_{5,7,5},\, \omega)}) = -65.
\end{equation}

There are four acyclic orientations of $P_{5,7,5}$; three have unique sinks, and one has two sinks. The two orientations $\gamma_1$ and $\gamma_3$ with unique sink $v_1$ and $v_3$ do not admit any admissible weight map $S$, as $w(\Sink_1(\gamma_1)) = w(\Sink_1(\gamma_3)) = 5$, and we require $\sum_{v \in V(P_{5,7,5})} |S(\{v\})_1| = 7$.

Of the other orientations, let us first consider the orientation $\gamma_{1,3}$ in which $v_1$ and $v_3$ are both sinks. Then we must have $|S(\{v_1\})_1 \cup S(\{v_3\})_1| = 7$. We consider all possible ways this can occur:
\begin{itemize}
    \item $|S(\{v_1\})_1| = 2$ and $|S(\{v_3\})_1| = 5$. Then the only $A \subseteq V(P_{5,7,5})$ that may have $S(A)_2 \neq \varnothing$ is $A = \{v_1\}$. Furthermore, $S(A)_2$ must be a subset of the sinks of a directed five-vertex cycle with two vertices deleted. However, there are at most two sinks in such a graph, contradicting that we require $|S_2| = 3$. Thus, in this case no valid $S$ is possible.
    \item $|S(\{v_1\})_1| = 5$ and $|S(\{v_3\})_1| = 2$. This is identical to the above case.
    \item $|S(\{v_1\})_1| = 3$ and $|S(\{v_3\})_1| = 4$. Then we require $S(\{v_3\})_2 = \omega(v_3) \setminus S(\{v_3\})_1$ and $|S(\{v_1\})_2| = 2$. Furthermore, if $C_{v_1}$ is the directed cycle with vertex set $\{1,2,3,4,5\}$ and edge set $\{1 \rightarrow 2, 2 \rightarrow 3, 3 \rightarrow 4, 4 \rightarrow 5, 5 \rightarrow 1\}$, $S(\{v_1\})_2$ must be a subset of the sinks of $C_{v_1} \setminus S(\{v_1\})_1$. We split into subcases based on $S(\{v_1\})_1$.
    \begin{itemize}
        \item $S(\{v_1\})_1$ consists of three consecutive vertices of $C_{v_1}$. Then there is only one sink in $C_{v_1} \setminus S(\{v_1\})_1$, so no valid $S$ is possible.
        \item In all other cases, there are exactly two sinks of $C_{v_1} \setminus S(\{v_1\})_1$, so $S_2(v_1)$ is uniquely determined.
    \end{itemize}
   There are $\binom{5}{4} = 5$ ways to choose $S(\{v_3\})_1$ and $S(\{v_3\})_2$, and it is straightforward to verify that there are $5$ valid choices of $S(\{v_1\})_1$, giving $25$ valid choices of $S$. One example is illustrated in Figure \ref{fig:conjecture-example}. Furthermore, note that all valid $S$ have sign $(-1)^{3-1} \cdot (-1)^{4-1} = -1$, so this subcase contributes $-25$ to the left-hand side of \eqref{eq:sigma63}.
   \item $|S(\{v_1\})_1| = 4$ and $|S(\{v_3\})_1| = 3$. This is identical to the previous case, contributing $-25$ to the left-hand side of \eqref{eq:sigma63}.
\end{itemize}
Thus, terms in the sum with $\gamma = \gamma_{1,3}$ contribute $-50$.

\begin{figure}[hbt]
\begin{center}
\tikzset{every picture/.style={line width=0.75pt}} 
\begin{tikzpicture}[x=0.75pt,y=0.75pt,yscale=-1,xscale=1]

\draw  [color={rgb, 255:red, 208; green, 2; blue, 27 }  ,draw opacity=1 ][fill={rgb, 255:red, 208; green, 2; blue, 27 }  ,fill opacity=1 ] (39.67,46.67) .. controls (39.67,45.41) and (40.63,44.39) .. (41.82,44.39) .. controls (43.01,44.39) and (43.97,45.41) .. (43.97,46.67) .. controls (43.97,47.92) and (43.01,48.94) .. (41.82,48.94) .. controls (40.63,48.94) and (39.67,47.92) .. (39.67,46.67) -- cycle ;
\draw  [fill={rgb, 255:red, 0; green, 0; blue, 0 }  ,fill opacity=1 ] (52.58,61.72) .. controls (52.58,60.47) and (53.54,59.45) .. (54.73,59.45) .. controls (55.92,59.45) and (56.88,60.47) .. (56.88,61.72) .. controls (56.88,62.98) and (55.92,64) .. (54.73,64) .. controls (53.54,64) and (52.58,62.98) .. (52.58,61.72) -- cycle ;
\draw  [color={rgb, 255:red, 208; green, 2; blue, 27 }  ,draw opacity=1 ][fill={rgb, 255:red, 208; green, 2; blue, 27 }  ,fill opacity=1 ] (46.58,79.72) .. controls (46.58,78.47) and (47.54,77.45) .. (48.73,77.45) .. controls (49.92,77.45) and (50.88,78.47) .. (50.88,79.72) .. controls (50.88,80.98) and (49.92,82) .. (48.73,82) .. controls (47.54,82) and (46.58,80.98) .. (46.58,79.72) -- cycle ;
\draw  [color={rgb, 255:red, 208; green, 2; blue, 27 }  ,draw opacity=1 ][fill={rgb, 255:red, 208; green, 2; blue, 27 }  ,fill opacity=1 ] (23.36,76) .. controls (23.36,74.74) and (24.33,73.72) .. (25.51,73.72) .. controls (26.7,73.72) and (27.67,74.74) .. (27.67,76) .. controls (27.67,77.26) and (26.7,78.28) .. (25.51,78.28) .. controls (24.33,78.28) and (23.36,77.26) .. (23.36,76) -- cycle ;
\draw  [fill={rgb, 255:red, 0; green, 0; blue, 0 }  ,fill opacity=1 ] (22.58,55.72) .. controls (22.58,54.47) and (23.54,53.45) .. (24.73,53.45) .. controls (25.92,53.45) and (26.88,54.47) .. (26.88,55.72) .. controls (26.88,56.98) and (25.92,58) .. (24.73,58) .. controls (23.54,58) and (22.58,56.98) .. (22.58,55.72) -- cycle ;
\draw    (44.48,47.55) -- (54.29,60.41) ;
\draw [shift={(55.5,62)}, rotate = 232.67] [color={rgb, 255:red, 0; green, 0; blue, 0 }  ][line width=0.75]    (6.56,-1.97) .. controls (4.17,-0.84) and (1.99,-0.18) .. (0,0) .. controls (1.99,0.18) and (4.17,0.84) .. (6.56,1.97)   ;
\draw    (54.5,61) -- (50.69,75.48) ;
\draw [shift={(50.18,77.42)}, rotate = 284.74] [color={rgb, 255:red, 0; green, 0; blue, 0 }  ][line width=0.75]    (6.56,-1.97) .. controls (4.17,-0.84) and (1.99,-0.18) .. (0,0) .. controls (1.99,0.18) and (4.17,0.84) .. (6.56,1.97)   ;
\draw    (46.48,79.61) -- (29.63,76.38) ;
\draw [shift={(27.67,76)}, rotate = 10.87] [color={rgb, 255:red, 0; green, 0; blue, 0 }  ][line width=0.75]    (6.56,-1.97) .. controls (4.17,-0.84) and (1.99,-0.18) .. (0,0) .. controls (1.99,0.18) and (4.17,0.84) .. (6.56,1.97)   ;
\draw    (24.73,73.45) -- (24.73,60) ;
\draw [shift={(24.73,58)}, rotate = 90] [color={rgb, 255:red, 0; green, 0; blue, 0 }  ][line width=0.75]    (6.56,-1.97) .. controls (4.17,-0.84) and (1.99,-0.18) .. (0,0) .. controls (1.99,0.18) and (4.17,0.84) .. (6.56,1.97)   ;
\draw    (24.73,55.72) -- (37.96,47.7) ;
\draw [shift={(39.67,46.67)}, rotate = 148.77] [color={rgb, 255:red, 0; green, 0; blue, 0 }  ][line width=0.75]    (6.56,-1.97) .. controls (4.17,-0.84) and (1.99,-0.18) .. (0,0) .. controls (1.99,0.18) and (4.17,0.84) .. (6.56,1.97)   ;
\draw   (4,64.75) .. controls (4,45.01) and (20.01,29) .. (39.75,29) .. controls (59.49,29) and (75.5,45.01) .. (75.5,64.75) .. controls (75.5,84.49) and (59.49,100.5) .. (39.75,100.5) .. controls (20.01,100.5) and (4,84.49) .. (4,64.75) -- cycle ;
\draw  [color={rgb, 255:red, 0; green, 0; blue, 0 }  ,draw opacity=1 ][fill={rgb, 255:red, 0; green, 0; blue, 0 }  ,fill opacity=1 ] (145.35,65.09) .. controls (145.35,62.13) and (147.62,59.72) .. (150.42,59.72) .. controls (153.23,59.72) and (155.5,62.13) .. (155.5,65.09) .. controls (155.5,68.05) and (153.23,70.46) .. (150.42,70.46) .. controls (147.62,70.46) and (145.35,68.05) .. (145.35,65.09) -- cycle ;
\draw  [color={rgb, 255:red, 208; green, 2; blue, 27 }  ,draw opacity=1 ][fill={rgb, 255:red, 208; green, 2; blue, 27 }  ,fill opacity=1 ] (266.35,46) .. controls (266.35,44.74) and (267.31,43.72) .. (268.5,43.72) .. controls (269.69,43.72) and (270.65,44.74) .. (270.65,46) .. controls (270.65,47.26) and (269.69,48.28) .. (268.5,48.28) .. controls (267.31,48.28) and (266.35,47.26) .. (266.35,46) -- cycle ;
\draw  [color={rgb, 255:red, 208; green, 2; blue, 27 }  ,draw opacity=1 ][fill={rgb, 255:red, 208; green, 2; blue, 27 }  ,fill opacity=1 ] (278.58,61.72) .. controls (278.58,60.47) and (279.54,59.45) .. (280.73,59.45) .. controls (281.92,59.45) and (282.88,60.47) .. (282.88,61.72) .. controls (282.88,62.98) and (281.92,64) .. (280.73,64) .. controls (279.54,64) and (278.58,62.98) .. (278.58,61.72) -- cycle ;
\draw  [color={rgb, 255:red, 208; green, 2; blue, 27 }  ,draw opacity=1 ][fill={rgb, 255:red, 208; green, 2; blue, 27 }  ,fill opacity=1 ] (272.58,79.72) .. controls (272.58,78.47) and (273.54,77.45) .. (274.73,77.45) .. controls (275.92,77.45) and (276.88,78.47) .. (276.88,79.72) .. controls (276.88,80.98) and (275.92,82) .. (274.73,82) .. controls (273.54,82) and (272.58,80.98) .. (272.58,79.72) -- cycle ;
\draw  [color={rgb, 255:red, 208; green, 2; blue, 27 }  ,draw opacity=1 ][fill={rgb, 255:red, 208; green, 2; blue, 27 }  ,fill opacity=1 ] (249.36,76) .. controls (249.36,74.74) and (250.33,73.72) .. (251.51,73.72) .. controls (252.7,73.72) and (253.67,74.74) .. (253.67,76) .. controls (253.67,77.26) and (252.7,78.28) .. (251.51,78.28) .. controls (250.33,78.28) and (249.36,77.26) .. (249.36,76) -- cycle ;
\draw  [fill={rgb, 255:red, 0; green, 0; blue, 0 }  ,fill opacity=1 ] (248.58,55.72) .. controls (248.58,54.47) and (249.54,53.45) .. (250.73,53.45) .. controls (251.92,53.45) and (252.88,54.47) .. (252.88,55.72) .. controls (252.88,56.98) and (251.92,58) .. (250.73,58) .. controls (249.54,58) and (248.58,56.98) .. (248.58,55.72) -- cycle ;
\draw    (270.48,47.55) -- (277.87,57.24) ;
\draw [shift={(279.08,58.83)}, rotate = 232.67] [color={rgb, 255:red, 0; green, 0; blue, 0 }  ][line width=0.75]    (6.56,-1.97) .. controls (4.17,-0.84) and (1.99,-0.18) .. (0,0) .. controls (1.99,0.18) and (4.17,0.84) .. (6.56,1.97)   ;
\draw    (280.22,63.67) -- (276.96,76.05) ;
\draw [shift={(276.45,77.99)}, rotate = 284.74] [color={rgb, 255:red, 0; green, 0; blue, 0 }  ][line width=0.75]    (6.56,-1.97) .. controls (4.17,-0.84) and (1.99,-0.18) .. (0,0) .. controls (1.99,0.18) and (4.17,0.84) .. (6.56,1.97)   ;
\draw    (272.4,79.6) -- (255.63,76.38) ;
\draw [shift={(253.67,76)}, rotate = 10.87] [color={rgb, 255:red, 0; green, 0; blue, 0 }  ][line width=0.75]    (6.56,-1.97) .. controls (4.17,-0.84) and (1.99,-0.18) .. (0,0) .. controls (1.99,0.18) and (4.17,0.84) .. (6.56,1.97)   ;
\draw    (250.73,73.45) -- (250.73,60) ;
\draw [shift={(250.73,58)}, rotate = 90] [color={rgb, 255:red, 0; green, 0; blue, 0 }  ][line width=0.75]    (6.56,-1.97) .. controls (4.17,-0.84) and (1.99,-0.18) .. (0,0) .. controls (1.99,0.18) and (4.17,0.84) .. (6.56,1.97)   ;
\draw    (250.73,55.72) -- (263.96,47.7) ;
\draw [shift={(265.67,46.67)}, rotate = 148.77] [color={rgb, 255:red, 0; green, 0; blue, 0 }  ][line width=0.75]    (6.56,-1.97) .. controls (4.17,-0.84) and (1.99,-0.18) .. (0,0) .. controls (1.99,0.18) and (4.17,0.84) .. (6.56,1.97)   ;
\draw   (230,64.75) .. controls (230,45.01) and (246.01,29) .. (265.75,29) .. controls (285.49,29) and (301.5,45.01) .. (301.5,64.75) .. controls (301.5,84.49) and (285.49,100.5) .. (265.75,100.5) .. controls (246.01,100.5) and (230,84.49) .. (230,64.75) -- cycle ;
\draw    (145.35,65.09) -- (77.5,64.76) ;
\draw [shift={(75.5,64.75)}, rotate = 0.28] [color={rgb, 255:red, 0; green, 0; blue, 0 }  ][line width=0.75]    (10.93,-3.29) .. controls (6.95,-1.4) and (3.31,-0.3) .. (0,0) .. controls (3.31,0.3) and (6.95,1.4) .. (10.93,3.29)   ;
\draw    (155.5,65.09) -- (228,64.76) ;
\draw [shift={(230,64.75)}, rotate = 179.74] [color={rgb, 255:red, 0; green, 0; blue, 0 }  ][line width=0.75]    (10.93,-3.29) .. controls (6.95,-1.4) and (3.31,-0.3) .. (0,0) .. controls (3.31,0.3) and (6.95,1.4) .. (10.93,3.29)   ;
\draw  [color={rgb, 255:red, 208; green, 2; blue, 27 }  ,draw opacity=1 ][fill={rgb, 255:red, 208; green, 2; blue, 27 }  ,fill opacity=1 ] (52.58,218.72) .. controls (52.58,217.47) and (53.54,216.45) .. (54.73,216.45) .. controls (55.92,216.45) and (56.88,217.47) .. (56.88,218.72) .. controls (56.88,219.98) and (55.92,221) .. (54.73,221) .. controls (53.54,221) and (52.58,219.98) .. (52.58,218.72) -- cycle ;
\draw  [color={rgb, 255:red, 208; green, 2; blue, 27 }  ,draw opacity=1 ][fill={rgb, 255:red, 208; green, 2; blue, 27 }  ,fill opacity=1 ] (22.58,212.72) .. controls (22.58,211.47) and (23.54,210.45) .. (24.73,210.45) .. controls (25.92,210.45) and (26.88,211.47) .. (26.88,212.72) .. controls (26.88,213.98) and (25.92,215) .. (24.73,215) .. controls (23.54,215) and (22.58,213.98) .. (22.58,212.72) -- cycle ;
\draw   (4,221.75) .. controls (4,202.01) and (20.01,186) .. (39.75,186) .. controls (59.49,186) and (75.5,202.01) .. (75.5,221.75) .. controls (75.5,241.49) and (59.49,257.5) .. (39.75,257.5) .. controls (20.01,257.5) and (4,241.49) .. (4,221.75) -- cycle ;
\draw  [color={rgb, 255:red, 0; green, 0; blue, 0 }  ,draw opacity=1 ][fill={rgb, 255:red, 0; green, 0; blue, 0 }  ,fill opacity=1 ] (145.35,222.09) .. controls (145.35,219.13) and (147.62,216.72) .. (150.42,216.72) .. controls (153.23,216.72) and (155.5,219.13) .. (155.5,222.09) .. controls (155.5,225.05) and (153.23,227.46) .. (150.42,227.46) .. controls (147.62,227.46) and (145.35,225.05) .. (145.35,222.09) -- cycle ;
\draw  [color={rgb, 255:red, 208; green, 2; blue, 27 }  ,draw opacity=1 ][fill={rgb, 255:red, 208; green, 2; blue, 27 }  ,fill opacity=1 ] (248.58,212.72) .. controls (248.58,211.47) and (249.54,210.45) .. (250.73,210.45) .. controls (251.92,210.45) and (252.88,211.47) .. (252.88,212.72) .. controls (252.88,213.98) and (251.92,215) .. (250.73,215) .. controls (249.54,215) and (248.58,213.98) .. (248.58,212.72) -- cycle ;
\draw   (230,221.75) .. controls (230,202.01) and (246.01,186) .. (265.75,186) .. controls (285.49,186) and (301.5,202.01) .. (301.5,221.75) .. controls (301.5,241.49) and (285.49,257.5) .. (265.75,257.5) .. controls (246.01,257.5) and (230,241.49) .. (230,221.75) -- cycle ;
\draw    (145.35,222.09) -- (77.5,221.76) ;
\draw [shift={(75.5,221.75)}, rotate = 0.28] [color={rgb, 255:red, 0; green, 0; blue, 0 }  ][line width=0.75]    (10.93,-3.29) .. controls (6.95,-1.4) and (3.31,-0.3) .. (0,0) .. controls (3.31,0.3) and (6.95,1.4) .. (10.93,3.29)   ;
\draw    (155.5,222.09) -- (228,221.76) ;
\draw [shift={(230,221.75)}, rotate = 179.74] [color={rgb, 255:red, 0; green, 0; blue, 0 }  ][line width=0.75]    (10.93,-3.29) .. controls (6.95,-1.4) and (3.31,-0.3) .. (0,0) .. controls (3.31,0.3) and (6.95,1.4) .. (10.93,3.29)   ;
\draw [line width=1.5]    (150.5,113.5) -- (150.5,160.5) ;
\draw [shift={(150.5,163.5)}, rotate = 270] [color={rgb, 255:red, 0; green, 0; blue, 0 }  ][line width=1.5]    (14.21,-4.28) .. controls (9.04,-1.82) and (4.3,-0.39) .. (0,0) .. controls (4.3,0.39) and (9.04,1.82) .. (14.21,4.28)   ;

\draw (39,30.4) node [anchor=north west][inner sep=0.75pt]  [font=\footnotesize,color={rgb, 255:red, 208; green, 2; blue, 27 }  ,opacity=1 ]  {$1$};
\draw (58,54.4) node [anchor=north west][inner sep=0.75pt]  [font=\footnotesize]  {$2$};
\draw (50.73,80.85) node [anchor=north west][inner sep=0.75pt]  [font=\footnotesize,color={rgb, 255:red, 208; green, 2; blue, 27 }  ,opacity=1 ]  {$3$};
\draw (15,74.4) node [anchor=north west][inner sep=0.75pt]  [font=\footnotesize,color={rgb, 255:red, 208; green, 2; blue, 27 }  ,opacity=1 ]  {$4$};
\draw (13,47.4) node [anchor=north west][inner sep=0.75pt]  [font=\footnotesize]  {$5$};
\draw (267.75,32.4) node [anchor=north west][inner sep=0.75pt]  [font=\footnotesize,color={rgb, 255:red, 208; green, 2; blue, 27 }  ,opacity=1 ]  {$13$};
\draw (283,55.4) node [anchor=north west][inner sep=0.75pt]  [font=\footnotesize,color={rgb, 255:red, 208; green, 2; blue, 27 }  ,opacity=1 ]  {$14$};
\draw (272.4,82) node [anchor=north west][inner sep=0.75pt]  [font=\footnotesize,color={rgb, 255:red, 208; green, 2; blue, 27 }  ,opacity=1 ]  {$15$};
\draw (240,78.4) node [anchor=north west][inner sep=0.75pt]  [font=\footnotesize,color={rgb, 255:red, 208; green, 2; blue, 27 }  ,opacity=1 ]  {$16$};
\draw (235,46.4) node [anchor=north west][inner sep=0.75pt]  [font=\footnotesize]  {$17$};
\draw (31,8.4) node [anchor=north west][inner sep=0.75pt]    {$v_{1}$};
\draw (132,39.4) node [anchor=north west][inner sep=0.75pt]    {$v_{2} ,\ 7$};
\draw (258,8.4) node [anchor=north west][inner sep=0.75pt]    {$v_{3}$};
\draw (58,211.4) node [anchor=north west][inner sep=0.75pt]  [font=\footnotesize,color={rgb, 255:red, 208; green, 2; blue, 27 }  ,opacity=1 ]  {$2$};
\draw (13,204.4) node [anchor=north west][inner sep=0.75pt]  [font=\footnotesize,color={rgb, 255:red, 208; green, 2; blue, 27 }  ,opacity=1 ]  {$5$};
\draw (234,203.4) node [anchor=north west][inner sep=0.75pt]  [font=\footnotesize,color={rgb, 255:red, 208; green, 2; blue, 27 }  ,opacity=1 ]  {$17$};
\draw (31,165.4) node [anchor=north west][inner sep=0.75pt]    {$v_{1}$};
\draw (132,196.4) node [anchor=north west][inner sep=0.75pt]    {$v_{2} ,\ 7$};
\draw (258,165.4) node [anchor=north west][inner sep=0.75pt]    {$v_{3}$};

\end{tikzpicture}
\end{center}
\captionsetup{skip=2pt}
\caption{One example of having $|S(\{v_1\})_1|=3$ and $|S(\{v_3\})_1|=4$ for the orientation $\gamma_{1,3}$. Elements selected by $S$ are colored in red.}
\label{fig:conjecture-example}
\end{figure}

It remains to consider the orientation $\gamma_2$ with unique sink $v_2$. In this case, we must choose $S(\{v_2\})_1 = \omega(v_2)$. Then $\Sink_2'(\gamma_2,S) = \{v_1,v_3\}$, so a unique nonempty subset $B$ of $\{v_1,v_3\}$ is the only $B \subseteq V(P_{5,7,5})$ such that $S_2(B) \neq \varnothing$. For each possible choice of $B$ we have $D_B = \{v_2\}$. We consider each case:
\begin{itemize}
    \item $B = \{v_1\}$. Then following the notation in Definition \ref{def:drop}, since $5 = w(B) \leq w(D_B) = 7$, we have $T = \omega(v_1) = \{1,2,3,4,5\}$, and we must choose $S(B)_2$ to be a three-element subset of $T$, which may be done in $\binom{5}{3} = 10$ ways. Since $|B| = 1$ and $|S(B)_2| = 2$. The sign of all such $S$ is $(-1)^{7-1}(-1)^{3-1} = 1$, so this contributes $10$ to the left-hand side of \eqref{eq:sigma63}.
    \item Similar to the above, if $B = \{v_3\}$, it is easy to verify that we get a contribution of $10$ across all valid $S$.
    \item If $B = \{v_1,v_3\}$, then this time since $10 = w(B) > W(D_B) = 7$, our choice of $T$ is only the seven smallest elements of $\omega(B)$, so $T = \{1,2,3,4,5,13,14\}$. We must select $S_2(B)$ to be a three-element subset of $T$, and there are $\binom{7}{3} = 35$ choices. Each such $S$ has sign $(-1)^{7-1}(-1)^{3-2} = -1$, so this contributes $-35$ to the left-hand side of \eqref{eq:sigma63}.
\end{itemize}

Adding all cases together, Conjecture \ref{conjecture} correctly determines that $$\sigma_{7,3}\lrp{X_{(P_{5,7,5},\, \omega)}} = -25-25+10+10-35 = -65.$$

In addition to this example, we have tested the conjecture on a variety of weighted graphs for many choices of $s$-allowable $\mu$ and $j$. 

\subsection{Conjecture \ref{conjecture} on Unweighted Claw-Free Graphs}

In the case of set-weighted graphs $(G,\omega)$ in which each vertex has weight $1$ (equivalently vertex-labelled graphs, the case most commonly studied in the literature), Conjecture \ref{conjecture} has interesting implications. We refer to these as \emph{unweighted graphs}.

Note from Definition \ref{def:sallow} that for any set-weighted graph $(G,\omega)$, \emph{every} $\mu$ is $s$-allowable if for every connected bipartite induced subgraph of $G$ with bipartition $U \sqcup T$, we have $|w(U)-w(T)| \leq 1$. In the case that $G$ is unweighted, this simplifies to saying that for all such induced subgraphs, $\big||U|-|T|\big| \leq 1$. We will show that this fact plays very nicely with claw-free graphs.

The \emph{claw} is the graph $Y$ with $V(Y) = \{1,2,3,4\}$ and $E(Y) = \{12, 13, 14\}$. A graph $G$ is said to be \emph{claw-free} if it has no induced subgraph isomorphic to the claw.

\begin{lemma}\label{lem:clawfree}

Let $G$ be a claw-free graph, and suppose that $U, T$ are disjoint stable sets of $G$ such that $G[U \sqcup T]$ is a connected bipartite graph. Then $\big||U|-|T|\big| \leq 1$.

\end{lemma}

\begin{proof}

Suppose otherwise, that without loss of generality $|U| > |T|+1$. Since $G[U \sqcup T]$ is connected, it has a spanning tree $P$. No vertex of $P$ has degree at least $3$, since then this vertex and any of three of its neighbours would form an induced claw in $G[U \sqcup T]$. Thus, every vertex of $P$ has degree equal to exactly $1$ or $2$. In this case, it is easy to verify that $P$ is a path (for instance, start at a leaf vertex and traverse the tree)\footnote{In fact, it is straightforward to show from here that $G[U \sqcup T]$ is either $P$ itself or a cycle formed by adding an edge to $P$, but we do not need this here.}. But the path $P$ in the bipartite graph $G[U \sqcup T]$ uses at most one more vertex from $U$ than from $T$, contradicting that $P$ is a spanning tree since $|U| > |T|+1$.
\end{proof}

Then it immediately follows from the above discussion that
\begin{cor}\label{cor:clawfree}
If Conjecture \ref{conjecture} holds, then we may use it to evaluate $\sigma_{\mu,j}(X_G)$ for any claw-free graph $G$ and any integers $\mu \geq 1$ and $j \in \{0,\dots,|V(G)|-\mu\}$.
\end{cor}

We now give an example of applying Corollary \ref{cor:clawfree} to an unweighted claw-free graph. In particular, we consider the net graph $G$ as shown in Figure \ref{fig:net-graph}.
Note that the net graph is claw-free, and it is not an incomparability graph of a poset.
We consider $\mu=(2)$ and $j=2$.
Note that we cannot apply Theorem \ref{thm:main} since $\mu=(2)$ is not maximal. (The largest stable set of $G$ has size 3.)
Using SageMath, we can compute that 
\[\sigma_{2,2}\lrp{X_G}=\sigma_{2,2}\lrp{6e_{321}-6e_{33}+6e_{411}+12e_{42}+18e_{51}+12e_6}=6.\]

\begin{figure}[hbt]
\begin{center}
\tikzset{every picture/.style={line width=0.75pt}} 
\begin{tikzpicture}[x=0.75pt,y=0.75pt,yscale=-1,xscale=1]

\draw  [color={rgb, 255:red, 0; green, 0; blue, 0 }  ,draw opacity=1 ][fill={rgb, 255:red, 0; green, 0; blue, 0 }  ,fill opacity=1 ] (74.58,13.11) .. controls (74.58,11.09) and (76.13,9.45) .. (78.04,9.45) .. controls (79.95,9.45) and (81.5,11.09) .. (81.5,13.11) .. controls (81.5,15.13) and (79.95,16.77) .. (78.04,16.77) .. controls (76.13,16.77) and (74.58,15.13) .. (74.58,13.11) -- cycle ;
\draw  [color={rgb, 255:red, 0; green, 0; blue, 0 }  ,draw opacity=1 ][fill={rgb, 255:red, 0; green, 0; blue, 0 }  ,fill opacity=1 ] (74.58,54.77) .. controls (74.58,52.75) and (76.13,51.11) .. (78.04,51.11) .. controls (79.95,51.11) and (81.5,52.75) .. (81.5,54.77) .. controls (81.5,56.79) and (79.95,58.43) .. (78.04,58.43) .. controls (76.13,58.43) and (74.58,56.79) .. (74.58,54.77) -- cycle ;
\draw  [color={rgb, 255:red, 0; green, 0; blue, 0 }  ,draw opacity=1 ][fill={rgb, 255:red, 0; green, 0; blue, 0 }  ,fill opacity=1 ] (44.58,93.11) .. controls (44.58,91.09) and (46.13,89.45) .. (48.04,89.45) .. controls (49.95,89.45) and (51.5,91.09) .. (51.5,93.11) .. controls (51.5,95.13) and (49.95,96.77) .. (48.04,96.77) .. controls (46.13,96.77) and (44.58,95.13) .. (44.58,93.11) -- cycle ;
\draw  [color={rgb, 255:red, 0; green, 0; blue, 0 }  ,draw opacity=1 ][fill={rgb, 255:red, 0; green, 0; blue, 0 }  ,fill opacity=1 ] (104.04,93.11) .. controls (104.04,91.09) and (105.59,89.45) .. (107.5,89.45) .. controls (109.41,89.45) and (110.96,91.09) .. (110.96,93.11) .. controls (110.96,95.13) and (109.41,96.77) .. (107.5,96.77) .. controls (105.59,96.77) and (104.04,95.13) .. (104.04,93.11) -- cycle ;
\draw  [color={rgb, 255:red, 0; green, 0; blue, 0 }  ,draw opacity=1 ][fill={rgb, 255:red, 0; green, 0; blue, 0 }  ,fill opacity=1 ] (12.58,123.11) .. controls (12.58,121.09) and (14.13,119.45) .. (16.04,119.45) .. controls (17.95,119.45) and (19.5,121.09) .. (19.5,123.11) .. controls (19.5,125.13) and (17.95,126.77) .. (16.04,126.77) .. controls (14.13,126.77) and (12.58,125.13) .. (12.58,123.11) -- cycle ;
\draw  [color={rgb, 255:red, 0; green, 0; blue, 0 }  ,draw opacity=1 ][fill={rgb, 255:red, 0; green, 0; blue, 0 }  ,fill opacity=1 ] (135.58,123.11) .. controls (135.58,121.09) and (137.13,119.45) .. (139.04,119.45) .. controls (140.95,119.45) and (142.5,121.09) .. (142.5,123.11) .. controls (142.5,125.13) and (140.95,126.77) .. (139.04,126.77) .. controls (137.13,126.77) and (135.58,125.13) .. (135.58,123.11) -- cycle ;
\draw    (78.04,16.77) -- (78.04,51.11) ;
\draw    (75.5,58) -- (49.5,91) ;
\draw    (51.5,93.11) -- (104.04,93.11) ;
\draw    (80.5,57) -- (106.5,91) ;
\draw    (18.5,121) -- (45.5,96) ;
\draw    (109.5,95) -- (137.5,122) ;

\draw (83,6.4) node [anchor=north west][inner sep=0.75pt]  [font=\small]  {$v_{1}$};
\draw (81,44.4) node [anchor=north west][inner sep=0.75pt]  [font=\small]  {$v_{2}$};
\draw (31,78.4) node [anchor=north west][inner sep=0.75pt]  [font=\small]  {$v_{3}$};
\draw (108,79.4) node [anchor=north west][inner sep=0.75pt]  [font=\small]  {$v_{4}$};
\draw (7,129.4) node [anchor=north west][inner sep=0.75pt]  [font=\small]  {$v_{5}$};
\draw (131,130.4) node [anchor=north west][inner sep=0.75pt]  [font=\small]  {$v_{6}$};

\end{tikzpicture}
\end{center}
\captionsetup{skip=2pt}
\caption{The net graph.}
\label{fig:net-graph}
\end{figure}

We then compute the right-hand side of equation \eqref{conjecture-equation}.
In order to obtain a weight sequence that starts with 2, we consider the acyclic orientations of $G$ that can give two sinks.
We first consider the case where $v_1$ and $v_3$ are sinks.
There are exactly two acyclic orientations where $\{v_1,v_3\}$ are sinks.
The first one, which we call $\gamma_1$, is given in Figure \ref{fig:net-orientation1}.
For the first step of a weight map $S$, we must select (i.e. be non-empty on) $\{v_1\}$ and $\{v_3\}$.
Then using Definition \ref{def:drop}, we see that $D=\{v_1,v_3\}$, $\Sink_2'(\gamma_1,S)=\{v_2,v_5\}$, and $G[D\cup \Sink_2'(\gamma_1,S)]$ has one connected component $M_1$.
It follows that we must have $B_1=\{v_2,v_5\}$ and $|S(B_1)_2|=2$. (Recall (3) of Definition \ref{def:drop}.)
Hence, for $\gamma_1$, there is only one 2-step weight map $S$ that gives a weight sequence of $(2, 2)$. We have $\mathrm{sgn}(\gamma,S)=1$, so this case contributes 1 to the sum in the right-hand side of equation \eqref{conjecture-equation}.

\begin{figure}[hbt]
\begin{center}
\tikzset{every picture/.style={line width=0.75pt}} 
\begin{tikzpicture}[x=0.75pt,y=0.75pt,yscale=-1,xscale=1]

\draw  [color={rgb, 255:red, 208; green, 2; blue, 27 }  ,draw opacity=1 ][fill={rgb, 255:red, 208; green, 2; blue, 27 }  ,fill opacity=1 ] (210.58,22.11) .. controls (210.58,20.09) and (212.13,18.45) .. (214.04,18.45) .. controls (215.95,18.45) and (217.5,20.09) .. (217.5,22.11) .. controls (217.5,24.13) and (215.95,25.77) .. (214.04,25.77) .. controls (212.13,25.77) and (210.58,24.13) .. (210.58,22.11) -- cycle ;
\draw  [color={rgb, 255:red, 245; green, 166; blue, 35 }  ,draw opacity=1 ][fill={rgb, 255:red, 245; green, 166; blue, 35 }  ,fill opacity=1 ] (210.58,63.77) .. controls (210.58,61.75) and (212.13,60.11) .. (214.04,60.11) .. controls (215.95,60.11) and (217.5,61.75) .. (217.5,63.77) .. controls (217.5,65.79) and (215.95,67.43) .. (214.04,67.43) .. controls (212.13,67.43) and (210.58,65.79) .. (210.58,63.77) -- cycle ;
\draw  [color={rgb, 255:red, 208; green, 2; blue, 27 }  ,draw opacity=1 ][fill={rgb, 255:red, 208; green, 2; blue, 27 }  ,fill opacity=1 ] (180.58,102.11) .. controls (180.58,100.09) and (182.13,98.45) .. (184.04,98.45) .. controls (185.95,98.45) and (187.5,100.09) .. (187.5,102.11) .. controls (187.5,104.13) and (185.95,105.77) .. (184.04,105.77) .. controls (182.13,105.77) and (180.58,104.13) .. (180.58,102.11) -- cycle ;
\draw  [color={rgb, 255:red, 0; green, 0; blue, 0 }  ,draw opacity=1 ][fill={rgb, 255:red, 0; green, 0; blue, 0 }  ,fill opacity=1 ] (240.04,102.11) .. controls (240.04,100.09) and (241.59,98.45) .. (243.5,98.45) .. controls (245.41,98.45) and (246.96,100.09) .. (246.96,102.11) .. controls (246.96,104.13) and (245.41,105.77) .. (243.5,105.77) .. controls (241.59,105.77) and (240.04,104.13) .. (240.04,102.11) -- cycle ;
\draw  [color={rgb, 255:red, 245; green, 166; blue, 35 }  ,draw opacity=1 ][fill={rgb, 255:red, 245; green, 166; blue, 35 }  ,fill opacity=1 ] (148.58,132.11) .. controls (148.58,130.09) and (150.13,128.45) .. (152.04,128.45) .. controls (153.95,128.45) and (155.5,130.09) .. (155.5,132.11) .. controls (155.5,134.13) and (153.95,135.77) .. (152.04,135.77) .. controls (150.13,135.77) and (148.58,134.13) .. (148.58,132.11) -- cycle ;
\draw  [color={rgb, 255:red, 0; green, 0; blue, 0 }  ,draw opacity=1 ][fill={rgb, 255:red, 0; green, 0; blue, 0 }  ,fill opacity=1 ] (271.58,132.11) .. controls (271.58,130.09) and (273.13,128.45) .. (275.04,128.45) .. controls (276.95,128.45) and (278.5,130.09) .. (278.5,132.11) .. controls (278.5,134.13) and (276.95,135.77) .. (275.04,135.77) .. controls (273.13,135.77) and (271.58,134.13) .. (271.58,132.11) -- cycle ;
\draw    (214.04,27.77) -- (214.04,60.11) ;
\draw [shift={(214.04,25.77)}, rotate = 90] [color={rgb, 255:red, 0; green, 0; blue, 0 }  ][line width=0.75]    (10.93,-3.29) .. controls (6.95,-1.4) and (3.31,-0.3) .. (0,0) .. controls (3.31,0.3) and (6.95,1.4) .. (10.93,3.29)   ;
\draw    (211.5,67) -- (186.74,98.43) ;
\draw [shift={(185.5,100)}, rotate = 308.23] [color={rgb, 255:red, 0; green, 0; blue, 0 }  ][line width=0.75]    (10.93,-3.29) .. controls (6.95,-1.4) and (3.31,-0.3) .. (0,0) .. controls (3.31,0.3) and (6.95,1.4) .. (10.93,3.29)   ;
\draw    (189.5,102.11) -- (240.04,102.11) ;
\draw [shift={(187.5,102.11)}, rotate = 0] [color={rgb, 255:red, 0; green, 0; blue, 0 }  ][line width=0.75]    (10.93,-3.29) .. controls (6.95,-1.4) and (3.31,-0.3) .. (0,0) .. controls (3.31,0.3) and (6.95,1.4) .. (10.93,3.29)   ;
\draw    (217.71,67.59) -- (242.5,100) ;
\draw [shift={(216.5,66)}, rotate = 52.59] [color={rgb, 255:red, 0; green, 0; blue, 0 }  ][line width=0.75]    (10.93,-3.29) .. controls (6.95,-1.4) and (3.31,-0.3) .. (0,0) .. controls (3.31,0.3) and (6.95,1.4) .. (10.93,3.29)   ;
\draw    (154.5,130) -- (180.03,106.36) ;
\draw [shift={(181.5,105)}, rotate = 137.2] [color={rgb, 255:red, 0; green, 0; blue, 0 }  ][line width=0.75]    (10.93,-3.29) .. controls (6.95,-1.4) and (3.31,-0.3) .. (0,0) .. controls (3.31,0.3) and (6.95,1.4) .. (10.93,3.29)   ;
\draw    (246.94,105.39) -- (273.5,131) ;
\draw [shift={(245.5,104)}, rotate = 43.96] [color={rgb, 255:red, 0; green, 0; blue, 0 }  ][line width=0.75]    (10.93,-3.29) .. controls (6.95,-1.4) and (3.31,-0.3) .. (0,0) .. controls (3.31,0.3) and (6.95,1.4) .. (10.93,3.29)   ;
\draw  [color={rgb, 255:red, 245; green, 166; blue, 35 }  ,draw opacity=1 ][fill={rgb, 255:red, 245; green, 166; blue, 35 }  ,fill opacity=1 ] (465.58,63.77) .. controls (465.58,61.75) and (467.13,60.11) .. (469.04,60.11) .. controls (470.95,60.11) and (472.5,61.75) .. (472.5,63.77) .. controls (472.5,65.79) and (470.95,67.43) .. (469.04,67.43) .. controls (467.13,67.43) and (465.58,65.79) .. (465.58,63.77) -- cycle ;
\draw  [color={rgb, 255:red, 0; green, 0; blue, 0 }  ,draw opacity=1 ][fill={rgb, 255:red, 0; green, 0; blue, 0 }  ,fill opacity=1 ] (495.04,102.11) .. controls (495.04,100.09) and (496.59,98.45) .. (498.5,98.45) .. controls (500.41,98.45) and (501.96,100.09) .. (501.96,102.11) .. controls (501.96,104.13) and (500.41,105.77) .. (498.5,105.77) .. controls (496.59,105.77) and (495.04,104.13) .. (495.04,102.11) -- cycle ;
\draw  [color={rgb, 255:red, 245; green, 166; blue, 35 }  ,draw opacity=1 ][fill={rgb, 255:red, 245; green, 166; blue, 35 }  ,fill opacity=1 ] (403.58,132.11) .. controls (403.58,130.09) and (405.13,128.45) .. (407.04,128.45) .. controls (408.95,128.45) and (410.5,130.09) .. (410.5,132.11) .. controls (410.5,134.13) and (408.95,135.77) .. (407.04,135.77) .. controls (405.13,135.77) and (403.58,134.13) .. (403.58,132.11) -- cycle ;
\draw  [color={rgb, 255:red, 0; green, 0; blue, 0 }  ,draw opacity=1 ][fill={rgb, 255:red, 0; green, 0; blue, 0 }  ,fill opacity=1 ] (526.58,132.11) .. controls (526.58,130.09) and (528.13,128.45) .. (530.04,128.45) .. controls (531.95,128.45) and (533.5,130.09) .. (533.5,132.11) .. controls (533.5,134.13) and (531.95,135.77) .. (530.04,135.77) .. controls (528.13,135.77) and (526.58,134.13) .. (526.58,132.11) -- cycle ;
\draw    (472.71,67.59) -- (497.5,100) ;
\draw [shift={(471.5,66)}, rotate = 52.59] [color={rgb, 255:red, 0; green, 0; blue, 0 }  ][line width=0.75]    (10.93,-3.29) .. controls (6.95,-1.4) and (3.31,-0.3) .. (0,0) .. controls (3.31,0.3) and (6.95,1.4) .. (10.93,3.29)   ;
\draw    (501.94,105.39) -- (528.5,131) ;
\draw [shift={(500.5,104)}, rotate = 43.96] [color={rgb, 255:red, 0; green, 0; blue, 0 }  ][line width=0.75]    (10.93,-3.29) .. controls (6.95,-1.4) and (3.31,-0.3) .. (0,0) .. controls (3.31,0.3) and (6.95,1.4) .. (10.93,3.29)   ;
\draw [line width=1.5]    (310,91) -- (366.5,91) ;
\draw [shift={(370.5,91)}, rotate = 180] [fill={rgb, 255:red, 0; green, 0; blue, 0 }  ][line width=0.08]  [draw opacity=0] (11.61,-5.58) -- (0,0) -- (11.61,5.58) -- cycle    ;
\draw  [color={rgb, 255:red, 208; green, 2; blue, 27 }  ,draw opacity=1 ][fill={rgb, 255:red, 208; green, 2; blue, 27 }  ,fill opacity=1 ] (3.58,65.11) .. controls (3.58,63.09) and (5.13,61.45) .. (7.04,61.45) .. controls (8.95,61.45) and (10.5,63.09) .. (10.5,65.11) .. controls (10.5,67.13) and (8.95,68.77) .. (7.04,68.77) .. controls (5.13,68.77) and (3.58,67.13) .. (3.58,65.11) -- cycle ;
\draw  [color={rgb, 255:red, 245; green, 166; blue, 35 }  ,draw opacity=1 ][fill={rgb, 255:red, 245; green, 166; blue, 35 }  ,fill opacity=1 ] (3.58,97.11) .. controls (3.58,95.09) and (5.13,93.45) .. (7.04,93.45) .. controls (8.95,93.45) and (10.5,95.09) .. (10.5,97.11) .. controls (10.5,99.13) and (8.95,100.77) .. (7.04,100.77) .. controls (5.13,100.77) and (3.58,99.13) .. (3.58,97.11) -- cycle ;
\draw  [color={rgb, 255:red, 155; green, 155; blue, 155 }  ,draw opacity=1 ] (156.18,62.95) .. controls (180.67,20.75) and (215.74,-4.62) .. (234.51,6.27) .. controls (253.28,17.17) and (248.65,60.21) .. (224.16,102.4) .. controls (199.67,144.59) and (164.6,169.97) .. (145.83,159.07) .. controls (127.06,148.18) and (131.7,105.14) .. (156.18,62.95) -- cycle ;

\draw (219,15.4) node [anchor=north west][inner sep=0.75pt]  [font=\small]  {$v_{1}$};
\draw (217,53.4) node [anchor=north west][inner sep=0.75pt]  [font=\small]  {$v_{2}$};
\draw (167,87.4) node [anchor=north west][inner sep=0.75pt]  [font=\small]  {$v_{3}$};
\draw (244,88.4) node [anchor=north west][inner sep=0.75pt]  [font=\small]  {$v_{4}$};
\draw (143,138.4) node [anchor=north west][inner sep=0.75pt]  [font=\small]  {$v_{5}$};
\draw (267,139.4) node [anchor=north west][inner sep=0.75pt]  [font=\small]  {$v_{6}$};
\draw (474,51.4) node [anchor=north west][inner sep=0.75pt]  [font=\small]  {$v_{2}$};
\draw (501,88.4) node [anchor=north west][inner sep=0.75pt]  [font=\small]  {$v_{4}$};
\draw (533.04,132.85) node [anchor=north west][inner sep=0.75pt]  [font=\small]  {$v_{6}$};
\draw (389,132.4) node [anchor=north west][inner sep=0.75pt]  [font=\small]  {$v_{5}$};
\draw (13,57.4) node [anchor=north west][inner sep=0.75pt]  [color={rgb, 255:red, 208; green, 2; blue, 27 }  ,opacity=1 ]  {$:D$};
\draw (13,87.4) node [anchor=north west][inner sep=0.75pt]  [color={rgb, 255:red, 245; green, 166; blue, 35 }  ,opacity=1 ]  {$:\text{Sink}_{2} '( \gamma _{1} ,S)$};
\draw (162,56.4) node [anchor=north west][inner sep=0.75pt]  [color={rgb, 255:red, 155; green, 155; blue, 155 }  ,opacity=1 ]  {$M_{1}$};

\end{tikzpicture}
\end{center}
\captionsetup{skip=2pt}
\caption{Orientation $\gamma_1$, which gives $v_1$ and $v_3$ as sinks.}
\label{fig:net-orientation1}
\end{figure}

We then consider the second acyclic orientation $\gamma_2$ whose sinks are $v_1$ and $v_3$, and we give an illustration in Figure \ref{fig:net-orientation2}.
Again, to obtain a weight sequence of $(2, 2)$, a weight map $S$ must be non-empty on $\{v_1\}$ and $\{v_3\}$ in the first step.
Then we have $D=\{v_1,v_3\}$, $\Sink_2'(\gamma_2,S)=\{v_4,v_5\}$, and $G[D\cup \Sink_2'(\gamma_2,S)]$ has two connected components $M_1$ and $M_2$.
Now to have 2 in the second term of the weight sequence, by Definition \ref{def:drop}(3), we must have $B_1=\varnothing$ and $B_2=\{v_4,v_5\}$.
Then $D_{B_2}=\{v_3\}$ so $1=w(D_{B_2})<w(B_2)=2$.
But now it is impossible to have $|S(B_2)_2|=2$ by the definition.
Therefore, $\gamma_2$ contribute 0 to the sum in the right-hand side of equation \eqref{conjecture-equation}.
We have finished the case where $\{v_1,v_3\}$ are sinks.

\begin{figure}[hbt]
\begin{center}
\tikzset{every picture/.style={line width=0.75pt}} 
\begin{tikzpicture}[x=0.75pt,y=0.75pt,yscale=-1,xscale=1]

\draw  [color={rgb, 255:red, 208; green, 2; blue, 27 }  ,draw opacity=1 ][fill={rgb, 255:red, 208; green, 2; blue, 27 }  ,fill opacity=1 ] (224.58,30.11) .. controls (224.58,28.09) and (226.13,26.45) .. (228.04,26.45) .. controls (229.95,26.45) and (231.5,28.09) .. (231.5,30.11) .. controls (231.5,32.13) and (229.95,33.77) .. (228.04,33.77) .. controls (226.13,33.77) and (224.58,32.13) .. (224.58,30.11) -- cycle ;
\draw  [color={rgb, 255:red, 0; green, 0; blue, 0 }  ,draw opacity=1 ][fill={rgb, 255:red, 0; green, 0; blue, 0 }  ,fill opacity=1 ] (224.58,71.77) .. controls (224.58,69.75) and (226.13,68.11) .. (228.04,68.11) .. controls (229.95,68.11) and (231.5,69.75) .. (231.5,71.77) .. controls (231.5,73.79) and (229.95,75.43) .. (228.04,75.43) .. controls (226.13,75.43) and (224.58,73.79) .. (224.58,71.77) -- cycle ;
\draw  [color={rgb, 255:red, 208; green, 2; blue, 27 }  ,draw opacity=1 ][fill={rgb, 255:red, 208; green, 2; blue, 27 }  ,fill opacity=1 ] (194.58,110.11) .. controls (194.58,108.09) and (196.13,106.45) .. (198.04,106.45) .. controls (199.95,106.45) and (201.5,108.09) .. (201.5,110.11) .. controls (201.5,112.13) and (199.95,113.77) .. (198.04,113.77) .. controls (196.13,113.77) and (194.58,112.13) .. (194.58,110.11) -- cycle ;
\draw  [color={rgb, 255:red, 245; green, 166; blue, 35 }  ,draw opacity=1 ][fill={rgb, 255:red, 245; green, 166; blue, 35 }  ,fill opacity=1 ] (254.04,110.11) .. controls (254.04,108.09) and (255.59,106.45) .. (257.5,106.45) .. controls (259.41,106.45) and (260.96,108.09) .. (260.96,110.11) .. controls (260.96,112.13) and (259.41,113.77) .. (257.5,113.77) .. controls (255.59,113.77) and (254.04,112.13) .. (254.04,110.11) -- cycle ;
\draw  [color={rgb, 255:red, 245; green, 166; blue, 35 }  ,draw opacity=1 ][fill={rgb, 255:red, 245; green, 166; blue, 35 }  ,fill opacity=1 ] (162.58,140.11) .. controls (162.58,138.09) and (164.13,136.45) .. (166.04,136.45) .. controls (167.95,136.45) and (169.5,138.09) .. (169.5,140.11) .. controls (169.5,142.13) and (167.95,143.77) .. (166.04,143.77) .. controls (164.13,143.77) and (162.58,142.13) .. (162.58,140.11) -- cycle ;
\draw  [color={rgb, 255:red, 0; green, 0; blue, 0 }  ,draw opacity=1 ][fill={rgb, 255:red, 0; green, 0; blue, 0 }  ,fill opacity=1 ] (285.58,140.11) .. controls (285.58,138.09) and (287.13,136.45) .. (289.04,136.45) .. controls (290.95,136.45) and (292.5,138.09) .. (292.5,140.11) .. controls (292.5,142.13) and (290.95,143.77) .. (289.04,143.77) .. controls (287.13,143.77) and (285.58,142.13) .. (285.58,140.11) -- cycle ;
\draw    (228.04,35.77) -- (228.04,68.11) ;
\draw [shift={(228.04,33.77)}, rotate = 90] [color={rgb, 255:red, 0; green, 0; blue, 0 }  ][line width=0.75]    (10.93,-3.29) .. controls (6.95,-1.4) and (3.31,-0.3) .. (0,0) .. controls (3.31,0.3) and (6.95,1.4) .. (10.93,3.29)   ;
\draw    (225.5,75) -- (200.74,106.43) ;
\draw [shift={(199.5,108)}, rotate = 308.23] [color={rgb, 255:red, 0; green, 0; blue, 0 }  ][line width=0.75]    (10.93,-3.29) .. controls (6.95,-1.4) and (3.31,-0.3) .. (0,0) .. controls (3.31,0.3) and (6.95,1.4) .. (10.93,3.29)   ;
\draw    (203.5,110.11) -- (254.04,110.11) ;
\draw [shift={(201.5,110.11)}, rotate = 0] [color={rgb, 255:red, 0; green, 0; blue, 0 }  ][line width=0.75]    (10.93,-3.29) .. controls (6.95,-1.4) and (3.31,-0.3) .. (0,0) .. controls (3.31,0.3) and (6.95,1.4) .. (10.93,3.29)   ;
\draw    (230.5,74) -- (255.29,106.41) ;
\draw [shift={(256.5,108)}, rotate = 232.59] [color={rgb, 255:red, 0; green, 0; blue, 0 }  ][line width=0.75]    (10.93,-3.29) .. controls (6.95,-1.4) and (3.31,-0.3) .. (0,0) .. controls (3.31,0.3) and (6.95,1.4) .. (10.93,3.29)   ;
\draw    (168.5,138) -- (194.03,114.36) ;
\draw [shift={(195.5,113)}, rotate = 137.2] [color={rgb, 255:red, 0; green, 0; blue, 0 }  ][line width=0.75]    (10.93,-3.29) .. controls (6.95,-1.4) and (3.31,-0.3) .. (0,0) .. controls (3.31,0.3) and (6.95,1.4) .. (10.93,3.29)   ;
\draw    (260.94,113.39) -- (287.5,139) ;
\draw [shift={(259.5,112)}, rotate = 43.96] [color={rgb, 255:red, 0; green, 0; blue, 0 }  ][line width=0.75]    (10.93,-3.29) .. controls (6.95,-1.4) and (3.31,-0.3) .. (0,0) .. controls (3.31,0.3) and (6.95,1.4) .. (10.93,3.29)   ;
\draw  [color={rgb, 255:red, 0; green, 0; blue, 0 }  ,draw opacity=1 ][fill={rgb, 255:red, 0; green, 0; blue, 0 }  ,fill opacity=1 ] (479.58,71.77) .. controls (479.58,69.75) and (481.13,68.11) .. (483.04,68.11) .. controls (484.95,68.11) and (486.5,69.75) .. (486.5,71.77) .. controls (486.5,73.79) and (484.95,75.43) .. (483.04,75.43) .. controls (481.13,75.43) and (479.58,73.79) .. (479.58,71.77) -- cycle ;
\draw  [color={rgb, 255:red, 245; green, 166; blue, 35 }  ,draw opacity=1 ][fill={rgb, 255:red, 245; green, 166; blue, 35 }  ,fill opacity=1 ] (509.04,110.11) .. controls (509.04,108.09) and (510.59,106.45) .. (512.5,106.45) .. controls (514.41,106.45) and (515.96,108.09) .. (515.96,110.11) .. controls (515.96,112.13) and (514.41,113.77) .. (512.5,113.77) .. controls (510.59,113.77) and (509.04,112.13) .. (509.04,110.11) -- cycle ;
\draw  [color={rgb, 255:red, 245; green, 166; blue, 35 }  ,draw opacity=1 ][fill={rgb, 255:red, 245; green, 166; blue, 35 }  ,fill opacity=1 ] (417.58,140.11) .. controls (417.58,138.09) and (419.13,136.45) .. (421.04,136.45) .. controls (422.95,136.45) and (424.5,138.09) .. (424.5,140.11) .. controls (424.5,142.13) and (422.95,143.77) .. (421.04,143.77) .. controls (419.13,143.77) and (417.58,142.13) .. (417.58,140.11) -- cycle ;
\draw  [color={rgb, 255:red, 0; green, 0; blue, 0 }  ,draw opacity=1 ][fill={rgb, 255:red, 0; green, 0; blue, 0 }  ,fill opacity=1 ] (540.58,140.11) .. controls (540.58,138.09) and (542.13,136.45) .. (544.04,136.45) .. controls (545.95,136.45) and (547.5,138.09) .. (547.5,140.11) .. controls (547.5,142.13) and (545.95,143.77) .. (544.04,143.77) .. controls (542.13,143.77) and (540.58,142.13) .. (540.58,140.11) -- cycle ;
\draw    (485.5,74) -- (510.29,106.41) ;
\draw [shift={(511.5,108)}, rotate = 232.59] [color={rgb, 255:red, 0; green, 0; blue, 0 }  ][line width=0.75]    (10.93,-3.29) .. controls (6.95,-1.4) and (3.31,-0.3) .. (0,0) .. controls (3.31,0.3) and (6.95,1.4) .. (10.93,3.29)   ;
\draw    (515.94,113.39) -- (542.5,139) ;
\draw [shift={(514.5,112)}, rotate = 43.96] [color={rgb, 255:red, 0; green, 0; blue, 0 }  ][line width=0.75]    (10.93,-3.29) .. controls (6.95,-1.4) and (3.31,-0.3) .. (0,0) .. controls (3.31,0.3) and (6.95,1.4) .. (10.93,3.29)   ;
\draw [line width=1.5]    (324,99) -- (380.5,99) ;
\draw [shift={(384.5,99)}, rotate = 180] [fill={rgb, 255:red, 0; green, 0; blue, 0 }  ][line width=0.08]  [draw opacity=0] (11.61,-5.58) -- (0,0) -- (11.61,5.58) -- cycle    ;
\draw  [color={rgb, 255:red, 208; green, 2; blue, 27 }  ,draw opacity=1 ][fill={rgb, 255:red, 208; green, 2; blue, 27 }  ,fill opacity=1 ] (17.58,73.11) .. controls (17.58,71.09) and (19.13,69.45) .. (21.04,69.45) .. controls (22.95,69.45) and (24.5,71.09) .. (24.5,73.11) .. controls (24.5,75.13) and (22.95,76.77) .. (21.04,76.77) .. controls (19.13,76.77) and (17.58,75.13) .. (17.58,73.11) -- cycle ;
\draw  [color={rgb, 255:red, 245; green, 166; blue, 35 }  ,draw opacity=1 ][fill={rgb, 255:red, 245; green, 166; blue, 35 }  ,fill opacity=1 ] (17.58,105.11) .. controls (17.58,103.09) and (19.13,101.45) .. (21.04,101.45) .. controls (22.95,101.45) and (24.5,103.09) .. (24.5,105.11) .. controls (24.5,107.13) and (22.95,108.77) .. (21.04,108.77) .. controls (19.13,108.77) and (17.58,107.13) .. (17.58,105.11) -- cycle ;
\draw  [color={rgb, 255:red, 155; green, 155; blue, 155 }  ,draw opacity=1 ] (195.61,89.95) .. controls (233.95,72.93) and (271.07,72.71) .. (278.51,89.47) .. controls (285.95,106.23) and (260.9,133.62) .. (222.56,150.64) .. controls (184.22,167.66) and (147.11,167.88) .. (139.66,151.12) .. controls (132.22,134.36) and (157.27,106.97) .. (195.61,89.95) -- cycle ;
\draw  [color={rgb, 255:red, 155; green, 155; blue, 155 }  ,draw opacity=1 ] (189.58,18.11) .. controls (189.58,13.69) and (193.16,10.11) .. (197.58,10.11) -- (251.58,10.11) .. controls (255.99,10.11) and (259.58,13.69) .. (259.58,18.11) -- (259.58,42.11) .. controls (259.58,46.53) and (255.99,50.11) .. (251.58,50.11) -- (197.58,50.11) .. controls (193.16,50.11) and (189.58,46.53) .. (189.58,42.11) -- cycle ;

\draw (233,23.4) node [anchor=north west][inner sep=0.75pt]  [font=\small]  {$v_{1}$};
\draw (231,61.4) node [anchor=north west][inner sep=0.75pt]  [font=\small]  {$v_{2}$};
\draw (181,95.4) node [anchor=north west][inner sep=0.75pt]  [font=\small]  {$v_{3}$};
\draw (258,96.4) node [anchor=north west][inner sep=0.75pt]  [font=\small]  {$v_{4}$};
\draw (157,146.4) node [anchor=north west][inner sep=0.75pt]  [font=\small]  {$v_{5}$};
\draw (281,147.4) node [anchor=north west][inner sep=0.75pt]  [font=\small]  {$v_{6}$};
\draw (488,59.4) node [anchor=north west][inner sep=0.75pt]  [font=\small]  {$v_{2}$};
\draw (515,96.4) node [anchor=north west][inner sep=0.75pt]  [font=\small]  {$v_{4}$};
\draw (547.04,140.85) node [anchor=north west][inner sep=0.75pt]  [font=\small]  {$v_{6}$};
\draw (403,140.4) node [anchor=north west][inner sep=0.75pt]  [font=\small]  {$v_{5}$};
\draw (27,65.4) node [anchor=north west][inner sep=0.75pt]  [color={rgb, 255:red, 208; green, 2; blue, 27 }  ,opacity=1 ]  {$:D$};
\draw (27,95.4) node [anchor=north west][inner sep=0.75pt]  [color={rgb, 255:red, 245; green, 166; blue, 35 }  ,opacity=1 ]  {$:\text{Sink}_{2} '( \gamma _{2} ,S)$};
\draw (192,14.4) node [anchor=north west][inner sep=0.75pt]  [color={rgb, 255:red, 155; green, 155; blue, 155 }  ,opacity=1 ]  {$M_{1}$};
\draw (207,130.4) node [anchor=north west][inner sep=0.75pt]  [color={rgb, 255:red, 155; green, 155; blue, 155 }  ,opacity=1 ]  {$M_{2}$};

\end{tikzpicture}
\end{center}
\captionsetup{skip=2pt}
\caption{Orientation $\gamma_2$, which also gives $v_1$ and $v_3$ as sinks.}
\label{fig:net-orientation2}
\end{figure}

For the cases where $\{v_1,v_4\}$, $\{v_2,v_5\}$, $\{v_2,v_6\}$, $\{v_3,v_6\}$, or $\{v_4,v_5\}$ are sinks, the analyses are identical by symmetry. So these cases, together with the case where $\{v_1,v_3\}$ are sinks, contribute a total of 6 to the sum in the right-hand side of equation \eqref{conjecture-equation}.

Next, we consider the case where $\{v_1,v_5\}$ are sinks. 
All acyclic orientations that give $\{v_1,v_5\}$ as sinks are illustrated in Figure \ref{fig:net-orientation3}.
However, we observe that for all of these acyclic orientations, there is only one second-level sink, so it is impossible to obtain a weight sequence of $(2,2)$ in these cases.
By symmetry, we also cannot obtain a weight sequence of $(2,2)$ when the sinks are $\{v_1,v_6\}$ or $\{v_5,v_6\}$.

\begin{figure}[hbt!]
\begin{center}

\tikzset{every picture/.style={line width=0.75pt}} 



\end{center}
\captionsetup{skip=2pt}
\caption{Orientations $\gamma_3,\ldots,\gamma_6$, which give $v_1$ and $v_5$ as sinks.}
\label{fig:net-orientation3}
\end{figure}

One can check that we have now covered all cases where we have two sinks.
Therefore, the sum in the right-hand side of equation \eqref{conjecture-equation} is equal to $6$.
Finally, since we are in the unweighted case, $(-1)^{d-n}=1$, so the right-hand side of equation \eqref{conjecture-equation} is $6$, which coincides with $\sigma_{2,2}(X_G)$.

\subsection{Theoretical Evidence Supporting Conjecture \ref{conjecture}} \label{conjecture-special-cases}

As further evidence supporting the conjecture, we present proofs of two special cases.

\subsubsection{Graphs with No Edges}

\begin{definition}\label{def:samuj}

Let $a$ and $\mu$ be positive integers and let $j$ be a non-negative integer. Let $C_a$ be a directed cycle with $a$ vertices $v_1, \dots, v_a$ and directed edges $v_1v_2, \dots, v_{a-1}v_a, v_av_1$.
We define
\[Z_{a,\mu,j}:=\{W:W\subseteq V(C_a), |W|=\mu, \text{and }C_a-W\text{ has } j\text{ components}\}.\]

\end{definition}

So $Z_{a,\mu,j}$ is the set of ways to color the beads of a labelled $a$-bead necklace either red or blue such that $\mu$ of the beads are red, and the removal of red beads produces $j$ blue strings.
For example, in Figure \ref{fig:graph-no-edge-proof-example}, we see that $|Z_{6,3,3}|=2$.
One can also verify that $|Z_{6,3,1}|=6$, $|Z_{6,3,2}|=12$, and $|Z_{6,3,j}|=0$ for $j\geq 4$.

\begin{figure}[hbt]
\begin{center}
\tikzset{every picture/.style={line width=0.75pt}} 
\begin{tikzpicture}[x=0.75pt,y=0.75pt,yscale=-1,xscale=1]

\draw  [color={rgb, 255:red, 0; green, 0; blue, 0 }  ,draw opacity=1 ][fill={rgb, 255:red, 0; green, 0; blue, 0 }  ,fill opacity=1 ] (18.58,41.11) .. controls (18.58,39.09) and (20.13,37.45) .. (22.04,37.45) .. controls (23.95,37.45) and (25.5,39.09) .. (25.5,41.11) .. controls (25.5,43.13) and (23.95,44.77) .. (22.04,44.77) .. controls (20.13,44.77) and (18.58,43.13) .. (18.58,41.11) -- cycle ;
\draw  [color={rgb, 255:red, 208; green, 2; blue, 27 }  ,draw opacity=1 ][fill={rgb, 255:red, 208; green, 2; blue, 27 }  ,fill opacity=1 ] (48.58,21.11) .. controls (48.58,19.09) and (50.13,17.45) .. (52.04,17.45) .. controls (53.95,17.45) and (55.5,19.09) .. (55.5,21.11) .. controls (55.5,23.13) and (53.95,24.77) .. (52.04,24.77) .. controls (50.13,24.77) and (48.58,23.13) .. (48.58,21.11) -- cycle ;
\draw  [color={rgb, 255:red, 0; green, 0; blue, 0 }  ,draw opacity=1 ][fill={rgb, 255:red, 0; green, 0; blue, 0 }  ,fill opacity=1 ] (78.58,41.11) .. controls (78.58,39.09) and (80.13,37.45) .. (82.04,37.45) .. controls (83.95,37.45) and (85.5,39.09) .. (85.5,41.11) .. controls (85.5,43.13) and (83.95,44.77) .. (82.04,44.77) .. controls (80.13,44.77) and (78.58,43.13) .. (78.58,41.11) -- cycle ;
\draw  [color={rgb, 255:red, 208; green, 2; blue, 27 }  ,draw opacity=1 ][fill={rgb, 255:red, 208; green, 2; blue, 27 }  ,fill opacity=1 ] (18.58,69.11) .. controls (18.58,67.09) and (20.13,65.45) .. (22.04,65.45) .. controls (23.95,65.45) and (25.5,67.09) .. (25.5,69.11) .. controls (25.5,71.13) and (23.95,72.77) .. (22.04,72.77) .. controls (20.13,72.77) and (18.58,71.13) .. (18.58,69.11) -- cycle ;
\draw  [color={rgb, 255:red, 0; green, 0; blue, 0 }  ,draw opacity=1 ][fill={rgb, 255:red, 0; green, 0; blue, 0 }  ,fill opacity=1 ] (49.58,87.11) .. controls (49.58,85.09) and (51.13,83.45) .. (53.04,83.45) .. controls (54.95,83.45) and (56.5,85.09) .. (56.5,87.11) .. controls (56.5,89.13) and (54.95,90.77) .. (53.04,90.77) .. controls (51.13,90.77) and (49.58,89.13) .. (49.58,87.11) -- cycle ;
\draw  [color={rgb, 255:red, 208; green, 2; blue, 27 }  ,draw opacity=1 ][fill={rgb, 255:red, 208; green, 2; blue, 27 }  ,fill opacity=1 ] (78.58,69.11) .. controls (78.58,67.09) and (80.13,65.45) .. (82.04,65.45) .. controls (83.95,65.45) and (85.5,67.09) .. (85.5,69.11) .. controls (85.5,71.13) and (83.95,72.77) .. (82.04,72.77) .. controls (80.13,72.77) and (78.58,71.13) .. (78.58,69.11) -- cycle ;
\draw    (54.5,23) -- (78.83,38.91) ;
\draw [shift={(80.5,40)}, rotate = 213.18] [color={rgb, 255:red, 0; green, 0; blue, 0 }  ][line width=0.75]    (10.93,-3.29) .. controls (6.95,-1.4) and (3.31,-0.3) .. (0,0) .. controls (3.31,0.3) and (6.95,1.4) .. (10.93,3.29)   ;
\draw    (82.5,45) -- (82.08,63.45) ;
\draw [shift={(82.04,65.45)}, rotate = 271.29] [color={rgb, 255:red, 0; green, 0; blue, 0 }  ][line width=0.75]    (10.93,-3.29) .. controls (6.95,-1.4) and (3.31,-0.3) .. (0,0) .. controls (3.31,0.3) and (6.95,1.4) .. (10.93,3.29)   ;
\draw    (79.5,71.5) -- (58.21,84.46) ;
\draw [shift={(56.5,85.5)}, rotate = 328.67] [color={rgb, 255:red, 0; green, 0; blue, 0 }  ][line width=0.75]    (10.93,-3.29) .. controls (6.95,-1.4) and (3.31,-0.3) .. (0,0) .. controls (3.31,0.3) and (6.95,1.4) .. (10.93,3.29)   ;
\draw    (50.5,86.5) -- (27.18,71.58) ;
\draw [shift={(25.5,70.5)}, rotate = 32.62] [color={rgb, 255:red, 0; green, 0; blue, 0 }  ][line width=0.75]    (10.93,-3.29) .. controls (6.95,-1.4) and (3.31,-0.3) .. (0,0) .. controls (3.31,0.3) and (6.95,1.4) .. (10.93,3.29)   ;
\draw    (22.5,65.5) -- (22.08,46.77) ;
\draw [shift={(22.04,44.77)}, rotate = 88.72] [color={rgb, 255:red, 0; green, 0; blue, 0 }  ][line width=0.75]    (10.93,-3.29) .. controls (6.95,-1.4) and (3.31,-0.3) .. (0,0) .. controls (3.31,0.3) and (6.95,1.4) .. (10.93,3.29)   ;
\draw    (24.5,37.5) -- (46.84,22.61) ;
\draw [shift={(48.5,21.5)}, rotate = 146.31] [color={rgb, 255:red, 0; green, 0; blue, 0 }  ][line width=0.75]    (10.93,-3.29) .. controls (6.95,-1.4) and (3.31,-0.3) .. (0,0) .. controls (3.31,0.3) and (6.95,1.4) .. (10.93,3.29)   ;
\draw  [color={rgb, 255:red, 208; green, 2; blue, 27 }  ,draw opacity=1 ][fill={rgb, 255:red, 208; green, 2; blue, 27 }  ,fill opacity=1 ] (159.58,41.11) .. controls (159.58,39.09) and (161.13,37.45) .. (163.04,37.45) .. controls (164.95,37.45) and (166.5,39.09) .. (166.5,41.11) .. controls (166.5,43.13) and (164.95,44.77) .. (163.04,44.77) .. controls (161.13,44.77) and (159.58,43.13) .. (159.58,41.11) -- cycle ;
\draw  [color={rgb, 255:red, 0; green, 0; blue, 0 }  ,draw opacity=1 ][fill={rgb, 255:red, 0; green, 0; blue, 0 }  ,fill opacity=1 ] (189.58,21.11) .. controls (189.58,19.09) and (191.13,17.45) .. (193.04,17.45) .. controls (194.95,17.45) and (196.5,19.09) .. (196.5,21.11) .. controls (196.5,23.13) and (194.95,24.77) .. (193.04,24.77) .. controls (191.13,24.77) and (189.58,23.13) .. (189.58,21.11) -- cycle ;
\draw  [color={rgb, 255:red, 208; green, 2; blue, 27 }  ,draw opacity=1 ][fill={rgb, 255:red, 208; green, 2; blue, 27 }  ,fill opacity=1 ] (219.58,41.11) .. controls (219.58,39.09) and (221.13,37.45) .. (223.04,37.45) .. controls (224.95,37.45) and (226.5,39.09) .. (226.5,41.11) .. controls (226.5,43.13) and (224.95,44.77) .. (223.04,44.77) .. controls (221.13,44.77) and (219.58,43.13) .. (219.58,41.11) -- cycle ;
\draw  [color={rgb, 255:red, 0; green, 0; blue, 0 }  ,draw opacity=1 ][fill={rgb, 255:red, 0; green, 0; blue, 0 }  ,fill opacity=1 ] (159.58,69.11) .. controls (159.58,67.09) and (161.13,65.45) .. (163.04,65.45) .. controls (164.95,65.45) and (166.5,67.09) .. (166.5,69.11) .. controls (166.5,71.13) and (164.95,72.77) .. (163.04,72.77) .. controls (161.13,72.77) and (159.58,71.13) .. (159.58,69.11) -- cycle ;
\draw  [color={rgb, 255:red, 208; green, 2; blue, 27 }  ,draw opacity=1 ][fill={rgb, 255:red, 208; green, 2; blue, 27 }  ,fill opacity=1 ] (190.58,87.11) .. controls (190.58,85.09) and (192.13,83.45) .. (194.04,83.45) .. controls (195.95,83.45) and (197.5,85.09) .. (197.5,87.11) .. controls (197.5,89.13) and (195.95,90.77) .. (194.04,90.77) .. controls (192.13,90.77) and (190.58,89.13) .. (190.58,87.11) -- cycle ;
\draw  [color={rgb, 255:red, 0; green, 0; blue, 0 }  ,draw opacity=1 ][fill={rgb, 255:red, 0; green, 0; blue, 0 }  ,fill opacity=1 ] (219.58,69.11) .. controls (219.58,67.09) and (221.13,65.45) .. (223.04,65.45) .. controls (224.95,65.45) and (226.5,67.09) .. (226.5,69.11) .. controls (226.5,71.13) and (224.95,72.77) .. (223.04,72.77) .. controls (221.13,72.77) and (219.58,71.13) .. (219.58,69.11) -- cycle ;
\draw    (195.5,23) -- (219.83,38.91) ;
\draw [shift={(221.5,40)}, rotate = 213.18] [color={rgb, 255:red, 0; green, 0; blue, 0 }  ][line width=0.75]    (10.93,-3.29) .. controls (6.95,-1.4) and (3.31,-0.3) .. (0,0) .. controls (3.31,0.3) and (6.95,1.4) .. (10.93,3.29)   ;
\draw    (223.5,45) -- (223.08,63.45) ;
\draw [shift={(223.04,65.45)}, rotate = 271.29] [color={rgb, 255:red, 0; green, 0; blue, 0 }  ][line width=0.75]    (10.93,-3.29) .. controls (6.95,-1.4) and (3.31,-0.3) .. (0,0) .. controls (3.31,0.3) and (6.95,1.4) .. (10.93,3.29)   ;
\draw    (220.5,71.5) -- (199.21,84.46) ;
\draw [shift={(197.5,85.5)}, rotate = 328.67] [color={rgb, 255:red, 0; green, 0; blue, 0 }  ][line width=0.75]    (10.93,-3.29) .. controls (6.95,-1.4) and (3.31,-0.3) .. (0,0) .. controls (3.31,0.3) and (6.95,1.4) .. (10.93,3.29)   ;
\draw    (191.5,86.5) -- (168.18,71.58) ;
\draw [shift={(166.5,70.5)}, rotate = 32.62] [color={rgb, 255:red, 0; green, 0; blue, 0 }  ][line width=0.75]    (10.93,-3.29) .. controls (6.95,-1.4) and (3.31,-0.3) .. (0,0) .. controls (3.31,0.3) and (6.95,1.4) .. (10.93,3.29)   ;
\draw    (163.5,65.5) -- (163.08,46.77) ;
\draw [shift={(163.04,44.77)}, rotate = 88.72] [color={rgb, 255:red, 0; green, 0; blue, 0 }  ][line width=0.75]    (10.93,-3.29) .. controls (6.95,-1.4) and (3.31,-0.3) .. (0,0) .. controls (3.31,0.3) and (6.95,1.4) .. (10.93,3.29)   ;
\draw    (165.5,37.5) -- (187.84,22.61) ;
\draw [shift={(189.5,21.5)}, rotate = 146.31] [color={rgb, 255:red, 0; green, 0; blue, 0 }  ][line width=0.75]    (10.93,-3.29) .. controls (6.95,-1.4) and (3.31,-0.3) .. (0,0) .. controls (3.31,0.3) and (6.95,1.4) .. (10.93,3.29)   ;

\draw (44,5.4) node [anchor=north west][inner sep=0.75pt]  [color={rgb, 255:red, 208; green, 2; blue, 27 }  ,opacity=1 ]  {$v_{1}$};
\draw (85,28.4) node [anchor=north west][inner sep=0.75pt]    {$v_{2}$};
\draw (87,64.4) node [anchor=north west][inner sep=0.75pt]  [color={rgb, 255:red, 208; green, 2; blue, 27 }  ,opacity=1 ]  {$v_{3}$};
\draw (47,94.4) node [anchor=north west][inner sep=0.75pt]    {$v_{4}$};
\draw (3,69.4) node [anchor=north west][inner sep=0.75pt]  [color={rgb, 255:red, 208; green, 2; blue, 27 }  ,opacity=1 ]  {$v_{5}$};
\draw (4,25.4) node [anchor=north west][inner sep=0.75pt]    {$v_{6}$};
\draw (185,5.4) node [anchor=north west][inner sep=0.75pt]    {$v_{1}$};
\draw (225,29.4) node [anchor=north west][inner sep=0.75pt]  [color={rgb, 255:red, 208; green, 2; blue, 27 }  ,opacity=1 ]  {$v_{2}$};
\draw (228,64.4) node [anchor=north west][inner sep=0.75pt]    {$v_{3}$};
\draw (188,94.4) node [anchor=north west][inner sep=0.75pt]  [color={rgb, 255:red, 208; green, 2; blue, 27 }  ,opacity=1 ]  {$v_{4}$};
\draw (144,69.4) node [anchor=north west][inner sep=0.75pt]    {$v_{5}$};
\draw (145,25.4) node [anchor=north west][inner sep=0.75pt]  [color={rgb, 255:red, 208; green, 2; blue, 27 }  ,opacity=1 ]  {$v_{6}$};
\draw (15,114.4) node [anchor=north west][inner sep=0.75pt]    {$Z_{6,3,3} =\{\{v_{1} ,v_{3} ,v_{5}\} ,\{v_{2} ,v_{4} ,v_{6}\}\}$};

\end{tikzpicture}
\end{center}
\captionsetup{skip=2pt}
\caption{Example of $Z_{6,3,3}$.}
\label{fig:graph-no-edge-proof-example}
\end{figure}

\begin{restatable}{lemma}{lemmanoedge}
\label{no-edge-lemma}
Let $a\geq 3$, $\mu\geq 2$, $j\geq 1$ be integers. Then
\begin{equation}\label{eq:samuj}
|Z_{a,\mu,j}|=|Z_{a-1,\mu-1,j}|+\sum_{i=\mu+j-2}^{a-2}|Z_{i,\mu-1,j-1}|.\end{equation}
\end{restatable}

The proof of this lemma is largely computational and is given in Appendix \ref{appendix}.

With this auxiliary lemma, we may prove Conjecture \ref{conjecture} for graphs with no edges.

\begin{theorem}
\label{no-edge-thm}
Let $(G,\omega)$ be a set-weighted graph with $n$ vertices, total weight $d$, and no edges. 

Write $X_{(G,\omega)}=\sum_{\lambda\vdash d}c_{\lambda}e_{\lambda}$.
Let $\mu \leq d$ be an integer (viewed as a partition with a single part), and
fix $j \in \{0, \dots, d-\mu\}$. Then 
\begin{equation}
\label{no-edge-conjecture-equation}
\sigma_{\mu,j}\bp{X_{(G,\omega)}}=(-1)^{d-n}\sum_{\substack{\wts(\gamma,S)=(\mu,j) \\ S\text{ admissible}}}\sgn(\gamma,S),
\end{equation}
summed over all acyclic orientations $\gamma$ of $G$ and all $\gamma$-admissible generalized $2$-step weight maps $S$ of $G$ such that $\wts(\gamma,S)=(\mu,j)$.

In particular, the formula of Conjecture \ref{conjecture} holds for graphs with no edges.

\end{theorem}

\begin{proof}

We first claim that for every choice of $a, \mu \geq 1$ and $j \geq 0$,
\begin{equation}
  \label{single-vertex}
  \sigma_{\mu,j}(p_a)=(-1)^{a+\mu}|Z_{a,\mu,j}|.
\end{equation}
We proceed by showing that the left- and right-hand sides of this equation satisfy the same base cases and the same recurrence relation. For each $a,\mu,j$ as above, let $T_{a,\mu,j}=(-1)^{a+\mu}|Z_{a,\mu,j}|$, the quantity on the right-hand side of \eqref{single-vertex}.

We will make use of Newton's identity \cite{mac}:
\begin{equation}\label{newton}
    p_a=(-1)^{a-1}ae_a+\sum_{i=1}^{a-1}(-1)^{a-1+i}e_{a-i}p_i.
\end{equation}

First, for the base cases, note that $p_1 = e_1$ and $p_2 = e_{11}-2e_2$, so \[\sigma_{1,0}(p_1)=1=T_{1,1,0},\quad \sigma_{1,0}(p_2)=0=T_{2,1,0}, \quad \sigma_{1,1}(p_2)=-2=T_{2,1,1}, \quad \sigma_{2,0}(p_2) = 1 = T_{2,2,0}\]
and it is easy to check that for any other choice of $a \in \{1,2\}, \mu \geq 1, j \geq 0$ both sides of \eqref{single-vertex} evaluate to $0$.

If $j=0$, then both sides of \eqref{single-vertex} are $0$ except when $\mu=a$, in which case we may verify from \eqref{newton} that $\sigma_{a,0}(p_a) = 1 = T_{a,a,0}$.

If $\mu = 1$, then both sides of \eqref{single-vertex} are $0$ unless $a=1$ and $j=0$, which was checked above, or $j=1$, in which case we may verify from \eqref{newton} that $\sigma_{1,1}(p_a) = (-1)^{a-1}a = T_{a,1,1}$. This establishes all necessary base cases.

For the recursive step, by Lemma \ref{no-edge-lemma}, we see that for $a\geq 3,\mu \geq 2,j \geq 1$ we have 
\begin{align}
  T_{a,\mu,j} &= (-1)^{a+\mu}\lrp{|Z_{a-1,\mu-1,j}|+\sum_{i=\mu+j-2}^{a-2}|Z_{i,\mu-1,j-1}|} \nonumber \\
  &= (-1)^{a+\mu-2}|Z_{a-1,\mu-1,j}|+\sum_{i=\mu+j-2}^{a-2}(-1)^{a-1+i}\cdot(-1)^{i+\mu-1}|Z_{i,\mu-1,j-1}| \nonumber \\
  &= T_{a-1,\mu-1,j}+\sum_{i=\mu+j-2}^{a-2}(-1)^{a-1+i}T_{i,\mu-1,j-1}.
  \label{T-recurrence}
\end{align}

Furthermore, when $a\geq 3$, $\mu\geq 2$, and $j\geq 1$, it follows from applying $\sigma_{\mu,j}$ to both sides of \eqref{newton} that
\begin{equation}
\label{sigma-recurrence}
\sigma_{\mu,j}(p_a)=\sigma_{\mu-1,j}(p_{a-1})+\sum_{i=\mu+j-2}^{a-2}(-1)^{a-1+i}\sigma_{\mu-1,j-1}(p_i).
\end{equation}
where we may determine that $\sigma_{\mu-1,j-1}(p_i) = 0$ for $i < \mu+j-2$ since by definition $\sigma_{\mu-1,j-1}$ will only give nonzero evaluation on $e$-basis terms that are homogeneous with degree at least $2(j-1)+(\mu-1-(j-1)) = (j-1)+(\mu-1) = \mu+j-2$.

Thus, combining \eqref{T-recurrence} and \eqref{sigma-recurrence}, we have shown that both sides of \eqref{single-vertex} have the same recurrence relation. Since they also have the same base cases, both sides are equal for all relevant $a,\mu,j$.

Now, suppose $(G,\omega)$ has vertices of weights $\lambda_1\geq\cdots\geq \lambda_n$.
We first consider the left-hand side of \eqref{no-edge-conjecture-equation}.
Since $G$ has no edges, from the definition of $\sigma$ we obtain 
\[\sigma_{\mu,j}\bp{X_{(G,\omega)}}=\sigma_{\mu,j}(p_{(\lambda_1,\ldots,\lambda_n)})=\sigma_{\mu,j}(p_{\lambda_1}\cdots p_{\lambda_n})=\sum_{\substack{(k_1,\ldots,k_n) \\ k_1+\cdots+k_n=\mu}}\sum_{\substack{(\ell_1,\ldots,\ell_n) \\ \ell_1+\cdots+\ell_n=j}}\sigma_{k_1,\ell_1}(p_{\lambda_1})\cdots \sigma_{k_n,\ell_n}(p_{\lambda_n}),\]
summed over all tuples $(k_1,\ldots,k_n)$ of positive integers such that $k_1+\cdots+k_n=\mu$, and all tuples $(\ell_1,\ldots,\ell_n)$ of non-negative integers with $\ell_1+\cdots+\ell_n=j$.

We then consider the right-hand side of \eqref{no-edge-conjecture-equation}.
Since $G$ has no edges, there is only one acyclic orientation, i.e. the empty orientation.
Then by definition, we have 
\begin{align*}
  (-1)^{d-n}\sum_{\substack{\wts(\gamma,S)=(\mu,j) \\ S\text{ admissible}}}\sgn(\gamma,S) 
  &= (-1)^{d-n}\sum_{\substack{(k_1,\ldots,k_n) \\ k_1+\cdots+k_n=\mu}}\sum_{\substack{(\ell_1,\ldots,\ell_n) \\ \ell_1+\cdots+\ell_n=j}}(-1)^{\mu-n}|Z_{\lambda_1,k_1,\ell_1}|\cdots|Z_{\lambda_n,k_n,\ell_n}|\end{align*}
  where the sum runs over the same tuples as in the previous equation. 
  Applying \eqref{single-vertex} and comparing to the previous equation, it is straightforward to verify that this is equal to
  \begin{align*}
  \sigma_{\mu,j}\bp{X_{(G,\omega)}},
\end{align*}
and this finishes the proof.
\end{proof}

\subsubsection{Graphs with Two Vertices Connected by an Edge}

In this section we will continue building upon ideas from the previous section.

\begin{restatable}{lemma}{lemmaoneedge}
\label{lem:one-edge}
Let $a$ and $b$ be positive integers such that $a\leq b$, let $\mu$ be a positive integer such that $\mu \leq a$ or $b \leq \mu \leq a+b$, and let $j \in \{0, \dots, a+b-\mu\}$, where $j = 0$ is only allowed if $\mu = a+b$.
Then 
\begin{equation}
\label{one-edge-lemma}
-|Z_{a+b,\mu,j}|+\sum_{k=1}^{\mu-1}\sum_{\ell=0}^{j}|Z_{a,k,\ell}|\cdot |Z_{b,\mu-k,j-\ell}|=
\begin{cases}
(-1)^j\cdot 2\binom{a}{j} & \text{ if }\mu=a=b  \\
(-1)^j\binom{a}{j}-|Z_{b,\mu,j}| & \text{ if }a<b\text{ and }\mu\in\{a,b\} \\ 
-|Z_{a,\mu,j}|-|Z_{b,\mu,j}| & \text{ if }\mu<a\text{ or }\mu>b 
\end{cases}.
\end{equation}
\end{restatable}

We need this lemma to understand what happens upon applying deletion-contraction to a graph with two vertices and one edge and expanding the result using Theorem \ref{no-edge-thm} on graphs with no edges. As with Lemma \ref{no-edge-lemma}, the proof is given in Appendix \ref{appendix}.

With this auxiliary lemma in hand, we can demonstrate that Conjecture \ref{conjecture} holds on graphs with two vertices and one edge.

\begin{theorem}
\label{one-edge-thm}
Let $(G,\omega)$ be a set-weighted graph such that $V(G)=\{v_1,v_2\}$ and $E(G)=\{v_1v_2\}$, with $w(v_1) = a \leq b = w(v_2)$. Let $\mu$ be a positive integer such that $\mu \leq a$ or $b \leq \mu \leq a+b$. Fix $j \in \{0, \dots, a+b-\mu\}$, where $j=0$ is possible if and only if $\mu = a+b$. Then \begin{equation}
\label{one-edge-theorem-equation}
\sigma_{\mu,j}\bp{X_{(G,\omega)}}=(-1)^{a+b-2}\sum_{\substack{\wts(\gamma,S)=(\mu,j) \\ S\text{ admissible}}}\sgn(\gamma,S),
\end{equation}
summed over pairs of (acyclic) orientations $\gamma$ of $G$ and all $\gamma$-admissible generalized $2$-step weight maps $S$ of $G$ such that $\wts(\gamma,S)=(\mu,j)$.

In particular, Conjecture \ref{conjecture} holds for graphs with two vertices with one edge between them.
\end{theorem}

\begin{proof}

Let $e=v_1v_2$. We first evaluate the left-hand side of \eqref{one-edge-theorem-equation}.
Using the deletion-contraction relation (Lemma \ref{lem:delconset}), we have 
\[\sigma_{\mu,j}\bp{X_{(G,\omega)}}=\sigma_{\mu,j}\bp{X_{(G\setminus e,\omega)}}-\sigma_{\mu,j}\bp{X_{(G/e,\omega/e)}}.\]
Since $G\setminus e$ and $G/e$ both have no edges, we can apply Theorem \ref{no-edge-thm} and get 
\begin{align*}
  \sigma_{\mu,j}\bp{X_{(G,\omega)}}
  &=(-1)^{a+b}\sum_{\substack{\wts(G\setminus e,\gamma,S)=(\mu,j) \\ S\text{ admissible}}}\sgn(G\setminus e,\gamma,S)-(-1)^{a+b-1}\sum_{\substack{\wts(G/e,\gamma,S)=(\mu,j) \\ S\text{ admissible}}}\sgn(G/e,\gamma,S) \\
  &= (-1)^{a+b}\cdot(-1)^{\mu-2}\sum_{k=1}^{\mu-1}\sum_{\ell=0}^{j}|Z_{a,k,\ell}|\cdot|Z_{b,\mu-k,j-\ell}|-(-1)^{a+b-1}\cdot(-1)^{\mu-1}|Z_{a+b,\mu,j}| \\
  &= (-1)^{a+b+\mu}\lrp{-|Z_{a+b,\mu,j}|+\sum_{k=1}^{\mu-1}\sum_{\ell=0}^{j}|Z_{a,k,\ell}|\cdot|Z_{b,\mu-k,j-\ell}|}.
\end{align*}
Applying Lemma \ref{lem:one-edge}, we thus have
\[
\sigma_{\mu,j}\bp{X_{(G,\omega)}} = 
\begin{cases}
  (-1)^{a+b+\mu}(-1)^{j}\cdot 2\binom{a}{j} & \text{ if }\mu=a=b  \\
  (-1)^{a+b+\mu}(-1)^{j}\binom{a}{j}-(-1)^{a+b+\mu}|Z_{b,\mu,j}| & \text{ if }a<b\text{ and }\mu\in\{a,b\} \\ 
  -(-1)^{a+b+\mu}|Z_{a,\mu,j}|-(-1)^{a+b+\mu}|Z_{b,\mu,j}| & \text{ if }\mu<a\text{ or }\mu>b 
\end{cases}
\]

On the other hand, expanding the right-hand side of \eqref{one-edge-theorem-equation}, we note that there are two acyclic orientations for $G$, one with $v_1\to v_2$ and the other with $v_2\to v_1$. 
Is straightforward to verify by casework that \eqref{one-edge-theorem-equation} evaluates to
\begin{align*}
&\qquad (-1)^{a+b} \sum_{\substack{\wts(G,\gamma,S)=(\mu,j) \\ S\text{ admissible}}}\sgn(G,\gamma,S) \\
&=\begin{cases}
  (-1)^{a+b}(-1)^{(\mu-1)+(j-1)}\cdot 2\binom{a}{j} & \text{ if }\mu=a=b  \\
  (-1)^{a+b}\bp{(-1)^{(\mu-1)+(j-1)}\binom{a}{j}+(-1)^{\mu-1}|Z_{b,\mu,j}|} & \text{ if }a<b\text{ and }\mu\in\{a,b\} \\ 
  (-1)^{a+b}\bp{(-1)^{\mu-1}|Z_{a,\mu,j}|+(-1)^{\mu-1}|Z_{b,\mu,j}|} & \text{ if }\mu<a\text{ or }\mu>b 
\end{cases}
\end{align*}
Note that for the middle case, we need to apply parts (3) and (4) of Definition \ref{def:drop} and count the $j$ second-level sinks from those of the vertex of lower weight. This concludes the proof.
\end{proof}

\section{Concluding Remarks}

The introduction of Conjecture \ref{conjecture} is something of a break from current trends in the research of $e$-basis expansions of chromatic symmetric functions. As far as the authors are aware, the conjectured weight-drop phenomenon is previously unknown. In particular, in light of the discussion in Section 4.2 (and as noted in the introduction), the authors believe that further work in this direction could lead to a formula in which every integer partition $\mu$ is $s$-allowable in unweighted claw-free graphs, and so could potentially give any individual $e$-basis coefficient for such graphs.

It would be interesting to compare Conjecture \ref{conjecture} with Hikita's interpretation of $e$-basis coefficients in unit interval graphs \cite{hikita} to see if it is possible to extend Hikita's groundbreaking work to either the more general class of claw-free graphs, or to vertex-weighted unit interval graphs.
\section{Acknowledgments}

The authors would like to thank Sophie Spirkl for helpful discussions and for the simple proof of Lemma \ref{lem:clawfree}. We also thank the anonymous referees for their helpful comments.

\bibliographystyle{plain}
\bibliography{bib}

\newpage

\appendix

\section{Proof of Technical Lemmas in Section \ref{conjecture-special-cases}}\label{appendix}

In this appendix we prove Lemmas \ref{no-edge-lemma} and \ref{lem:one-edge}. We first slightly restate a relevant definition and define a new term that will be useful.

\begin{definition}[Equivalent Restatement of Definition \ref{def:samuj} plus definition of \textbf{block}]

Let $a$ and $\mu$ be positive integers and let $j$ be a non-negative integer. Let $C_a$ be a directed cycle with $a$ vertices $v_1, \dots, v_a$ and directed edges $v_1v_2, \dots, v_{a-1}v_a, v_av_1$.
We define
\[Z_{a,\mu,j}:=\{(a,W):W\subseteq V(C_a), |W|=\mu, \text{and }C_a-W\text{ has } j\text{ components}\}.\]

Given $W \subseteq V(C_a)$, each maximal subset $W^* \subseteq W$ such that $C_a|_{W^*}$ is connected is called a \textbf{block} of $W$.

\end{definition}

So as before $Z_{a,\mu,j}$ is the set of ways to color the beads of a labelled $a$-bead necklace either red or blue such that $\mu$ of the beads are red, and the removal of red beads produces $j$ blue strings (so also when $j \neq 0$, there are $j$ blocks in each valid choice of $W \subseteq V(C_a)$). Note the difference with this new definition from that presented previously: the presence of $a$ in $(a,W)$. While formally since $a$ is specified in $Z_{a,\mu,j}$ it does not need to be specified in $(a,W) \in Z_{a,\mu,j}$, in the following proofs we will be modifying $W$ while transitioning between sets with different values of $a$, and so including the cycle size with $W$ will make these arguments easier to follow.

\lemmanoedge*

\begin{proof}[Proof of Lemma \ref{no-edge-lemma}]
Note that if $\mu+j>a$, then both sides of the equation are $0$. For all $a\geq 3$ and $\mu+j\leq a$, define 
\begin{align*}
  Z_{a,\mu,j}(1,1) &:= \{(a,W)\in Z_{a,\mu,j}:v_1\in W, v_a\in W\}, \\
  Z_{a,\mu,j}(1,0) &:= \{(a,W)\in Z_{a,\mu,j}:v_1\in W, v_a\notin W\}, \\
  Z_{a,\mu,j}(0,1) &:= \{(a,W)\in Z_{a,\mu,j}:v_1\notin W, v_a\in W\}, \\
  Z_{a,\mu,j}(0,0) &:= \{(a,W)\in Z_{a,\mu,j}:v_1\notin W, v_a\notin W\}.
\end{align*}

Thus, the arguments in the parenthesis are indicators for whether $v_a$ and $v_1$ are in $W$ respectively.
It is easy to see that $Z_{a,\mu,j}=Z_{a,\mu,j}(1,1)\sqcup Z_{a,\mu,j}(0,1)\sqcup Z_{a,\mu,j}(1,0)\sqcup Z_{a,\mu,j}(0,0)$ as a disjoint union. 

In what follows, we fix $a\geq 3$ and $\mu,j$ such that $\mu+j\leq a$.
We first demonstrate that four auxiliary equations hold by establishing bijections.

\subsection*{Case 1: $Z_{a,\mu,j}(1,1)$}

Let 
\[\varphi_{1,1}:Z_{a,\mu,j}(1,1)\to Z_{a-1,\mu-1,j}(1,0)\sqcup Z_{a-1,\mu-1,j}(1,1)\]
be given by $\varphi_{1,1}(a,W)=(a-1,W\setminus\{v_a\})$. 
Note that by definition, we have $v_a\in W$ so that $|W\setminus\{v_a\}|=\mu-1$. 
Further, we note that $C_{a-1}-(W\setminus\{v_a\})$ has $j$ components for all $(a,W)\in Z_{a,\mu,j}$ since the block of $(a,W)$ containing $v_a$ and $v_1$ is not split. Also note that we have $v_1\in W\setminus\{v_a\}$, so the function's range is correctly given.

We claim that $\varphi_{1,1}$ is a bijection.
To show injectivity, note that if $(a,W_1),(a,W_2)\in Z_{a,\mu,j}$ such that $\varphi_{1,1}(a,W_1)=\varphi_{1,1}(a,W_2)$, then $W_1\setminus\{v_a\}=W_2\setminus\{v_a\}$, which implies that $W_1=W_2$ since $a$ is in both $W_1$ and $W_2$ by definition. For surjectivity, note that for all $(a-1,W)\in Z_{a-1,\mu-1,j}(1,0)\cup Z_{a-1,\mu-1,j}(1,1)$, we may verify that $(a,W\cup\{v_a\})\in Z_{a,\mu,j}(1,1)$ and $\varphi_{1,1}(a,W\cup\{v_a\})=(a-1,W)$. 
Therefore, $\varphi_{1,1}$ is bijective and $|Z_{a,\mu,j}(1,1)|=|Z_{a-1,\mu-1,j}(1,0)|+|Z_{a-1,\mu-1,j}(1,1)|$.

\subsection*{Case 2: $Z_{a,\mu,j}(0,1)$}

For all $(a,W)\in Z_{a,\mu,j}(0,1)$, write $W=\{{v_{i_1}},\ldots,{v_{i_{\mu-1}}},v_a\}$, where 
$1<i_1<\cdots<i_{\mu-1}<a$. Note that if $i_{\mu-1} < a-1$, then between $v_{i_{\mu-1}}$ and $v_a$ lies one component of $C_a-W$, so there are $j-1$ components of $C_a-W$ induced by $\{v_1,\dots,v_{i_{\mu-1}-1}\}-W$. Furthermore, among these $i_{\mu-1}-1$ vertices are at least the remaining $\mu-2$ members of $W$ and $j-1$ other vertices, so when $i_{\mu-1} < a-1$ we have $i_{\mu-1}-1 \geq (\mu-2)+(j-1)$, or
$\mu+j-2\leq i_{\mu-1}$.
Define 
\[\varphi_{0,1}:Z_{a,\mu,j}(0,1)\to Z_{a-1,\mu-1,j}(0,1)\sqcup\ts{\bigsqcup\limits_{i=\mu+j-2}^{a-2}Z_{i,\mu-1,j-1}(0,1)}\]
by $\varphi_{0,1}(a,W)=(i_{\mu-1},W\setminus\{v_a\})$ for all $W\in Z_{a,\mu,j}(0,1)$. As before, clearly $|W \setminus \{v_a\}| = \mu-1$.

Note that if $i_{\mu-1}=a-1$, then since $v_a\in W$ there remain $j$ components in the image. Otherwise, as the component of $C_a-W$ between $v_{i_{\mu-1}}$ and $v_a$ is deleted, the image has $j-1$ components. Thus, the function's range is correctly given.

We claim that $\varphi_{0,1}$ is a bijection. As above, if $W_1$ and $W_2$ are such that $\varphi_{0,1}(a,W_1)=\varphi_{0,1}(a,W_2)$, then $W_1\setminus\{v_a\}=W_2\setminus\{v_a\}$, implying that $W_1=W_2$ and verifying that $\varphi_{0,1}$ is injective.

We then show that $\varphi_{0,1}$ is surjective. First suppose we have $(a-1,W)\in Z_{a-1,\mu-1,j}(0,1)$. Then we may verify that $(a,W\cup\{v_a\})\in Z_{a,\mu,j}(0,1)$ and $\varphi_{0,1}(a,W\cup\{v_a\})=(a-1,W)$ since ${a-1}\in W$.
Now suppose we have some $\mu+j-2\leq i\leq a-2$ and $(i,W)\in Z_{i,\mu-1,j-1}(0,1)$. In this case, we may again check that $(a,W\cup\{v_a\})\in Z_{a,\mu,j}(0,1)$ and $\varphi_{0,1}(a,W\cup\{v_a\})=(i,W)$ since $i$ is the second largest index such that $v_i\in W\cup\{a\}$. Therefore, we conclude that $\varphi_{0,1}$ is a bijection. It follows that 
$|Z_{a,\mu,j}(0,1)|=|Z_{a-1,\mu-1,j}(0,1)|+\sum_{i=\mu+j-2}^{a-2}|Z_{i,\mu-1,j-1}(0,1)|$.

\subsection*{Case 3: $Z_{a,\mu,j}(1,0)$}

For $(a,W)\in Z_{a,\mu,j}(1,0)$, write $W=\{v_1,v_{i_2},\ldots,v_{i_\mu}\}$ where $1 <i_2<\cdots<i_\mu<a$. As before, there is a component of $C_a-W$ consisting of the vertices between $v_{i_{\mu}}$ and $v_1$, so there are $j-1$ components of $C_a-W$ induced by $\{v_1,\dots,v_{i_{\mu}}\}$. Thus, among these $i_{\mu}$ vertices are at least the $\mu$ vertices of $W$ as well as $j-1$ others, so $i_{\mu} \geq \mu+j-1$, and $i_{\mu}-1 \geq \mu+j-2$.
Define 
\[\varphi_{1,0}:Z_{a,\mu,j}(1,0)\to\ts{\bigsqcup\limits_{i=\mu+j-2}^{a-2}\bp{Z_{i,\mu-1,j-1}(1,1)\sqcup Z_{i,\mu-1,j-1}(1,0)}}\]
such that $\varphi_{1,0}(a,W)=(i_\mu-1,W\setminus\{{i_\mu}\})$ for all $(a,W)\in Z_{a,\mu,j}(1,0)$. Clearly $|W\setminus \{v_{i_{\mu}}\}| = \mu-1$. From the above we indeed have $\mu+j-2\leq i_\mu-1\leq a-2$.
Furthermore, since $v_1\in W$, $v_a\notin W$, and $v_{i_{\mu}}\in W$, we know that $C_{i_\mu-1}-(W\setminus\{{v_{i_\mu}}\})$ has $j-1$ components, so the given range is correct.

We show that $\varphi_{1,0}$ is a bijection.
Suppose $W_1$ and $W_2$ are such that $\varphi_{1,0}(a,W_1)=\varphi_{1,0}(a,W_2)$. Write $W_1=\{v_1,v_{i_2},\ldots,{v_{i_\mu}}\}$ and $W_2=\{v_1,v_{j_2},\ldots,v_{j_\mu}\}$. Then $(i_\mu-1,W_1\setminus\{{i_\mu}\})=(j_\mu-1,W_2\setminus\{{j_\mu}\})$. Thus, $i_{\mu}-1 = j_{\mu}-1$ and $W_1 \setminus \{v_{i_{\mu}}\} = W_2 \setminus \{v_{j_{\mu}}\}$, so $W_1=W_2$ and $\varphi_{1,0}$ is injective.

For surjectivity, suppose we choose some $\mu+j-2\leq i\leq a-2$ and some $(i,W)\in Z_{i,\mu-1,j-1}(1,1)\cup Z_{i,\mu-1,j-1}(1,0)$.
Then since $i\leq a-2$, we know ${i+1}\neq a$, so we may verify that $C_a-(W\cup\{{v_{i+1}}\})$ has $j$ components.
It follows that $(a,W\cup\{v_{i+1}\})\in Z_{a,\mu,j}(1,0)$. It is then straightforward to check that $\varphi_{1,0}(a,W\cup\{v_{i+1}\})=(i,W)$. This proves that $\varphi_{1,0}$ is a surjection, and that 
$|Z_{a,\mu,j}(1,0)|=\sum_{i=\mu+j-2}^{a-2}\bp{|Z_{i,\mu-1,j-1}(1,1)|+|Z_{i,\mu-1,j-1}(1,0)|}$.

\subsection*{Case 4: $Z_{a,\mu,j}(0,0)$}

Again, for $(a,W)\in Z_{a,\mu,j}(0,0)$, we write $W=\{v_{i_1},\ldots,v_{i_\mu}\}$ where $1 < i_1<\cdots<i_\mu < a$. As before it is straightforward to verify that $i_{\mu}-1 \geq \mu+j-2$ (in fact the inequality is stronger in this case, but we do not need this).
We then define the map
\[\varphi_{0,0}:Z_{a,\mu,j}(0,0)\to Z_{a-1,\mu-1,j}(0,0)\sqcup\ts{\bigsqcup\limits_{i=\mu+j-2}^{a-2}Z_{i,\mu-1,j-1}(0,0)}\]
such that 
\[\varphi_{0,0}(a,W)=\begin{cases}
  (a-1, W\setminus\{v_{i_\mu}\}) & \text{ if } v_{i_\mu-1}\in W \\
  (i_\mu-1, W\setminus\{v_{i_\mu}\}) & \text{ if } v_{i_\mu-1}\notin W.
\end{cases}\]
Note that if $v_{i_\mu-1}\in W$, then $C_{a-1}-(W\setminus\{v_{i_\mu}\})$ also has $j$ components. 
Furthermore, since $v_a \notin W$, by definition either $i_{\mu} = a-1$ or $v_{a-1} \notin W$, and either way $v_{a-1} \notin W \setminus \{v_{i_{\mu}}\}$, which means that $\varphi_{0,0}(a,W)\in Z_{a-1,\mu-1,j}(0,0)$ when $v_{i_\mu-1}\in W$. 
If $v_{i_\mu-1}\notin W$, then $C_{i_\mu-1}-(W\setminus\{v_{i_\mu}\})$ has $j-1$ components, and thus $\varphi_{0,0}(a,W)\in Z_{i_\mu-1,\mu-1,j-1}(0,0)$, so the range of $\varphi_{0,0}$ is correctly given.

We claim that $\varphi_{0,0}$ is a bijection.
Suppose $W_1$ and $W_2$ are such that $\varphi_{0,0}(a,W_1)=\varphi_{0,0}(a,W_2)$. Write $W_1=\{v_{i_1},\ldots,v_{i_\mu}\}$ where $1<i_1<\cdots<i_\mu<a$, and write $W_2=\{v_{j_1},\ldots,v_{j_\mu}\}$ where $1<j_1<\cdots<j_\mu<a$.

If $v_{i_\mu-1}\in W_1$, then $\varphi_{0,0}(a,W_1)=(a-1,W_1\setminus\{v_{i_\mu}\})=\varphi_{0,0}(a,W_2)$. Since $j_\mu-1\neq a-1$, we must have $v_{j_\mu-1}\in W_2$. It follows that $\varphi_{0,0}(a,W_2)=(a-1,W_2\setminus\{v_{j_\mu}\})$, which means that $W_1\setminus\{v_{i_\mu}\}=W_2\setminus\{v_{j_\mu}\}$. 
Furthermore, since $i_{\mu-1} = i_{\mu}-1$ and $j_{\mu-1} = j_{\mu}-1$, we must have $i_\mu=j_\mu$, which shows that $W_1=W_2$.

If $v_{i_\mu-1}\notin W_1$, then $\varphi_{0,0}(a,W_1)=(i_\mu-1,W_1\setminus\{v_{i_\mu}\})=\varphi_{0,0}(a,W_2)$. Since $i_\mu-1\neq a-1$, we must have $v_{j_\mu-1}\notin W_2$. Thus, $(i_\mu-1,W_1\setminus\{v_{i_\mu}\})=(j_\mu-1,W_2\setminus\{v_{j_\mu}\})$, and as above it follows that $i_\mu=j_\mu$ and thus $W_1=W_2$. This shows that $\varphi_{0,0}$ is injective.

We show that $\varphi_{0,0}$ is surjective. First consider some $(a-1,W)\in Z_{a-1,\mu-1,j}(0,0)$. Write $W=\{v_{i_1},\ldots,v_{i_{\mu-1}}\}$ such that $i_1<\cdots<i_{\mu-1}$. Since $i_{\mu-1}\leq a-2$,
it is easy to see that $(a,W\cup\{v_{1+i_{\mu-1}}\})\in Z_{a,\mu,j}(0,0)$, and $\varphi_{0,0}(a,W\cup\{v_{1+i_{\mu-1}}\})=(a-1,W)$. 

Next, suppose we have some $i \in \{\mu+j-2, \dots, a-2\}$ and some $(i,W)\in Z_{i,\mu-1,j-1}(0,0)$.
Since $v_i\notin W$, we know that $C_a-(W\cup\{v_{i+1}\})$ has $j$ components.
Also since $i\leq a-2$, we have $v_a\notin W\cup\{v_{i+1}\}$, and thus $(a,W\cup\{v_{i+1}\})\in Z_{a,\mu,j}(0,0)$, and $\varphi_{0,0}(a,W\cup\{v_{i+1}\})=(i,W)$. Therefore, $\varphi_{0,0}$ is a surjection and thus a bijection. This means that 
$|Z_{a,\mu,j}(0,0)|=|Z_{a-1,\mu-1,j}(0,0)|+\sum_{i=\mu+j-2}^{a-2}|Z_{i,\mu-1,j-1}(0,0)|$.

Finally, combining all four cases together, we have 
\begin{align*}
  |Z_{a,\mu,j}| &= |Z_{a,\mu,j}(1,1)|+|Z_{a,\mu,j}(0,1)|+|Z_{a,\mu,j}(1,0)|+|Z_{a,\mu,j}(0,0)| \\
  &= |Z_{a-1,\mu-1,j}(1,0)|+|Z_{a-1,\mu-1,j}(1,1)| \\
  &\quad +|Z_{a-1,\mu-1,j}(0,1)|+\ts{\sum\limits_{i=\mu+j-2}^{a-2}|Z_{i,\mu-1,j-1}(0,1)|} \\
  &\quad +\ts{\sum\limits_{i=\mu+j-2}^{a-2}\bp{|Z_{i,\mu-1,j-1}(1,1)|+|Z_{i,\mu-1,j-1}(1,0)|}} \\
  &\quad +|Z_{a-1,\mu-1,j}(0,0)|+\ts{\sum\limits_{i=\mu+j-2}^{a-2}|Z_{i,\mu-1,j-1}(0,0)|} \\
  &=|Z_{a-1,\mu-1,j}|+\sum_{i=\mu+j-2}^{a-2}|Z_{i,\mu-1,j-1}|,
\end{align*}
as desired.
\end{proof}

We then prove Lemma \ref{lem:one-edge}.

\lemmaoneedge*

\begin{proof}[Proof of Lemma \ref{lem:one-edge}]

First, it is straightforward to verify that \eqref{one-edge-lemma} holds when $\mu = a+b$ and $j = 0$. Thus, for the remainder of this proof we may assume that $j$ is a positive integer.

We will introduce notation that will be used throughout this proof.

For $a \leq b$ positive integers, and $(i_1,i_2)\in\{0,1\}^2$, define
\[Z_{a,\mu,j}^{a+b}(i_1,i_2):=\{(a+b,W):(a,W)\in Z_{a,\mu,j}(i_1,i_2)\}.\]
and
\[Z_{b,\mu,j}^{a+b}(i_1,i_2):=\{(a+b,W):(b,\{w-a:w \in W\})\in Z_{b,\mu,j}(i_1,i_2)\}.\]

That is, $Z_{a,\mu,j}^{a+b}(i_1,i_2)$ is the set of choices of $W$ along an $a+b$ vertex cycle such that all $w \in W$ are among $\{v_1,\dots,v_a\}$, and that if we broke the cycle between $v_a$ and $v_{a+1}$ and also between $v_{a+b}$ and $v_1$, and reattached $v_1$ to $v_a$, the result would be a valid element of $Z_{a,\mu,j}(i_1,i_2)$. Analogously, $Z_{b,\mu,j}^{a+b}(i_1,i_2)$ is when instead all elements of $W$ are among $v_{a+1}, \dots, v_{a+b}$, and forming the cycle by breaking the same way and attaching $v_{a+1}$ to $v_{a+b}$ (and subtracting $a$ from all vertex labels) we get a valid element of $Z_{b,\mu,j}(i_1,i_2)$.

For $(i_1,i_2,i_3,i_4)\in\{0,1\}^4$, we define $Z_{a+b,\mu,j}(i_1,i_2,i_3,i_4)\subseteq Z_{a+b,\mu,j}$ as subsets $(a+b,W)$ where $i_1$ is an indicator for whether $v_1$ is selected, $i_2$ is an indicator for whether $v_a$ is selected, $i_3$ is an indicator for whether $v_{a+1}$ is selected, and $i_4$ is an indicator for whether $v_{a+b}$ is selected.

For example, we have $Z_{a+b,\mu,j}(0,1,1,0)=\{(a+b,W)\in Z_{a+b,\mu,j}:v_1,v_{a+b}\notin W,v_a,v_{a+1}\in W\}$.

Note that 
\begin{itemize}
    \item If $(i_1,i_2)=(0,0)$ and $(i_3,i_4) \neq (1,1)$, then $Z_{b,\mu,j}^{a+b}\subseteq Z_{a+b,\mu,j}(i_1,i_2,i_3,i_4)$.
    \item If $(i_1,i_2)\neq (0,0)$, then $Z_{b,\mu,j}^{a+b}\cap Z_{a+b,\mu,j}(i_1,i_2,i_3,i_4)=\varnothing$.
    \item If $(i_3,i_4)=(0,0)$ and $(i_1,i_2) \neq (1,1)$, then $Z_{a,\mu,j}^{a+b}\subseteq Z_{a+b,\mu,j}(i_1,i_2,i_3,i_4)$.
    \item If $(i_3,i_4)\neq (0,0)$, then $Z_{a,\mu,j}^{a+b}\cap Z_{a+b,\mu,j}(i_1,i_2,i_3,i_4)=\varnothing$.
\end{itemize} 

We first simplify the left-hand side of \eqref{one-edge-lemma}. 
Let 
\[P=\{(1,1,0,0),(1,0,0,1),(0,1,1,0),(0,0,1,1)\}.\]
We claim that for all $(i_1,i_2,i_3,i_4)\in\{0,1\}^4\setminus P$, there is a bijection between
\[
\bigsqcup_{k=1}^{\mu-1}\bigsqcup_{\ell=0}^{j}Z_{a,k,\ell}(i_1,i_2)\times Z_{b,\mu-k,j-\ell}(i_3,i_4)
\quad\text{and}\quad
Z_{a+b,\mu,j}(i_1,i_2,i_3,i_4)\setminus\bp{Z_{a,\mu,j}^{a+b}(i_1,i_2)\cup Z_{b,\mu,j}^{a+b}(i_3,i_4)}.
\]

Indeed, let $\varphi$ be a function from the left set to the right set given by 
\[\varphi\bp{(a,W_1),(b,W_2)}=\bp{a+b,W_1\cup\{v_{a+i}:i\in W_2\}}.\]
One can check manually case by case that this is a well-defined function from the left set to the right set (in particular, that the number of components induced by $((a,W_1),(b,W_2))$ is correct), and that it is a bijection.

For $(i_1,i_2,i_3,i_4)\in P$, it is straightforward to check that the above map $\phi$ gives a bijection between each of the following pairs:
\[
\bigsqcup_{k=1}^{\mu-1}\bigsqcup_{\ell=0}^{j}Z_{a,k,\ell}(1,1)\times Z_{b,\mu-k,j-\ell}(0,0)
\quad\text{and}\quad
Z_{a+b,\mu,j+1}(1,1,0,0)\setminus Z_{a,\mu,j}^{a+b}(1,1),
\]
\[
\bigsqcup_{k=1}^{\mu-1}\bigsqcup_{\ell=0}^{j}Z_{a,k,\ell}(1,0)\times Z_{b,\mu-k,j-\ell}(0,1)
\quad\text{and}\quad
Z_{a+b,\mu,j-1}(1,0,0,1),
\]
\[
\bigsqcup_{k=1}^{\mu-1}\bigsqcup_{\ell=0}^{j}Z_{a,k,\ell}(0,1)\times Z_{b,\mu-k,j-\ell}(1,0)
\quad\text{and}\quad
Z_{a+b,\mu,j-1}(0,1,1,0),
\]
\[
\bigsqcup_{k=1}^{\mu-1}\bigsqcup_{\ell=0}^{j}Z_{a,k,\ell}(0,0)\times Z_{b,\mu-k,j-\ell}(1,1)
\quad\text{and}\quad
Z_{a+b,\mu,j+1}(0,0,1,1)\setminus Z_{b,\mu,j}^{a+b}(1,1).
\]

Thus,
\begin{align*}
  &\quad \sum_{k=1}^{\mu-1}\sum_{\ell=0}^{j} |Z_{a,k,\ell}|\cdot|Z_{b,\mu-k,j-\ell}|
  = \sum_{k=1}^{\mu-1}\sum_{\ell=0}^{j} \lrp{\ts{\sum\limits_{i_1,i_2}|Z_{a,k,\ell}(i_1,i_2)|}}\cdot\lrp{\ts{\sum\limits_{i_3,i_4}|Z_{b,\mu-k,j-\ell}(i_3,i_4)|}}  \\
  &= \sum_{k=1}^{\mu-1}\sum_{\ell=0}^{j}\sum_{i_1,\ldots,i_4}
  |Z_{a,k,\ell}(i_1,i_2)|\cdot|Z_{b,\mu-k,j-\ell}(i_3,i_4)|  \\
  &= \sum_{i_1,\ldots,i_4}\sum_{k=1}^{\mu-1}\sum_{\ell=0}^{j}|Z_{a,k,\ell}(i_1,i_2)|\cdot|Z_{b,\mu-k,j-\ell}(i_3,i_4)|  \\
  &= \sum_{(i_1,\ldots,i_4)\notin P}\sum_{k=1}^{\mu-1}\sum_{\ell=0}^{j}|Z_{a,k,\ell}(i_1,i_2)|\cdot|Z_{b,\mu-k,j-\ell}(i_3,i_4)|  \\
  &\quad+ \sum_{(i_1,\ldots,i_4)\in P}\sum_{k=1}^{\mu-1}\sum_{\ell=0}^{j}|Z_{a,k,\ell}(i_1,i_2)|\cdot|Z_{b,\mu-k,j-\ell}(i_3,i_4)|  \\
  &= (|Z_{a+b,\mu,j}|-|Z_{a+b,\mu,j}(1,1,0,0)|-|Z_{a+b,\mu,j}(1,0,0,1)|-|Z_{a+b,\mu,j}(0,1,1,0)|-|Z_{a+b,\mu,j}(0,0,1,1)|
  \\
  &\quad-|Z_{a,\mu,j}|-|Z_{b,\mu,j}|+|Z_{a,\mu,j}^{a+b}(1,1)|+|Z_{b,\mu,j}^{a+b}(1,1)|)
  \\
  &\quad+ (|Z_{a+b,\mu,j+1}(1,1,0,0)|+|Z_{a+b,\mu,j-1}(1,0,0,1)|+|Z_{a+b,\mu,j-1}(0,1,1,0)|+|Z_{a+b,\mu,j+1}(0,0,1,1)|
  \\
  &\quad- |Z_{a,\mu,j}^{a+b}(1,1)|-|Z_{b,\mu,j}^{a+b}(1,1)|)
  \\
  &= -|Z_{a,\mu,j}|-|Z_{b,\mu,j}|+|Z_{a+b,\mu,j}|  \\
  &\quad -|Z_{a+b,\mu,j}(1,1,0,0)|-|Z_{a+b,\mu,j}(1,0,0,1)|-|Z_{a+b,\mu,j}(0,1,1,0)|-|Z_{a+b,\mu,j}(0,0,1,1)|  \\
  &\quad +|Z_{a+b,\mu,j+1}(1,1,0,0)|+|Z_{a+b,\mu,j-1}(1,0,0,1)|+|Z_{a+b,\mu,j-1}(0,1,1,0)|+|Z_{a+b,\mu,j+1}(0,0,1,1)|,
\end{align*}
where we use the bijections as described previously. 

Therefore, the left-hand side of \eqref{one-edge-lemma} becomes 
\begin{align}
  \label{lhs}
  &-|Z_{a,\mu,j}|-|Z_{b,\mu,j}| \nonumber \\
  & -|Z_{a+b,\mu,j}(1,1,0,0)|-|Z_{a+b,\mu,j}(1,0,0,1)|-|Z_{a+b,\mu,j}(0,1,1,0)|-|Z_{a+b,\mu,j}(0,0,1,1)| \nonumber \\
  & +|Z_{a+b,\mu,j+1}(1,1,0,0)|+|Z_{a+b,\mu,j-1}(1,0,0,1)|+|Z_{a+b,\mu,j-1}(0,1,1,0)|+|Z_{a+b,\mu,j+1}(0,0,1,1)|
\end{align}

We now define a number of sets that will help to break the proof down into smaller components.

Consider 
\begin{align*}
  A=\{&(a+b,W)\in Z_{a+b,\mu,j-1}(0,1,1,0)\cup Z_{a+b,\mu,j+1}(0,0,1,1): \\
  &\text{for all }i\in\{1,\dots,a\},\text{exactly one of }v_i\text{ and }v_{i+b}\text{ is in }W \},
\end{align*}
\[B=\bp{Z_{a+b,\mu,j-1}(0,1,1,0)\cup Z_{a+b,\mu,j+1}(0,0,1,1)}\setminus A,\]
\begin{align*}
  C=\{&(a+b,W)\in Z_{a+b,\mu,j}(0,1,1,0)\cup Z_{a+b,\mu,j}(0,0,1,1): \\
  &\text{for all }i\in\{1,\dots,a\},\text{exactly one of }v_i\text{ and }v_{i+b}\text{ is in }W \},
\end{align*}
\[D=\bp{Z_{a+b,\mu,j}(0,1,1,0)\cup Z_{a+b,\mu,j}(0,0,1,1)}\setminus C.\]
Similarly, let 
\begin{align*}
  A'=\{&(a+b,W)\in Z_{a+b,\mu,j-1}(1,1,0,0)\cup Z_{a+b,\mu,j+1}(1,0,0,1): \\
  &\text{for all }i\in\{1,\dots,a\},\text{exactly one of }v_i\text{ and }v_{i+b}\text{ is in }W \},
\end{align*}
\[B'=\bp{Z_{a+b,\mu,j-1}(1,1,0,0)\cup Z_{a+b,\mu,j+1}(1,0,0,1)}\setminus A',\]
\begin{align*}
  C'=\{&(a+b,W)\in Z_{a+b,\mu,j}(1,1,0,0)\cup Z_{a+b,\mu,j}(1,0,0,1): \\
  &\text{for all }i\in\{1,\dots,a\},\text{exactly one of }v_i\text{ and }v_{i+b}\text{ is in }W \},
\end{align*}
\[D'=\bp{Z_{a+b,\mu,j}(1,1,0,0)\cup Z_{a+b,\mu,j}(1,0,0,1)}\setminus C'.\]

Note that all of these sets depend on $a,b,\mu,j$, but we suppress explicit mention of this for clarity.

However, we claim that for any choice of $a,b,\mu,j$, we have $|D|=|B|$.
Consider the function $\varphi:D\to B$ defined as follows.
Fix $(a+b,W)\in D$. 
We let $i_0$ be the largest index from $\{1,\ldots,a-1\}$ such that either both $v_{i_0}$ and $v_{i_0+b}$ are in $W$ or both $v_{i_0}$ and $v_{i_0+b}$ are not in $W$, which exists by construction.
Then $\varphi(a+b,W)=(a+b,W')$, where:
\begin{itemize}
  \item For $i_0<i\leq a$, $v_i\in W'$ if and only if $v_{i+b}\in W$, or equivalently, $v_i\in W'$ if and only if $v_i\notin W$;
  \item For $i_0+b<i\leq a+b$, $v_i\in W'$ if and only if $v_{i-b}\in W$, or equivalently, $v_i\in W'$ if and only if $v_i\notin W$;
  \item For all other $i \in \{1,\dots,a+b\}$, $v_i\in W'$ if and only if $v_i\in W$.
\end{itemize}

Now if the function's range is indeed $B$, then it is easy to see that it is a bijection between $D$ and $B$ since it is clearly reversible.

By construction, clearly $|W'|=\mu$, and it is also easy to see that $\varphi(a+b,W)\notin A$ since $(a+b,W)$ was not in $C$.

It thus suffices to prove that $\varphi(a+b,W)\in Z_{a+b,\mu,j-1}(0,1,1,0)\cup Z_{a+b,\mu,j+1}(0,0,1,1)$, meaning that the number of components of $C_{a+b} \setminus W$ and the inclusions of $v_1,v_a,v_{a+1},v_b$ match one of the two possible cases.

Let $\varphi(a+b,W)=(a+b,W')$. In the graph $C_{a+b} \setminus W$, we define the following induced subgraphs:
\begin{itemize}
    \item $G_1$ is the subgraph induced by $\{v_1,\dots,v_{i_0}\} \setminus W$ and has $\ell_1$ components.
    \item $G_a$ is the subgraph induced by $\{v_{i_0+1}, \dots, v_a\} \setminus W$ and has $\ell_a$ components.
    \item $G_{a+1}$ is the subgraph induced by $\{v_{a+1}, \dots, v_{i_0+b}\} \setminus W$ and has $\ell_{a+1}$ components.
    \item $G_{a+b}$ is the subgraph induced by $\{v_{i_0+b+1}, \dots, v_{a+b}\} \setminus W$ and has $\ell_{a+b}$ components.
\end{itemize}

We define induced subgraphs $G_1', G_a', G_{a+1}', G_{a+b}'$ of $C_{a+b} \setminus W'$ analogously, with number of connected components $\ell_1', \ell_a', \ell_{a+1}', \ell_{a+b}'$ respectively. 

Note that according to the construction, each of these graphs is necessarily nonempty, and furthermore $G_1 = G_1', G_{a+1} = G_{a+1}'$, and $V(G_a) \cap V(G_a') = V(G_{a+b}) \cap V(G_{a+b}') = \varnothing$. It follows that $\ell_1 = \ell_1'$, $\ell_{a+1} = \ell_{a+1}'$, and it is simple to verify that
\begin{itemize}
    \item $\ell'_a = \ell_a$ if exactly one of $v_{i_0+1}$ and $v_a$ is in $W$,
    \item $\ell'_a = \ell_a-1$ if $v_{i_0+1}, v_a \notin W$,
    \item $\ell'_a = \ell_a+1$ if $v_{i_0+1}, v_a \in W$,
\end{itemize}
and analogously for $\ell'_{a+b}$.

We first suppose that $(a+b,W)\in Z_{a+b,\mu,j}(0,1,1,0)$.
By construction, $v_1,v_a\notin W'$ and $v_{a+1},v_{a+b}\in W'$.
Now:
\begin{itemize}
  \item If $v_{i_0},v_{i_0+1}\notin W$, then by construction, we have $v_{i_0+b}\notin W$ and $v_{i_0+b+1}\in W$, and it is straightforward to verify that $j=\ell_1+\ell_a+\ell_{a+1}+\ell_{a+b}-2$, since two components are joined between $G_{a+b}$ and $G_1$, and between $G_1$ and $G_a$.
  
  We then have $v_{i_0}\notin W'$, $v_{i_0+1}\in W'$, $v_{i_0+b}\notin W'$, and $v_{i_0+b+1}\notin W'$. Furthermore, using the above observations, it is easy to verify that $\ell_i' = \ell_i$ for $i \in \{1,a,a+1,a+b\}$, and that the only components of $C_{a+b} \setminus W'$ unified across different $G_i$ are between $G'_{a+1}$ and $G'_{a+b}$, so there are $\ell'_1+\ell'_a+\ell'_{a+1}+\ell'_{a+b}-1=j+1$ components.

  \item If $v_{i_0}\notin W$ and $v_{i_0+1}\in W$, then by construction, we have $v_{i_0+b}\notin W$ and $v_{i_0+b+1}\notin W$, and similar to the above argument we may verify that $j=\ell_1+\ell_a+\ell_{a+1}+\ell_{a+b}-2$. 
  
  We then have $v_{i_0}\notin W'$, $v_{i_0+1}\notin W'$, $v_{i_0+b}\notin W'$, $v_{i_0+b+1}\in W'$, and we have $\ell'_a = \ell_a+1$ and $\ell'_{a+b} = \ell_{a+b}-1$. Then we may check that in $C_{a+b} \setminus W'$ there are $\ell'_1+\ell'_a+\ell'_{a+1}+\ell'_{a+b}-1=j+1$ components.

  \item If $v_{i_0}\in W$ and $v_{i_0+1}\notin W$, then $i_0+b,v_{i_0+b+1}\in W$, and $j=\ell_1+\ell_a+\ell_{a+1}+\ell_{a+b}-1$.
  
  We then have $v_{i_0}\in W'$, $v_{i_0+1}\in W'$, $v_{i_0+b}\in W'$, and $v_{i_0+b+1}\notin W'$, and we may check that $\ell'_a = \ell_a$ and $\ell'_{a+b} = \ell_{a+b}$, so in $C_{a+b} \setminus W'$ there are $\ell'_1+\ell'_a+\ell'_{a+1}+\ell'_{a+b}=j+1$ components.
  
  \item If $v_{i_0},v_{i_0+1}\in W$, we then know that $v_{i_0+b}\in W$, $v_{i_0+b+1}\notin W$, and $j=\ell_1+\ell_a+\ell_{a+1}+\ell_{a+b}-1$.
  
  We then have $v_{i_0}\in W'$, $v_{i_0+1}\notin W'$, $v_{i_0+b}\in W'$, and $v_{i_0+b+1}\in W'$. We may check that $\ell'_a = \ell_a+1$ and $\ell'_{a+b} = \ell_{a+b}-1$, so in $C_{a+b} \setminus W'$ there are $\ell'_1+\ell'_a+\ell'_{a+1}+\ell'_{a+b}=j+1$ components.
\end{itemize}
Thus in every case, $(a+b,W')\in Z_{a+b,\mu,j+1}(0,0,1,1)$. An analogous argument shows that if $(a+b,W) \in Z_{a+b,\mu,j}(0,0,1,1)$, then $\varphi(a+b,W) \in Z_{a+b,\mu,j-1}(0,1,1,0)$, completing the proof that $|D| = |B|$; An essentially identical proof shows that also $|D'| = |B'|$ for any $a,b,\mu,j$.

We now prove \eqref{one-edge-lemma} by evaluating the left-hand side of the equation as given in \eqref{lhs}, splitting into cases for different choices of $a,b,\mu,j$. Within each case, any statements about the sets $A,B,C,D,A',B',C',D'$ hold for all choices of $a,b,\mu,j$ considered by that case.

Throughout the proof, given a graph $G$ and $S \subseteq V(G)$, we let $G[S]$ denote the subgraph of $G$ induced by $S$, and $G \bk S$ denote the subgraph of $G$ induced by $V(G) \setminus S$.

\subsection*{Case 1. $j$ is even.}

\subsubsection*{Case 1.1. $\mu=a=b$.}
Since we are assuming $j\geq 1$, we have $|Z_{a,\mu,j}|=|Z_{b,\mu,j}|=0$.
Furthermore, using the map $(2a,W) \rightarrow (2a,[2a]\setminus W)$, we may verify that $|Z_{2a,a,j}(1,1,0,0)|=|Z_{2a,a,j}(0,0,1,1)|$,
$|Z_{2a,a,j}(1,0,0,1)|=|Z_{2a,a,j}(0,1,1,0)|$,
$|Z_{2a,a,j+1}(1,1,0,0)|=|Z_{2a,a,j+1}(0,0,1,1)|$, and
$|Z_{2a,a,j-1}(1,0,0,1)|=|Z_{2a,a,j-1}(0,1,1,0)|$.
Thus, \eqref{lhs} becomes
\[2\bp{-|Z_{2a,a,j}(0,1,1,0)|-|Z_{2a,a,j}(0,0,1,1)|+|Z_{2a,a,j-1}(0,1,1,0)|+|Z_{2a,a,j+1}(0,0,1,1)|}.\]
Using the fact that $j$ is even, it suffices to prove that 
\begin{equation*}
  |Z_{2a,a,j}(0,0,1,1)|+|Z_{2a,a,j}(0,1,1,0)|+\binom{a}{j}=|Z_{2a,a,j-1}(0,1,1,0)|+|Z_{2a,a,j+1}(0,0,1,1)|.
\end{equation*}
In terms of the sets defined above, we equivalently must show that
\begin{equation*}
|C|+|D|+\binom{a}{j} = |A|+|B|.
\end{equation*}
Since $|D| = |B|$, it suffices to show that $|C| = 0$, and $|A| = \binom{a}{j}$.

We first show that $C=\varnothing$. Assume for a contradiction that there exists some $(2a,W)\in C$. Let $H_a = C_{a+b}[\{v_1,\dots,v_a\}]$, and let $H_b = C_{a+b}[\{v_{a+1}, \dots, v_{2a}\}]$. Let $W_a = W \cap V(H_a)$ and let $W_b = W \cap V(H_b)$.
Suppose $H_a \setminus W_a$ has $\ell$ components. 
\begin{itemize}
  \item If $(2a,W)\in Z_{2a,a,j}(0,1,1,0)$, then since $v_1 \notin W$ and $v_a \in W$, there are also $\ell$ components of the subgraph of $H_a[W_a]$.
  Since $(2a,W) \in C$, by symmetry there are $\ell$ components of $H_b \setminus W$.
  Now since $v_1,v_{2a}\notin W$ and $v_a,v_{a+1}\in W$, we have that $j=\ell+\ell-1 = 2\ell-1$, which contradicts that $j$ is even.
  \item If $(2a,W)\in Z_{2a,a,j}(0,0,1,1)$, then since $v_1,v_a \notin W$, there are $\ell-1$ components of $H_a[W_a]$. Since $(2a,W) \in C$, by symmetry there are $\ell-1$ components of $H_b \setminus W$.
  Now since $v_{a+1}\in W$ and $v_{2a}\in W$, we have $j=\ell+(\ell-1)=2\ell-1$, which contradicts that $j$ is even.
\end{itemize}
This proves that $C = \varnothing$.

We then show that $|A|=\binom{a}{j}$.
Consider the function $\psi:A\to\{S\subseteq \{1,\dots,a\}:|S|=j\}$ where $\psi(2a,W)=\{i_1,\ldots,i_j\}$ is defined as follows:
\begin{itemize}
  \item $i_1$ is the largest index such that $v_1,\ldots,v_{i_1}\notin W$ (note that necessarily $i_1 \leq a$).
  \item For $k \in \{2,\dots,j\}$, if $k$ is even, $i_k$ is the largest index such that $v_{i_{k-1}+1},\ldots,v_{i_k}\in W$; if $k$ is odd, $i_k$ is the largest index such that $v_{i_{k-1}+1},\ldots,v_{i_k}\notin W$.
\end{itemize}
As described, $\phi(2a,W)$ simply produces a set of $j$ positive integers; we claim that indeed $i_j\leq a$.
We first suppose $(2a,W)\in Z_{2a,a,j-1}(0,1,1,0)$.
Using the notation above, suppose there are $\ell$ components of $H_a \setminus W_a$. Then since $v_1 \notin W$ and $v_a \in W$, there are also $\ell$ components of $H_a[W_a]$, and so by the symmetry of $A$, there are also $\ell$ components of $H_b \setminus W_b$. Since $v_a, v_{a+1} \in W$ but $v_1, v_{2a} \notin W$, there are $2\ell-1$ components of $C_{a+b} \setminus W$.
This means that $2\ell-1=j-1$, or $\ell=j/2$. Thus there are $j/2$ components in each of $H_a[W_a]$ and $H_a \setminus W_a$, so we may verify that $i_j=a$.

Now suppose $(2a,W)\in Z_{2a,a,j+1}(0,0,1,1)$, and suppose there are $\ell$ components of $H_a \setminus W_a$.
Then as $v_1,v_a \notin W$, there are $\ell-1$ components of $H_a[W_a]$, so by symmetry also $\ell-1$ components of $H_b \setminus W_b$.
Since $v_{a+1},v_{2a} \in W$, there are $2\ell-1$ components of $C_{a+b} \setminus W$, so $2\ell-1=j+1$, or $\ell=j/2+1$, and it follows that $i_j<a$. 
Hence, the range of $\psi$ is correctly given.

Note that when $(2a,W)\in Z_{2a,a,j+1}(0,0,1,1)$, we further know that $v_{i_j+1},\ldots,v_a\notin W$.

We claim that $\psi$ is a bijection. To show injectivity, suppose $(2a,W),(2a,W')\in A$ are such that $\psi(2a,W)=\psi(2a,W')$. 
Then by construction and our observations above, we have that $W\cap\{1,\ldots,a\}=W'\cap\{1,\ldots,a\}$. By the symmetry of $A$, this implies $W=W'$.

To show surjectivity, suppose we are given $S=\{i_1,\ldots,i_j\}$ such that $1\leq i_1<\cdots<i_j\leq a$.
Let 
\begin{align*}
  W = &\,\, \{v_{i_1+1},\ldots,v_{i_2}\}\cup\cdots\cup\{v_{i_{j-1}+1},\ldots,v_{i_j}\} \\
  &\cup\{v_{a+1},\ldots,v_{a+i_1}\}\cup\cdots\cup\{v_{a+i_{j-2}+1},\ldots,v_{a+i_{j-1}}\}\cup\{v_{a+i_j+1},\ldots,v_{2a}\}.
\end{align*}
Then it is easy to check that $(2a,W)\in A$ and $\psi(2a,W)=S$.
This proves that $\psi$ is a bijection and $|A|=\binom{a}{j}$, and this finishes the proof of this case.

\subsubsection*{Case 1.2. $\mu=a<b$.}

Since we assume $j\geq 1$, we have $|Z_{a,a,j}|=0$. It then suffices to prove that 
\begin{align*}
  &\quad  |Z_{a+b,a,j+1}(1,1,0,0)|+|Z_{a+b,a,j-1}(1,0,0,1)|+|Z_{a+b,a,j-1}(0,1,1,0)|+|Z_{a+b,a,j+1}(0,0,1,1)|  \\
  &= |Z_{a+b,a,j}(1,1,0,0)|+|Z_{a+b,a,j}(1,0,0,1)|+|Z_{a+b,a,j}(0,1,1,0)|+|Z_{a+b,a,j}(0,0,1,1)|+\binom{a}{j}
\end{align*}
In terms of the previously described sets, we equivalently need to show that
\begin{equation*}
    |A|+|B|+|A'|+|B'| = |C|+|D|+|C'|+|D'|+\binom{a}{j}.
\end{equation*}
Since $|D| = |B|$ and $|D'| = |B'|$, it is enough to show that $|A| = |C| = |C'| = 0$, and $|A'| = \binom{a}{j}$.

We first claim that $A=\varnothing$. Assume for a contradiction that $A\neq\varnothing$ and let $(a+b,W)\in A$.
Since for all $i\in\{1,\dots,a\}$, exactly one of $v_i,v_{i+b}$ is in $W$ and $|W| = a$, it must be the case that $v_{a+1},\ldots,v_b\notin W$, contradicting that $v_{a+1}\in W$ for all $(a+b,W) \in A$.

We then claim that $C=\varnothing$ and $C'=\varnothing$.
Assume for a contradiction that there exists some $(a+b,W)\in C\cup C'$.
Then as before, we have $v_{a+1},\ldots,v_b\notin W$. We retain the previous notation for $H_a$ and $W_a$, but we now define $H_{b+1} = C_{a+b}[\{v_{b+1},\dots,v_{b+a}\}]$ and $W_{b+1} = W \cap V(H_{b+1})$.
Suppose that there are $\ell$ components of $H_a \setminus W_a$. We proceed in a similar manner as in the previous subcase.
\begin{itemize}
  \item If $(a+b,W)\in Z_{a+b,a,j}(1,1,0,0)$, then we may verify that there are $\ell+1$ components of $H_{b+1} \setminus W_{b+1}$. Furthermore, since $v_1 \in W$, $v_{b+1}\notin W$ by construction, so one component of $H_{b+1} \setminus W_{b+1}$ is joined with $\{v_{a+1}, \dots, v_b\}$ as a component in $C_{a+b} \setminus W$. Thus the number of components of $C_{a+b} \setminus W$ is $j=\ell+\ell+1=2\ell+1$, which is odd, a contradiction.
  \item If $(a+b,W)\in Z_{a+b,a,j}(1,0,0,1)$, then we may verify that there are $\ell$ components of $H_{b+1} \setminus W_{b+1}$. As above, $v_{b+1}\notin W$. Since also $v_a \notin A$, so one component from each of $H_a \setminus W_a$ and $H_{b+1} \setminus W_{b+1}$ are joined via $\{v_{a+1},\dots,v_b\}$, so it follows that $j=\ell+\ell-1=2\ell-1$, which is again a contradiction to $j$ being even.
  \item If $(a+b,W)\in Z_{a+b,a,j}(0,1,1,0)\cup Z_{a+b,a,j}(0,0,1,1)$, then $v_{a+1}\in W$, which is a contradiction.
\end{itemize}
This that $C = C' = \varnothing$.

Finally, we claim that $|A'|=\binom{a}{j}$.
Consider the function $\psi:A'\to\{S\subseteq\{1,\dots,a\}:|S|=j\}$ where $\psi(a+b,W)=\{i_1,\ldots,i_j\}$ is a finite set such that:
\begin{itemize}
  \item $i_1$ is the largest index such that $1,\ldots,i_1\in W$.
  \item For $k \in \{2,\dots,j\}$, if $k$ is even, $i_k$ is the largest index such that $v_{i_{k-1}+1},\ldots,v_{i_k}\notin W$; if $k$ is odd, $i_k$ is the largest index such that $v_{i_{k-1}+1},\ldots,v_{i_k}\in W$.
\end{itemize}
Using an argument exactly analogous to that of Case 1.1 and the fact that $v_{a+1},\ldots,v_b\notin W$
shows that this is a bijection, and finishes the proof of this subcase.

\subsubsection*{Case 1.3. $a<b=\mu$.}

Since by assumption $j\geq 1$, we have $|Z_{a,b,j}|=|Z_{b,b,j}|=0$.
It then suffices to prove that 
\begin{align*}
  &\quad  |Z_{a+b,b,j+1}(1,1,0,0)|+|Z_{a+b,b,j-1}(1,0,0,1)|+|Z_{a+b,b,j-1}(0,1,1,0)|+|Z_{a+b,b,j+1}(0,0,1,1)|  \\
  &= |Z_{a+b,b,j}(1,1,0,0)|+|Z_{a+b,b,j}(1,0,0,1)|+|Z_{a+b,b,j}(0,1,1,0)|+|Z_{a+b,b,j}(0,0,1,1)|+\binom{a}{j}.
\end{align*}
As before, this is equivalent to the claim that
\begin{equation*}
    |A|+|B|+|A'|+|B'| = |C|+|D|+|C'|+|D'|+\binom{a}{j}.
\end{equation*}
Since $|D| = |B|$ and $|D'| = |B'|$, it is enough to show that $|A'| = |C| = |C'| = 0$, and $|A| = \binom{a}{j}$.

We first claim that $A'=\varnothing$.
Assume for a contradiction that $A'\neq \varnothing$ and let $(a+b,W)\in A'$.
Since for all $i\in\{1,\dots,a\}$, exactly one of $v_i,v_{i+b}$ is in $W$, and $|W| = b$, it must be the case that $v_{a+1},\ldots,v_b\in W$, contradicting that $v_{a+1}\notin W$ for all $(a+b,W) \in A'$.

We then claim that $C=\varnothing$ and $C'=\varnothing$. The proof is exactly analogous to the corresponding proof in Case 1.2. We assume there are $\ell$ components of $H_a \setminus W_a$.
\begin{itemize}
  \item If $(a+b,W)\in Z_{a+b,b,j}(1,1,0,0)\cup Z_{a+b,b,j}(1,0,0,1)$, then $v_{a+1}\notin W$ is a contradiction.
  \item If $(a+b,W)\in Z_{a+b,b,j}(0,1,1,0)$, then we may verify that $j=2\ell-1$, which is a contradiction to the fact that $j$ is even.
  \item If $(a+b,W)\in Z_{a+b,b,j}(0,0,1,1)$, then we may again verify that $j=2\ell-1$, which is a contradiction.
\end{itemize}

Finally, we claim that $|A|=\binom{a}{j}$, and we will use a bijection exactly analogous to that of the first two cases.

First note that for all $(a+b,W)\in A$, we must have $v_{a+1},\ldots,v_b\in W$.
Consider the function $\psi:A\to\{S\subseteq\{1,\dots,a\}:|S|=j\}$ where $\psi(2a,W)=\{i_1,\ldots,i_j\}$ is defined as follows:
\begin{itemize}
  \item $i_1$ is the largest index such that $v_1,\ldots,v_{i_1}\notin W$.
  \item For $k \in \{2,\dots,j\}$, if $k$ is even, $i_k$ is the largest index such that $v_{i_{k-1}+1},\ldots,v_{i_k}\in W$; if $k$ is odd, $i_k$ is the largest index such that $v_{i_{k-1}+1},\ldots,v_{i_k}\notin W$.
\end{itemize}
Then using an argument analogous to that presented in Case 1.1 and the fact that $v_{a+1},\ldots,v_b\in W$ for all $(a+b,W)\in A$, we can show that $\psi$ is indeed a bijection, and this finishes the proof of this subcase.

\subsubsection*{Case 1.4. $\mu<a<b$ or $a<b<\mu$.}

In this case, we are required to prove that 
\begin{align*}
  &\quad  |Z_{a+b,\mu,j+1}(1,1,0,0)|+|Z_{a+b,\mu,j-1}(1,0,0,1)|+|Z_{a+b,\mu,j-1}(0,1,1,0)|+|Z_{a+b,\mu,j+1}(0,0,1,1)|  \\
  &= |Z_{a+b,\mu,j}(1,1,0,0)|+|Z_{a+b,\mu,j}(1,0,0,1)|+|Z_{a+b,\mu,j}(0,1,1,0)|+|Z_{a+b,\mu,j}(0,0,1,1)|,
\end{align*}
or equivalently that
\begin{equation*}
    |A|+|B|+|A'|+|B'| = |C|+|D|+|C'|+|D'|.
\end{equation*}
But this follows from $|B| = |D|$, $|B'| = |D'|$, and $A=C=A'=C'=\varnothing$, since if $(a+b,W)\in A\cup C\cup A'\cup C'$, then we must have $a\leq|W|\leq b$, which is a contradiction.

\subsection*{Case 2. $j$ is odd.}

\subsubsection*{Case 2.1. $\mu=a=b$.}

As in Case 1.1, except using that $j$ is odd, it suffices to show that 
\[|Z_{2a,a,j}(0,0,1,1)|+|Z_{2a,a,j}(0,1,1,0)|=|Z_{2a,a,j-1}(0,1,1,0)|+|Z_{2a,a,j+1}(0,0,1,1)|+\binom{a}{j}.\]
Equivalently, we would like to show that
\begin{equation*}
    |C|+|D| = |A|+|B|+\binom{a}{j}.
\end{equation*}
Since $|D| = |B|$, it suffices to show that $|A| = 0$ and $|C| = \binom{a}{j}$.

We first claim that $A=\varnothing$.
Assume for a contradiction that there exists some $(2a,W)\in A$. Using the notation as in Case 1.1, suppose that $H_a \setminus W_a$ has $\ell$ components.
\begin{itemize}
  \item If $(2a,W)\in Z_{2a,a,j-1}(0,1,1,0)$, then as in previous cases we may verify that there are also $\ell$ components of $H_b \setminus W_b$. Since $v_a,v_{a+1} \in W$ and $v_1,v_{2a} \notin W$, it follows that $j-1=\ell+\ell-1 = 2\ell-1$, which contradicts that $j$ is odd.
  \item If $(2a,W)\in Z_{2a,a,j+1}(0,0,1,1)$, then there are $\ell-1$ components of $H_b \setminus W_b$. Since $v_{a+1} \in W$ and $v_{2a} \in W$ it follows that $j+1=\ell+(\ell-1) = 2\ell-1$, which also contradicts that $j$ is odd.
\end{itemize}
This proves that $A=\varnothing$.

We then claim that $|C|=\binom{a}{j}$, and we use the same style of bijection as in Case 1.1. Consider the function $\psi:C\to\{S\subseteq\{1,\dots,a\}:|S|=j\}$ where $\psi(2a,W)=\{i_1,\ldots,i_j\}$ is defined as follows:
\begin{itemize}
  \item $i_1$ is the largest index such that $v_1,\ldots,v_{i_1}\notin W$.
  \item For $k \in \{2,\dots,j\}$, if $k$ is even, $i_k$ is the largest index such that $v_{i_{k-1}+1},\ldots,v_{i_k}\in W$; if $k$ is odd, $i_k$ is the largest index such that $v_{i_{k-1}+1},\ldots,v_{i_k}\notin W$.
\end{itemize}
Using the same argument as in Case 1.1, suppose that $H_a \setminus W_a$ has $\ell$ components.
If $(2a,W)\in Z_{2a,a,j}(0,1,1,0)$, then there are also $\ell$ components in $H_b \setminus W_b$, so $j=\ell+\ell-1$, and $\ell=\frac{j+1}{2}$, from which
it follows that $i_j<a$ since $v_a \in W$. Furthermore, in this case $v_{i_j+1},\ldots,v_a\in W$.

If $(2a,W)\in Z_{2a,a,j}(0,0,1,1)$, then there are $\ell-1$ components of $H_b \setminus W_b$, and it follows that $j=2\ell-1$, or $\ell=\frac{j+1}{2}$, and it is easy then to check that $i_j=a$ since $v_a \notin W$ and $v_{a+1} \in W$. Thus, we have verified that the range of $\psi$ is correct.


As in Case 1.1, it is straightforward to verify the injectivity and surjectivity of $\psi$. Therefore, $\psi$ is a bijection, and $|C|=\binom{a}{j}$, which completes the proof of this subcase.

\subsubsection*{Case 2.2. $\mu=a<b$.}

Similar to Case 1.2, it suffices to prove that 
\begin{align*}
  &\quad  |Z_{a+b,a,j+1}(1,1,0,0)|+|Z_{a+b,a,j-1}(1,0,0,1)|+|Z_{a+b,a,j-1}(0,1,1,0)|+|Z_{a+b,a,j+1}(0,0,1,1)|+\binom{a}{j}  \\
  &= |Z_{a+b,a,j}(1,1,0,0)|+|Z_{a+b,a,j}(1,0,0,1)|+|Z_{a+b,a,j}(0,1,1,0)|+|Z_{a+b,a,j}(0,0,1,1)|.
\end{align*}
which is equivalent to proving
\begin{equation*}
    |A|+|B|+|A'|+|B'|+\binom{a}{j} = |C|+|D|+|C'|+|D'|.
\end{equation*}
Since $|D| = |B|$ and $|D'| = |B'|$, it is sufficient to prove that $|A| = |A'| = |C| = 0$ and $|C| = \binom{a}{j}$.

As in Case 1.2, we show that $C=\varnothing$. If on the contrary there exists $(a+b,W)\in C$. Then $v_{a+1},\ldots,v_b\notin W$ since $|W|=a$ and for each $i \in \{1,\dots,a\}$, exactly one of $v_i$ and $v_{i+b}$ is in $W$. But this contradicts the fact that $v_{a+1}\in W$ for each $(a+b,W) \in C$.

We then claim that $A=\varnothing$ and $A'=\varnothing$.
Assume for a contradiction that there exists some $(a+b,W)\in A\cup A'$. Then $v_{a+1},\ldots,v_b\notin W$ since $|W|=a$.
Using notation as in Case 1.2, suppose there are $\ell$ components in $H_a \setminus W_a$.
\begin{itemize}
  \item If $(a+b,W)\in Z_{a+b,a,j-1}(0,1,1,0)\cup Z_{a+b,a,j+1}(0,0,1,1)$, then $v_{a+1}\in W$, a contradiction.
  \item If $(a+b,W)\in Z_{a+b,a,j-1}(1,1,0,0)$, then we may verify that $j-1=2\ell+1$, contradicting that $j$ is odd.
  \item If $(a+b,W)\in Z_{a+b,a,j+1}(1,0,0,1)$, then we may verify that $j+1=2\ell-1$, which again contradicts that $j$ is odd.
\end{itemize}
Thus, $A=\varnothing$ and $A'=\varnothing$.

Finally, we claim that $|C'|=\binom{a}{j}$.
Consider the function $\psi:C'\to\{S\subseteq\{1,\dots,a\}:|S|=j\}$ where $\psi(a+b,W)=\{i_1,\ldots,i_j\}$ defined as follows:
\begin{itemize}
  \item $i_1$ is the largest index such that $v_1,\ldots,v_{i_1}\in W$.
  \item For $k \in \{2,\dots,j\}$, if $k$ is even, $i_k$ is the largest index such that $v_{i_{k-1}+1},\ldots,v_{i_k}\notin W$; if $k$ is odd, $i_k$ is the largest index such that $v_{i_{k-1}+1},\ldots,v_{i_k}\in W$.
\end{itemize}
Using the same reasoning as in previous cases, it is straightforward to verify that $\psi$ is a bijection, completing the proof of this subcase.

\subsubsection*{Case 2.3. $a<b=\mu$.}

Similar to case 1.3, it suffices to prove that 
\begin{align*}
  &\quad  |Z_{a+b,b,j+1}(1,1,0,0)|+|Z_{a+b,b,j-1}(1,0,0,1)|+|Z_{a+b,b,j-1}(0,1,1,0)|+|Z_{a+b,b,j+1}(0,0,1,1)| +\binom{a}{j} \\
  &= |Z_{a+b,b,j}(1,1,0,0)|+|Z_{a+b,b,j}(1,0,0,1)|+|Z_{a+b,b,j}(0,1,1,0)|+|Z_{a+b,b,j}(0,0,1,1)|,
\end{align*}
or equivalently that
\begin{equation*}
    |A|+|B|+|A'|+|B'|+\binom{a}{j} = |C|+|D|+|C'|+|D'|.
\end{equation*}
Since $|D| = |B|$ and $|D'| = |B'|$, it is enough to show that $|A'| = |A| = |C'| = 0$, and $|C| = \binom{a}{j}$.

Note that $C'=\varnothing$, since if on the contrary $(a+b,W)\in C'$, then since exactly one of $v_i$ and $v_{i+b}$ is in $W$ for each $i \in \{1,\dots,a\}$, we must have $v_{a+1},\ldots,v_b\in W$ to have $|W|=b$, contradicting that $v_{a+1}\notin W$ for each $(a+b,W) \in C'$.

We then claim that $A=\varnothing$ and $A'=\varnothing$.
Assuming for a contradiction that there exists some $(a+b,W)\in A\cup A'$, then as above $v_{a+1},\ldots,v_b\in W$.

Using notation as in previous cases, suppose that $H_a \setminus W_a$ has $\ell$ components.
\begin{itemize}
  \item If $(a+b,W)\in Z_{a+b,a,j-1}(0,1,1,0)$, then there are $\ell$ components of $H_{b+1} \setminus W_{b+1}$, and it follows that $j-1=2\ell-1$, which contradicts that $j$ is odd.
  \item If $(a+b,W)\in Z_{a+b,a,j+1}(0,0,1,1)$, then there are $\ell-1$ components of $H_{b+1} \setminus W_{b+1}$. This means $j+1=2\ell-1$, which again contradicts the fact that $j$ is odd.
  \item If $(a+b,W)\in Z_{a+b,a,j-1}(1,1,0,0)\cup Z_{a+b,a,j+1}(1,0,0,1)$, then $v_{a+1}\notin W$, a contradiction.
\end{itemize}
This proves that $A=\varnothing$ and $A'=\varnothing$.

Finally, show that $|C|=\binom{a}{j}$ by considering the function $\psi:C\to\{S\subseteq\{1,\dots,a\}:|S|=j\}$ where $\psi(a+b,W)=\{i_1,\ldots,i_j\}$ defined as follows:
\begin{itemize}
  \item $i_1$ is the largest index such that $v_1,\ldots,v_{i_1}\notin W$.
  \item For $k \in \{2,\dots,j\}$, if $k$ is even, $i_k$ is the largest index such that $v_{i_{k-1}+1},\ldots,v_{i_k}\in W$; if $k$ is odd, $i_k$ is the largest index such that $v_{i_{k-1}+1},\ldots,v_{i_k}\notin W$.
\end{itemize}
Using the same argument as in previous cases, we can show that $\psi$ is a bijection, and this completes the proof of this subcase.

\subsubsection*{Case 2.4. $\mu<a<b$ or $a<b<\mu$.}

The proof is exactly the same as Case 1.4.

This concludes the examination of all possible subcases, so we are done.
\end{proof}

\end{document}